%% file: main.tex
\documentclass{article}
\usepackage{graphicx} 
\usepackage{geometry}
\usepackage{amsmath}
\usepackage{amsfonts}
\usepackage{amsthm}
\usepackage{hyperref}
\usepackage{xcolor}
\usepackage{mathtools}
\usepackage{ulem}
\usepackage{amssymb}
\usepackage[dvipsnames]{xcolor}
\usepackage{amscd}
\usepackage{cleveref}
\usepackage{tikz-cd}
\usepackage{tikz}
\usetikzlibrary{arrows.meta,
            decorations.markings,cd}
\usepackage{pgfplots}
\usepackage{enumitem}
\usepackage{subcaption}
\usepackage{quiver}
\usepackage[abs]{overpic}
\usepackage{float}

\newcommand{\R}{\mathbb{R}}

\newcommand{\rank}{\operatorname{rank}}
\newcommand{\bx}{\mathbf{x}}
\newcommand{\bh}{\mathbf{h}}
\newcommand{\ba}{\mathbf{a}}
\newcommand{\bb}{\mathbf{b}}
\newcommand{\by}{\mathbf{y}}
\newcommand{\bc}{\mathbf{c}}

\newcommand{\bzero}{\mathbf{0}}

\newcommand{\Vect}{\operatorname{Vect}}

\newcommand{\VR}{\mathrm{VR}}

\newcommand{\proj}{{\textrm{Proj}}}
\newcommand{\dis}{{\textrm{dis}}}
\newcommand{\codis}{{\textrm{codis}}}

\newcommand{\C}{\mathcal{C}} 

\newtheorem{theorem}{Theorem}[subsection]
\newtheorem{corollary}[theorem]{Corollary}
\newtheorem{lemma}[theorem]{Lemma}
\newtheorem{prop}[theorem]{Proposition}

\theoremstyle{definition}
\newtheorem{example}[theorem]{Example}
\newtheorem{definition}[theorem]{Definition}
\newtheorem{remark}[theorem]{Remark}

\newtheorem{proposition}[theorem]{Proposition}

\newcommand{\define}[1]{\textbf{#1}}

\title{Persistent Homology for Labeled Datasets: Gromov-Hausdorff Stability and Generalized Landscapes}

\author{Yaoying Fu, Evgeniya Lagoda, 
Shiying Li, \\
Tom Needham, Lander Ver Hoef, and 
Morgan Weiler}

\begin{document}

\maketitle

\begin{abstract}
    Techniques from metric geometry have become fundamental tools in modern mathematical data science, providing principled methods for comparing datasets modeled as finite metric spaces. Two of the central tools in this area are the Gromov-Hausdorff distance and persistent homology, both of which yield isometry-invariant notions of distance between datasets. However, these frameworks do not account for categorical labels, which are intrinsic to many real-world datasets, such as labeled images, pre-clustered data, and semantically segmented shapes. In this paper, we introduce a general framework for labeled metric spaces and develop new notions of Gromov-Hausdorff distance and persistent homology which are adapted to this setting. Our main result shows that our persistent homology construction is stable with respect to our novel notion of Gromov-Hausdorff distance, extending a classic result in topological data analysis. To facilitate computation, we also introduce a labeled version of persistence landscapes and show that the landscape map is Lipschitz.
\end{abstract}

\section{Introduction} 

Techniques from metric geometry have come to play an increasingly important role in modern mathematical data science. From this perspective, a dataset is typically modeled as a finite metric space, and the focus is on developing methods for comparing these metric spaces. Two of the most prominent tools in this area are:

\medskip
\noindent The \textit{Gromov-Hausdorff distance}, which provides a notion of distance between finite metric spaces, defined by finding a correspondence between their underlying sets which optimally preserves metric structure (see Section~\ref{subsec:GH_distance} for details), leading to an isometry-invariant method for shape registration~\cite{Memoli2007}.

\medskip
\noindent \textit{Persistent homology}, which featurizes a finite metric space by tracking homological changes in an associated sequence of growing simplicial complexes. The resulting isometry-invariant signatures, referred to as \textit{persistence modules}, can be compared via a canonical metric called the \textit{interleaving distance}~\cite{chazal2009proximity}, and this pipeline is known to be stable, in the sense that the map taking a finite metric space to its persistence module is 2-Lipschitz with respect to the Gromov-Hausdorff and interleaving distances~\cite{chazal2009gromov}.

\smallskip
These methods (and variants thereof---e.g., the related \textit{Gromov-Wasserstein distance}~\cite{memoli2011gromov}) have found applications in network analysis~\cite{lee2011computing,sizemore2019importance}, community detection~\cite{chowdhury2021generalized}, multi-omic data alignment~\cite{demetci2022scot}, and image segmentation~\cite{hu2019topology}, among many others.

This paper was inspired by the observation that the techniques described above largely ignore an important aspect of many datasets: in practice, data is frequently endowed with \textit{categorical labels}. For example, a prototypical supervised learning task is to train a model to recognize the subject of an image in a dataset of labeled images; the dataset of interest is therefore not just a finite metric space, but a finite metric space endowed with additional label data. Further examples of label structures include datasets which come with pre-defined clusters, or shapes endowed with semantic segmentations---see Example \ref{ex:labeled_metric_spaces} and Figure \ref{fig:Data_Examples}. Neither Gromov-Hausdorff distance nor persistent homology include mechanisms for incorporating this important label structure into dataset comparisons. There has been a substantial amount of recent work addressing this deficiency for special cases of label structures~\cite{di2022chromatic,needham2023geometric,di2024chromatic,zhang2024geometry,stolz2024relational,draganov2025gromov}, but a truly general solution is lacking.

In this paper, we introduce a very general model for labeled datasets, which captures the various structures defined above, as well as new notions of Gromov-Hausdorff distance, persistent homology, and interleaving distance which are adapted to compare these objects. Our main contributions are as follows:

\begin{itemize}
    \item {\bf Labeled Gromov-Hausdorff Distances.} We define a \textit{labeled metric space} to be a metric space $(X,d_X)$ endowed with a finite sequence of subsets $L_X = (X_1,\ldots,X_k)$ that cover $X$ (Definition \ref{def:labeled_metric_space}). This model is flexible enough to capture the various commonly-occuring data structures described in Example \ref{ex:labeled_metric_spaces}. We then define three increasingly general notions of Gromov-Hausdorff distance for comparing labeled metric spaces; these are designed to handle specific scenarios that one may encounter when performing such comparisons, defined by prior knowledge of correspondences between the labels. Metric properties of these Gromov-Hausdorff distances are established in Theorem \ref{thm:k_GH}, Corollary \ref{cor:GH_mod_perm}, and Theorem \ref{thm:GH_mod_stab}, respectively. 

    \item {\bf Labeled Persistent Homology and Interleaving Distance.} We introduce a notion of persistent homology of labeled metric spaces which tracks the topological interactions of different labeled components of the datasets---see Definition \ref{def:labeled_vietoris_rips}. The resulting structure is a particular example of a \textit{multiparameter persistence module}~\cite{botnan2022introduction}, which we endow with a notion of interleaving distance that is different than the one usually considered in the multiparameter persistent homology literature (Definition \ref{def:epsilon-interleaving}). Our main results extend the well-known Gromov-Hausdorff stability result of~\cite{chazal2009gromov} to the labeled setting. These results comprise Theorem \ref{thm:stability registered} and Corollaries \ref{cor:stability_unregistered_classes} and \ref{cor:stability_modulo_stabilizations}. 

    \item {\bf Labeled Persistence Landscapes.} The labeled persistence modules we introduce are useful at a theoretical or conceptual level, but are inconvenient to work with computationally. To overcome this difficulty, we provide a method for vectorizing these modules by generalizing the notion of persistence landscapes~\cite{Bubenik15,Vipond20} to the labeled setting. This provides a way to map a labeled persistence module into a Banach space, and we show in Theorem \ref{thm:stability_of_landscapes} that this mapping is Lipschitz stable, with respect to the Banach space and interleaving distances.
\end{itemize}

The paper is organized as follows. In Section \ref{sec:GH distance for labeled}, we define labeled metric spaces, the associated Gromov-Hausdorff distances, and establish properties of these distances. Section \ref{sec:persistent_homology} is devoted to developing the persistent homology theory for labeled datasets. Persistence landscapes for labeled data are studied theoretically in Section \ref{sec:landscapes}. Finally, computational aspects of landscapes are treated in Section \ref{s:experiments}, where proof-of-concept numerical examples are provided.

\section{Gromov-Hausdorff Distance for Labeled Datasets}\label{sec:GH distance for labeled}

\begin{figure}[h]
\begin{overpic}[width=\linewidth,unit=1mm]{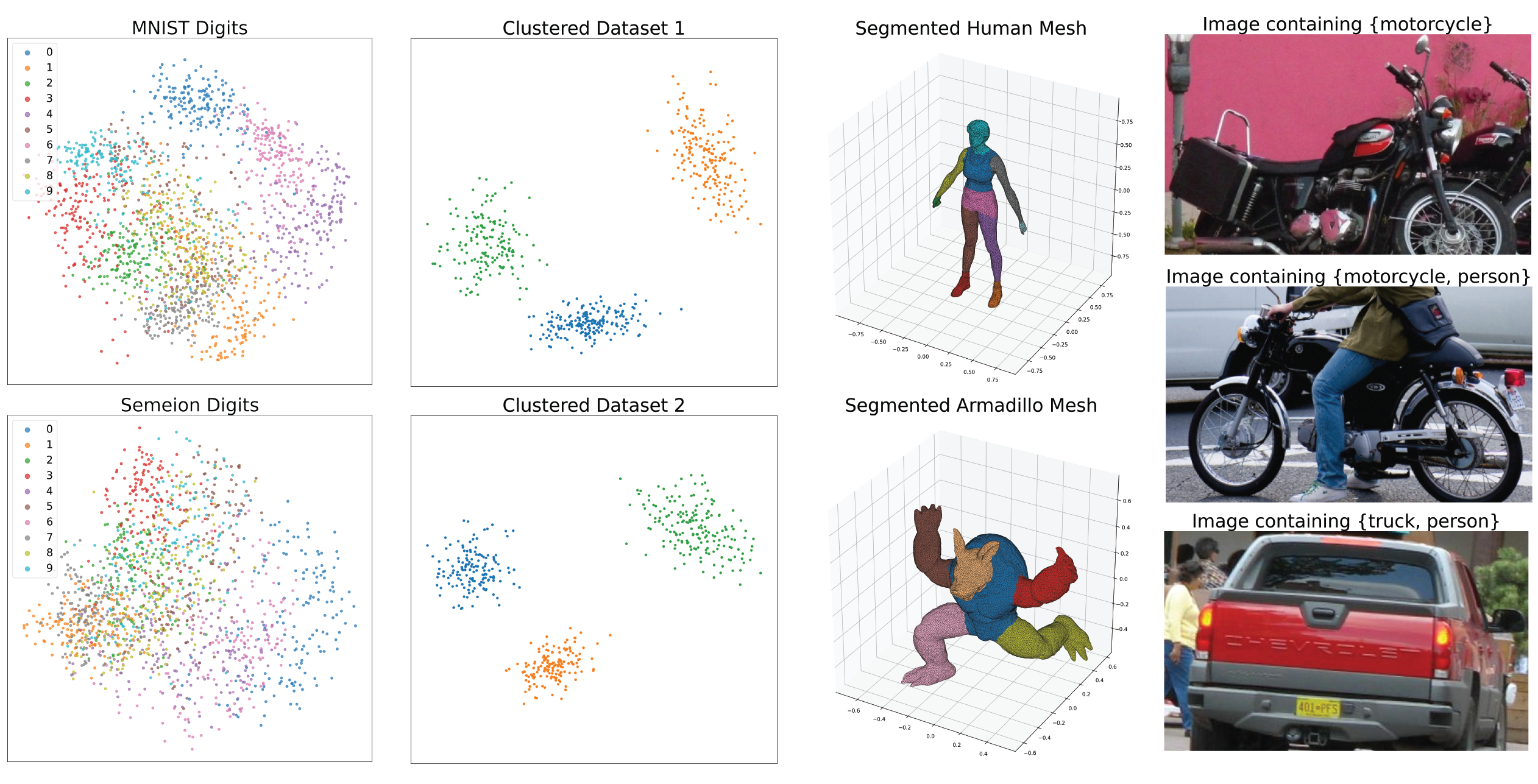}
\put(16,-4){(a)}
\put(56,-4){(b)}
\put(93,-4){(c)}
\put(130,-4){(d)}
\end{overpic}
\vspace{1mm}
\caption{Examples of labeled metric spaces arising in data science applications, representing several modalities for our framework:  {\bf (a)} registered labels,  {\bf (b)} permuted labels, {\bf (c)}  unregistered labels, and {\bf (d)} multiple labels. See Example \ref{ex:labeled_metric_spaces} for details.}\label{fig:Data_Examples}
\end{figure}
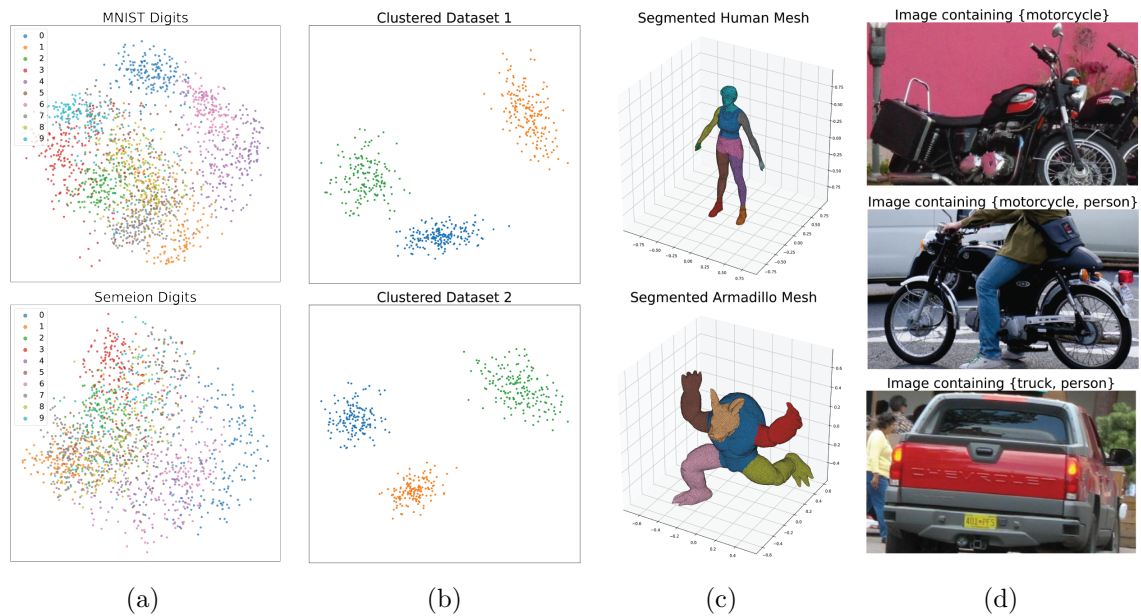

This section is devoted to extending the Gromov-Hausdorff distance \eqref{eq:GH_distance} to the setting where the metric spaces $X$ and $Y$ have additional structure. It is frequently the case in data science applications that a dataset $X$ is \define{labeled}; that is, it is endowed with a map to some finite set of labels, $\ell:X \to \{\ell_1,\ldots,\ell_k\}$. These label structures motivate the additional structure we describe below. 

\subsection{Labeled Metric Spaces} 

 Motivated by the labeled datasets discussed above, the following definition captures a more general notion than this labeling map. 
 
\begin{definition}[Labeled Metric Space]\label{def:labeled_metric_space}
    Let $k$ be a positive integer. A \define{$k$-labeled metric space} is a compact metric space $(X,d_X)$ endowed with a finite sequence $L_X = (X_1,\ldots,X_k)$ of nonempty subsets $X_i \subset X$ such that the sets cover $X$, i.e., $X = \cup_{i=1}^k X_i$. We denote the triple $\mathcal{X} = (X,d_X,L_X)$. Each $X_i$ is called a \define{label set} and $L_X$ is referred to as the \define{labels} of $\mathcal{X}$.  We use the terminology \define{labeled metric space} for a $k$-labeled metric space when we wish to keep the number of labels $k$ ambiguous.
\end{definition}

\begin{example}\label{ex:labeled_metric_spaces}
In the following subsections, we will introduce methods for comparing labeled metric spaces $\mathcal{X} = (X,d_X,L_X)$ and $\mathcal{Y} = (Y,d_Y,L_Y)$. Our framework handles  several realistic situations of increasing generality, which we now illustrate with simple examples, as shown in Figure \ref{fig:Data_Examples}.

\paragraph{Registered Labels.} Suppose that $\mathcal{X}$ and $\mathcal{Y}$ are both $k$-labeled datasets, for some $k$. It may be the case that there is a known correspondence between label sets, i.e., $X_i$ should be compared to $Y_i$ for each $i \in \{1,\ldots,k\}$. In this case, we say that $\mathcal{X}$ and $\mathcal{Y}$ have \define{registered labels}. Our philosophy in this situation is that, when comparing the structure of such a pair $\mathcal{X}$ and $\mathcal{Y}$, the label data should be taken into account. The specifics of our approach are explained in \S\ref{subsec:GH_registered}.

Consider the example shown in Figure \ref{fig:Data_Examples}(a). Here, $\mathcal{X}$ is the MNIST dataset of handwritten digits: the set $X$ consists of a large collection of greyscale images of digits $0-9$, written by different subjects, the metric $d_X$ is Euclidean distance between the image matrices, and the $k=10$ labels correspond to digit classifications ($X_1$ contains all instances of the digit zero, $X_2$ contains all instances of the digit one, etc.). The labeled metric space $\mathcal{Y}$ has a very similar structure, consisting of an alternative dataset of handwritten digits $0-9$, called the \textit{Semeion dataset}~\cite{semeion_handwritten_digit_178}. Here $Y_1$ contains instances of the digit zero, $Y_2$ contains instances of the digit one, etc., so that $\mathcal{X}$ and $\mathcal{Y}$ have registered labels. The pointclouds in Figure \ref{fig:Data_Examples}(a) are MDS embeddings of the respective datasets. 

\paragraph{Permuted Labels.} Next, consider the situation where $\mathcal{X}$ and $\mathcal{Y}$ are both $k$-labeled metric spaces such that a natural matching between labels exists, but is not known. Here, we say that $\mathcal{X}$ and $\mathcal{Y}$ have \define{permuted labels}. Our idea in this case is that one should attempt to register the label classes when comparing structures of this form, as we describe in \S\ref{subsec:GH_permuted}.

Figure \ref{fig:Data_Examples}(b) shows a pair of synthetic point clouds in $\R^2$, $X$ and $Y$, each of which has been clustered via $k$-means with $k=3$. These clusterings define label structures on the datasets, yielding 3-labeled metric spaces $\mathcal{X}$ and $\mathcal{Y}$, respectively. There is a natural correspondence between clusters of $X$ and $Y$, but this is not necessarily reflected by the labels, as $k$-means clustering returns the clusters in random order. Hence $\mathcal{X}$ and $\mathcal{Y}$ have permuted labels. In situations such as this, one must determine the correct correspondence between labels to compare the datasets while preserving cluster structure. We remark that such a procedure is typically a subroutine in ``divide-and-conquer'' point cloud registration algorithms, such as those considered in~\cite{blumberg2020mrec,chowdhury2021quantized}.

\paragraph{Unregistered Labels.} Iterating on the above idea, we now consider the case where $\mathcal{X}$ and $\mathcal{Y}$ have \define{unregistered labels}, potentially of unequal size---i.e.,  where $\mathcal{X}$ is a $k$-labeled metric space and $\mathcal{Y}$ is an $\ell$-labeled metric space, with $k \neq \ell$ allowed, and where there is no known correspondence between label sets. The goal when comparing data of this form is once again to find a registration of the labels in some sense, but the correct notion is not as clear. Our particular approach is described in \S\ref{subsec:GH_unregistered}.

An example of the unregistered situation is shown in Figure \ref{fig:Data_Examples}(c). Here, $X$ is a surface (given in mesh form)  representing a human figure, which is endowed with a metric $d_X$ (e.g., extrinsic Euclidean distance or intrinsic geodesic distance), and labels are given by an automated segmentation of the figure, leading to a labeled metric space $\mathcal{X}$. Similarly, $\mathcal{Y}$ is a segmented surface representing an armadillo---both examples are taken from the mesh segmentation benchmark dataset first published in~\cite{Chen:2009:ABF}. Observe that there is an intuitive semantic correspondence between label sets of $\mathcal{X}$ and $\mathcal{Y}$, but that the human figure has a more detailed segmentation, leading to an unequal number of labels between the spaces. While one may wish to compare these two meshes in a part-aware manner, this imbalance in label sets renders the correct approach to this task unclear. 

\paragraph{Overlapping Labels.} Finally, we give an example which is intended to clarify a particular choice in our definition. Observe that the definition of a labeled metric space does not require the labels $L_X$ to form a partition of $X$; that is, it is allowed for $X_i \cap X_j$ to be nonempty for $i \neq j$. This flexibility will be utilized to mathematically handle the unregistered labels setting in \S\ref{subsec:GH_unregistered}. Moreover, this definition is useful for modeling real data, where labels may naturally overlap. 

Consider a labeled metric space $\mathcal{X}$ where the underlying set $X$ is a collection of natural images, and the label indices correspond to potential subjects in the images---for example,  subject labels `motorcycle', `person', and `truck' may be enumerated as 1, 2 and 3 (see Figure \ref{fig:Data_Examples}(d)). Then an image in $X$ may receive multiple labels if it contains multiple subjects of interest, leading to overlaps in the label sets. The images in Figure \ref{fig:Data_Examples}(d), coming from a Kaggle dataset~\cite{kaggle_dataset}, represent elements of $X_1$, $X_1 \cap X_2$ and $X_2 \cap X_3$, respectively (with the enumeration described above). 

\end{example}

Below, we will define three increasingly general notions of Gromov-Hausdorff distance between labeled metric spaces, corresponding to the first three classes of examples described above.

\subsection{Background on Gromov-Hausdorff Distance}\label{subsec:GH_distance}

Before introducing our new notions of Gromov-Hausdorff distances, we recall the definition and related concepts in the classical setting. 

Given two sets $X$ and $Y$, a subset $C \subset X \times Y$ is called a \define{correspondence} between $X$ and $Y$ if 
\begin{equation}\label{eq:correspondence}
    \proj_X C = X \quad  \textrm{and } \quad \proj_Y C = Y,
\end{equation}
where $\proj_X$ and $\proj_Y$ are the coordinate projection maps. Let $\C(X,Y)$ denote the set of all correspondences between $X$ and $Y$. 

Let $(X,d_X)$ and $(Y,d_Y)$ be compact metric spaces, which we denote as $X$ and $Y$, respectively, when the presence of fixed metrics is understood from context. Recall that the \define{Gromov-Hausdorff distance} between $X$ and $Y$ can be expressed as (see, for example, \cite[Chapter 7]{burago2001course}) 
\begin{equation}\label{eq:GH_distance}
\mathsf{GH}(X,Y) \coloneqq \frac{1}{2} \inf_{C \in \mathcal{C}(X,Y)} \dis(C),
\end{equation}
where the \define{distortion} of the correspondence $C$ is defined by 
\begin{equation}\label{eq:distortion}
\dis(C) \coloneqq \sup_{\substack{(x_1, y_1) \in C \\ (x_2, y_2) \in C}}|d_X(x_1,x_2)-d_Y(y_1,y_2)|.
\end{equation}

\subsection{Gromov-Hausdorff Distance for Registered Labels}\label{subsec:GH_registered}

The first notion of Gromov-Hausdorff distance that we consider is relevant when two $k$-labeled datasets $\mathcal{X}$ and $\mathcal{Y}$ have registered labels; i.e., one has prior knowledge that label set $X_1$ should correspond to $Y_1$, $X_2$ to $Y_2$, et cetera (see Example \ref{ex:labeled_metric_spaces}).  It will use the following generalization of distortion.

\begin{definition}[Generalized Distortion]\label{def:generalized_distortion}
The \define{distortion} associated with two correspondences $C_1\in \C(X_1, Y_1), C_2 \in \C(X_2,Y_2)$, with $X_1, X_2 \subset X$ and $Y_1, Y_2 \subset Y$ nonempty subsets, is defined to be 
\begin{equation}
    \dis(C_1, C_2)\coloneqq \sup_{\substack{(x_1, y_1) \in C_1 \\ (x_2, y_2) \in C_2}}|d_X(x_1,x_2)-d_Y(y_1,y_2)|.
\end{equation}
\end{definition}

This generalized definition recovers the notion used in \eqref{eq:distortion} when $X_1 = X_2 = X$, $Y_1 = Y_2 =  Y$, and we adopt the notational convention that $\dis(C,C)= \dis(C)$. We also remark that a version of the distortion in Definition \ref{def:generalized_distortion} was used recently in~\cite{needham2024stability}, in the context of comparing \textit{hypergraph} structures.

\begin{definition}[$k$-Labeled Gromov-Hausdorff Distance]\label{GH def1}
 Let $\mathcal{X}$ and $\mathcal{Y}$ be $k$-labeled metric spaces. The \define{$k$-labeled Gromov-Hausdorff distance} between $\mathcal{X}$ and $\mathcal{Y}$ is 
 \begin{equation}\label{eqn:k_labeled_GH}
     \mathsf{GH}_k(\mathcal{X},\mathcal{Y}) \coloneqq \frac{1}{2}\inf_{C}\max_{i,j=1,...,k}\dis(C_i, C_j),
 \end{equation}
 where the infimum is over all $k$-tuples of correspondences $C\coloneqq(C_1,\ldots,C_k)$ with $C_i \in \mathcal{C}(X_i,Y_i)$. 
\end{definition}


It will be useful to consider an alternative formulation, based on mappings between metric spaces.

\begin{definition}\label{GH def2}
Define the \define{mapping version of $k$-labeled Gromov-Hausdorff distance} between $k$-labeled metric spaces $\mathcal{X}$ and $\mathcal{Y}$ to be
\begin{equation}\label{eqn:k_labeled_mapping}
\widetilde{\mathsf{GH}}_k(\mathcal{X},\mathcal{Y})= \frac{1}{2}\inf_{\phi,\psi} \max_{i,j} \{\dis(\phi_i, \phi_j), \dis(\psi_i, \psi_j), \codis(\phi_i, \psi_j)\},
\end{equation}
where the infimum is over all length-$k$ sequences $\phi=(\phi_1,\ldots,\phi_k)$ and $\psi=(\psi_1,\ldots,\psi_k)$ of (not necessarily continuous) maps $\phi_i:X_i \to Y_i$ and $\psi_i:Y_i \to X_i$, and 
\begin{align*}
    \dis(\phi_i,\phi_j)&\coloneqq \sup_{x_i\in X_i, x_j\in X_j}\big|d_X(x_i,x_j)- d_Y(\phi_i(x_i), \phi_j(x_j))\big|,\\
    \dis(\psi_i,\psi_j)&\coloneqq \sup_{y_i\in Y_i, y_j\in Y_j}\big|d_X(\psi_i(y_i),\psi_j(y_j))- d_Y(y_i,y_j)\big|,\\
    \codis(\phi_i,\psi_j) &\coloneqq \sup_{x_i\in X_i, y_j \in Y_j} \big|d_X(x_i, \psi_j(y_j))-d_Y(\phi_i(x_i),y_j)\big|.
\end{align*}  
We will  adopt the notational  convention that $ \dis(\phi_i) \coloneqq \dis(\phi_i, \phi_i)$.
\end{definition}

The next proposition shows that the formulations of GH distance given in \eqref{eqn:k_labeled_GH} and \eqref{eqn:k_labeled_mapping} are equivalent. This mirrors the situation for classical GH distance, as was shown in \cite{kalton1999distances}. 

\begin{proposition}\label{prop:equality_of_GH}
    The Gromov-Hausdorff distances given in  \Cref{GH def1} and \Cref{GH def2} are equivalent; i.e., $\mathsf{GH}_k(\mathcal{X},\mathcal{Y})= \widetilde{\mathsf{GH}}_k(\mathcal{X},\mathcal{Y})$.
\end{proposition}

\begin{proof}
    The proof follows by mildly adapting the proof given in \cite{kalton1999distances} in the classical case. As the construction will be useful later, let us give details for the bound $\mathsf{GH}_k \geq \widetilde{\mathsf{GH}}_k$. Let $C_i \in \mathcal{C}(X_i,Y_i)$ for each $i \in \{1,\ldots,k\}$. By the Axiom of Choice, there exists $\phi_i: X_i \rightarrow Y_i$ and $\psi_i: Y_i\rightarrow X_i$ such that 
    \[
    \{(x,\phi_i(x)) \mid x \in X_i\} \cup \{(\psi_i(y),y) \mid y \in Y_i\} \subset C_i
    \]
    Observe that, for all $i,j$, 
\begin{equation*}
    \dis(C_i, C_j) \geq \max \{\dis(\phi_i, \phi_j), \dis(\psi_i,\psi_j), \codis(\phi_i,\psi_j), \codis(\phi_j,\psi_i)\}.
\end{equation*}
    This proves the claimed inequality.
\end{proof}

We now define an appropriate notion of equivalence for labeled metric spaces. 

\begin{definition}[Isomorphic Labeled Metric Spaces]\label{def:k_isomorphic}
We say that two $k$-labeled metric spaces $\mathcal{X}$ and $\mathcal{Y}$ are \define{$k$-isomorphic} if there exists an isometry $\phi: X\rightarrow Y$ such that $\phi(X_i) = Y_i$ for each $i=1,...,k$. Such a map is called a \define{$k$-isomorphism}. 
\end{definition}

A simple example of $2$-isomorphic labeled metric spaces is shown in Figure \ref{fig:Isomorphism_examples}. It is straightforward to show that $k$-isomorphism defines an equivalence relation on the space of $k$-labeled metric spaces. This leads to the following theorem on the metric structure of $\mathsf{GH}_k$.

\begin{theorem}\label{thm:k_GH}
    The $k$-labeled Gromov-Hausdorff distance defines a metric on the space of $k$-labeled metric spaces, considered up to $k$-isomorphism. 
\end{theorem}

The proof uses the following lemma, which is an adaptation of the analogous statement in the setting of GH distance, found in \cite[Proposition 1.1]{chowdhury2018explicit}. The proof of our lemma follows by mildly adapting the existing proof, so we omit it. 

\begin{lemma}\label{lem:optimal_correspondences}
    The distance $\mathsf{GH}_k(\mathcal{X},\mathcal{Y})$ is always realized by some correspondences $(C_1,C_2,\ldots,C_k)$; that is, the infimum in \eqref{eqn:k_labeled_GH} is a minimum.
\end{lemma}

\begin{proof}[Proof of Theorem \ref{thm:k_GH}]
    Symmetry and non-negativity are obvious, and the triangle inequality essentially follows by  standard arguments (see \cite{burago2001course} for the standard Gromov-Hausdorff setting and \cite{needham2024stability} for a variant of Gromov-Hausdorff distance which uses a distortion similar to that of Definition \ref{def:generalized_distortion}). 
    
    It remains to show that $\mathsf{GH}_k(\mathcal{X},\mathcal{Y}) = 0$ if and only if $\mathcal{X}$ and $\mathcal{Y}$ are $k$-isomorphic, which requires some new arguments. If $\phi:X \to Y$ is a $k$-isomorphism, then one obtains tuples $(\phi_i:X_i \to Y_i)_i$ and $(\psi_i = \phi_i^{-1}:Y_i \to X_i)_i$ by restriction, and it is easy to show that
    \begin{equation}\label{eqn:distortions_zero}
   \dis(\phi_i,\phi_j) =\dis(\psi_i,\psi_j) = \codis(\phi_i,\psi_j) = 0
    \end{equation}
    for all $i,j = 1,\ldots,k$, so that $\mathsf{GH}_k(\mathcal{X},\mathcal{Y}) = 0$, by Proposition \ref{prop:equality_of_GH}. 
    
    Conversely, suppose that $\mathsf{GH}_k(\mathcal{X},\mathcal{Y}) = 0$. Lemma \ref{lem:optimal_correspondences}, implies the existence of  correspondences $C_1,\ldots,C_k$ such that $\dis(C_i,C_j) = 0$ for all $i,j$. Following the proof of Proposition \ref{prop:equality_of_GH}, one can then construct maps $\phi_i:X_i \to Y_i$ and $\psi_i:Y_i \to X_i$ such that \eqref{eqn:distortions_zero} holds. Define disjoint sets $X_i'$ recursively by
    \[
    X_1' = X_1, \; X_2' = X_2 \setminus X_1', \; X_3' = X_3 \setminus (X_1' \cup X_2'), \ldots;
    \]
    some of these sets may be empty, but the collection covers $X$. We then define a map $\phi:X \to Y$ by 
    \[
    \phi(x) = \phi_i(x) \Leftrightarrow x \in X_i'.
    \]
    This is a well defined map, as the collection $X_1', \ldots, X_k'$ covers $X$, and the conditions \eqref{eqn:distortions_zero} imply that it satisfies
    \[
    d_X(x,x') = d_Y(\phi(x),\phi(x'))
    \]
    for all $x,x' \in X$. Moreover, we claim that $\phi(X_i) = Y_i$, hence that $\phi$ is surjective (so it is an isometry), which proves that it is a $k$-isomorphism. 
    To see the claim, first observe that if $x \in X_i \cap X_j$, then $\phi_i(x) = \phi_j(x)$; indeed, 
    \[
    0 = \dis(\phi_i,\phi_j) \geq |d_X(x,x) - d_Y(\phi_i(x),\phi_j(x))| = d_Y(\phi_i(x),\phi_j(x)).
    \]
    To show that $\phi(X_i) \subset Y_i$,  suppose that $x \in X_i$ and consider the two cases:
    \begin{enumerate}
        \item if $x \in X_i'$, then $\phi(x) = \phi_i(x) \in Y_i$;
        \item if $x \not \in X_i'$, then $x \in X_i \cap X_j'$ for some $j < i$. In particular, $x \in X_j'$ and $x \in X_i \cap X_j$, so
        \[
        \phi(x) = \phi_j(x) = \phi_i(x) \in Y_i.
        \]
    \end{enumerate}
    Finally, we show that $Y_i \subset \phi(X_i)$. Note that $\codis(\phi_i,\psi_i) = 0$ implies that $\phi_i$ and $\psi_i$ are inverse maps. Let $y \in Y_i$ and set $x = \psi_i(y) \in X_i$. Considering the same two cases above, one can show that $\phi(x) = y$, proving the claim.
\end{proof}

\begin{figure}
    \centering
    \includegraphics[width=0.85\linewidth]{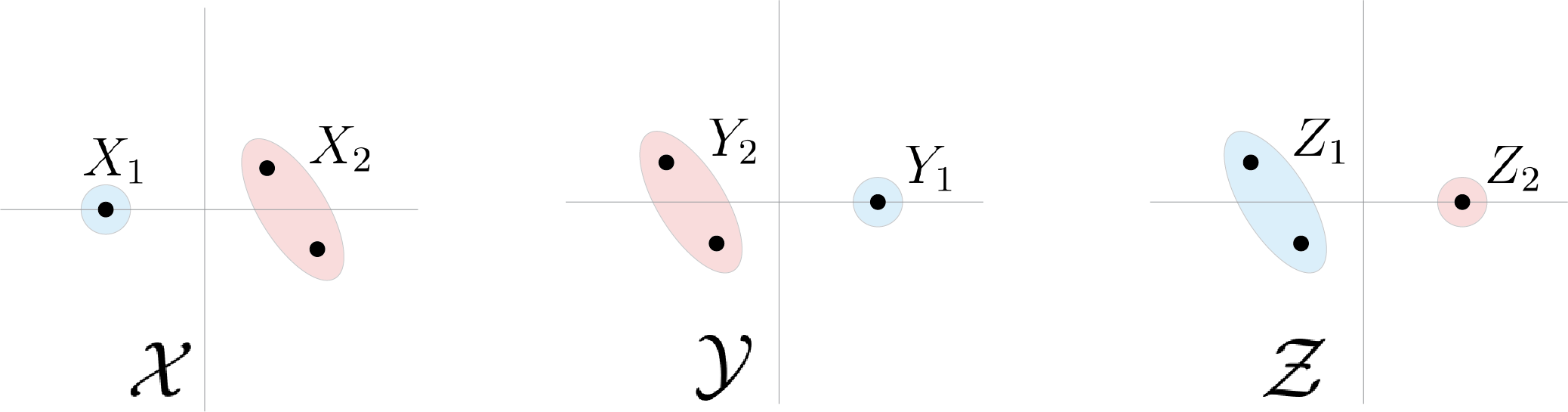}
    \caption{Simple examples of labeled datasets. Each metric space is a three-point subset of the Euclidean plane, with labels indicated by colors. We have $\mathsf{GH}_2(\mathcal{X},\mathcal{Y}) = 0$, as the underlying metric spaces are isometric, via an isometry which preserves label structure. On the other hand,  $\mathsf{GH}_2(\mathcal{X},\mathcal{Z}) > 0$, as, although the underlying metric spaces are still isometric, there is no isometry which preserves labels. Passing to GH distance for permuted labels, we have $\mathsf{GH}_{S_2}(\mathcal{X},\mathcal{Z}) = 0$ (realized by the nontrivial permutation in $S_2$).}
    \label{fig:Isomorphism_examples}
\end{figure}

\subsection{Gromov-Hausdorff for Permuted Labels}\label{subsec:GH_permuted}

The symmetric group on $k$ letters $S_k$ acts on the space of $k$-labeled metric spaces by permuting the labels: for a $k$-labeled metric space $\mathcal{X} = (X,d_X,L_X)$ and $\sigma \in S_k$, $\sigma \cdot X$ is the $k$-labeled metric space $(X,d_X,\sigma \cdot L_X)$, with 
\begin{equation}\label{eqn:symmetric_action}
\sigma \cdot L_X = (X_{\sigma(1)},\ldots,X_{\sigma(k)}).
\end{equation}
It is not hard to show that this action is by isometries with respect to $\mathsf{GH}_k$, so that the metric descends to a well-defined distance on the quotient. 

\begin{definition}[GH Modulo Relabeling]
    The \define{$k$-labeled Gromov-Hausdorff distance modulo relabeling} between $k$-labeled metric spaces $\mathcal{X}$ and $\mathcal{Y}$ is 
    \[
    \mathsf{GH}_{S_k}(\mathcal{X},\mathcal{Y}) \coloneqq \min_{\sigma \in S_k} \mathsf{GH}_k(\sigma \cdot \mathcal{X}, \mathcal{Y}).
    \]
\end{definition} 

The distance $\mathsf{GH}_{S_k}$ is appropriate in the setting that the labels for $\mathcal{X}$ and $\mathcal{Y}$ are unregistered---optimizing over $S_k$ amounts to finding the best registration between the labels (see Example \ref{ex:labeled_metric_spaces}). From the work above, we immediately obtain the following conclusions about the structure of $\mathsf{GH}_{S_k}$.

\begin{corollary}\label{cor:GH_mod_perm}
The $k$-labeled Gromov-Hausdorff distance modulo relabeling:
    \begin{enumerate}
        \item can be expressed as 
        \[
        \mathsf{GH}_{S_k}(\mathcal{X},\mathcal{Y}) = \min_{\sigma \in S_k} \widetilde{\mathsf{GH}}_k(\sigma \cdot \mathcal{X},\mathcal{Y});
        \]
        \item defines a metric on the space of $k$-labeled metric spaces, considered up to $k$-isomorphism and the action of the symmetric group. That is, $\mathsf{GH}_{S_k}(\mathcal{X},\mathcal{Y})= 0$ if and only if there is $\sigma \in S_k$ such that $\sigma \cdot \mathcal{X}$ and $\mathcal{Y}$ are $k$-isomorphic. 
    \end{enumerate}
\end{corollary}

\subsection{Gromov-Hausdorff for Unregistered Labels}\label{subsec:GH_unregistered}

It remains to consider the case of labeled metric spaces $\mathcal{X}$ and $\mathcal{Y}$ which are not only unregistered, but for which the labels are of potentially unequal size, so that the symmetric group no longer appropriately captures the notion of registration (see Example \ref{ex:labeled_metric_spaces}). The basic idea for adapting Gromov-Hausdorff distance to this setting is straightforward: we allow a label in $\mathcal{X}$ to be registered to multiple labels in $\mathcal{Y}$, and vice-versa; this is implemented by adding additional copies of labels as necessary to facilitate this registration. To formalize this idea, we introduce the following method for modifying a labeled metric space. Below, and throughout the rest of the paper, we use the following notation, for any positive integer $k$:
\[
[k] \coloneqq \{1,\ldots,k\}.
\]

\begin{definition}[Stabilization]\label{def:stabilization}
    Let $\mathcal{X} = (X,d_X,L_X=(X_1,\ldots,X_k))$ be a $k$-labeled metric space and let $q$ be an integer with $q \geq k$. A \define{$q$-stabilization} of $\mathcal{X}$ is a labeled metric space $\widehat{\mathcal{X}} = (X,d_X,\widehat{L}_X)$, where $\widehat{L}_X$ is of the form
    \[
    \widehat{L}_X =  (\widehat{X}_1,\ldots,\widehat{X}_q),
    \]
    such that there exists a surjective map $\rho:[q] \to [k]$ satisfying 
    \[
    \rho(i) = j \Leftrightarrow \widehat{X}_i = X_j.
    \]
    When we wish to supress reference to a specific $q$, we refer to $\widehat{\mathcal{X}}$ simply as a  \define{stabilization} of $\mathcal{X}$. 
\end{definition}

\begin{remark}\label{rem:comments_on_stabilization}
    Said more plainly, a stabilization is constructed by adding identical copies of each of the label sets (the number of copies of $X_j$ is given by the cardinality of $\rho^{-1}(j)$) and permuting the result. Observe that if $\mathcal{X}$ is a $k$-labeled metric space and $\sigma \in S_k$, then $\sigma \cdot \mathcal{X}$ is a $k$-stabilization of $\mathcal{X}$. 
\end{remark}

 The notion of stabilization allows us to define our most general notion of Gromov-Hausdorff distance for labeled metric spaces.

\begin{definition}[GH Modulo Stabilization]
    The \define{labeled Gromov-Hausdorff distance modulo stabilization} between labeled metric spaces $\mathcal{X}$ and $\mathcal{Y}$ (possibly with labels of unequal size) is 
    \begin{equation}\label{eqn:GH_stab}
    \mathsf{GH}_{st}(\mathcal{X},\mathcal{Y}) \coloneqq \inf_{\widehat{\mathcal{X}},\widehat{\mathcal{Y}}} \mathsf{GH}_{q}(\widehat{\mathcal{X}}, \widehat{\mathcal{Y}}),
    \end{equation}
    where the infimum is over the choice of $q$ and of $q$-stabilizations  $\widehat{\mathcal{X}}$ and $\widehat{\mathcal{Y}}$.
\end{definition}

We describe the metric properties of $\mathsf{GH}_{st}$ below, but first we record the following easy comparison result.

\begin{prop}
    For $k$-labeled metric spaces $\mathcal{X}$ and $\mathcal{Y}$, 
    \[
    \mathsf{GH}_{st}(\mathcal{X},\mathcal{Y}) \leq \mathsf{GH}_{S_k}(\mathcal{X},\mathcal{Y}) \leq \mathsf{GH}_{k}(\mathcal{X},\mathcal{Y}).
    \]
\end{prop}

Note that the simple examples given in Figures \ref{fig:Isomorphism_examples} and \ref{fig:isomorphism_examples2} show that these inequalities are not, in general, equalities. 

\begin{proof}
    It follows immediately by their definitions that $\mathsf{GH}_{S_k} \leq \mathsf{GH}_{k}$. By Remark \ref{rem:comments_on_stabilization}, a permutation action on a $k$-labeled metric space can be viewed as a stabilization, so that $\mathsf{GH}_{st} \leq \mathsf{GH}_{S_k}$.
\end{proof}

In order to describe the metric structure of $\mathsf{GH}_{st}$, we give two alternative formulations.  For a $k$-labeled metric space $\mathcal{X}$ and an $\ell$-labeled metric space $\mathcal{Y}$, consider the following two variants:
\begin{enumerate}
    \item A \define{correspondence version} of $\mathsf{GH}_{st}$ is defined by 
    \begin{equation}\label{eqn:GH_stab_correspondence}
\overline{\mathsf{GH}}_{st}(\mathcal{X},\mathcal{Y}) \coloneqq \frac{1}{2} \inf_{D \in \mathcal{C}([k],[\ell])} \inf_{C_D} \max_{(i,j),(i',j') \in D} \dis(C_{i,j},C_{i',j'}),
\end{equation}
where, given $D \in \mathcal{C}([k],[\ell])$, $C_D$ is a tuple of correspondences of the form
\[
C_D  = (C_{i,j} \in \mathcal{C}(X_i,Y_j))_{(i,j) \in D}.
\]
\item A \define{mapping version} of $\mathsf{GH}_{st}$ is defined by 
\begin{equation}\label{eqn:GH_stab_mapping}
    \widetilde{\mathsf{GH}}_{st}(\mathcal{X},\mathcal{Y})\coloneqq \frac{1}{2}\inf_{D\in \mathcal{C}([k],[\ell])}\inf_{\phi_{D},\psi_{D}} \max_{(i,j),(i',j') \in D} \{\dis(\phi_{i,j}, \phi_{i',j'}), \dis(\psi_{j,i}, \psi_{j',i'}), \codis(\phi_{i,j}, \psi_{j,i})\},
\end{equation}
where, given $D \in \mathcal{C}([k],[\ell])$, $\phi_D$ and $\psi_D$ are tuples of functions of the form 
\[
\phi_D=(\phi_{i,j}\colon  X_i \to Y_j)_{(i,j)\in D} \quad \mbox{and} \quad \psi_D=(\psi_{j,i} \colon Y_j \to X_i)_{(i,j)\in D}.
\]
\end{enumerate}

We then have the following equivalences.

\begin{prop}\label{prop:GH_equivalences}
    The Gromov-Hausdorff distances given in \eqref{eqn:GH_stab}, \eqref{eqn:GH_stab_correspondence} and \eqref{eqn:GH_stab_mapping} are equivalent; i.e., 
    \[
    \mathsf{GH}_{st}(\mathcal{X},\mathcal{Y}) = \overline{\mathsf{GH}}_{st}(\mathcal{X},\mathcal{Y}) = \widetilde{\mathsf{GH}}_{st}(\mathcal{X},\mathcal{Y})
    \]
    for any labeled metric spaces $\mathcal{X}$ and $\mathcal{Y}$. 
\end{prop}

\begin{proof}
    The proof that $\overline{\mathsf{GH}}_{st}(\mathcal{X},\mathcal{Y}) = \widetilde{\mathsf{GH}}_{st}(\mathcal{X},\mathcal{Y})$ follows by an argument which is very similar to the proof of Proposition \ref{prop:equality_of_GH} (which, in turn, follows \cite{kalton1999distances}), so we focus on $\mathsf{GH}_{st}(\mathcal{X},\mathcal{Y}) = \overline{\mathsf{GH}}_{st}(\mathcal{X},\mathcal{Y})$, which requires a more novel proof. 

    To show that $\mathsf{GH}_{st}(\mathcal{X},\mathcal{Y}) \leq \overline{\mathsf{GH}}_{st}(\mathcal{X},\mathcal{Y})$, we show that, given $D \in \mathcal{C}([k],[\ell])$, and a tuple of correspondences 
    \[
    C_D = (C_{i,j} \in \mathcal{C}(X_i,X_j))_{(i,j) \in D},
    \]
    we can construct $q$-stabilizations (for some choice of $q$ to be determined) $\widehat{\mathcal{X}}$ and $\widehat{\mathcal{Y}}$ such that 
 
\begin{equation}\label{eqn:GH_cost_bound}
    \inf_{C_i \in \mathcal{C}(\widehat{X}_i,\widehat{Y}_i)} \max_{i,j = 1,\ldots,q} \dis(C_i,C_j) \leq \max_{(i,j),(i',j') \in D} \dis(C_{i,j},C_{i',j'}).
    \end{equation}
    This is conceptually straightforward, but requires some bookkeeping. To this end, for each $(i,j) \in D$, we make copies
    \[
    X_{i,j} = X_i \quad \mbox \quad Y_{i,j} = Y_j.
    \]
    Letting $q = |D|$, we choose an (arbitrary) ordering 
    of the elements of $D$ as 
    \[
    \big( (i_1,j_1),(i_2,j_2),\ldots,(i_q,j_q) \big)
    \]
    and reindex our copies as 
    \[
    \widehat{X}_m = X_{i_m,j_m} \quad \mbox{and} \quad \widehat{Y}_m = Y_{i_m,j_m},
    \]
    so that the labels in our stabilizations are defined to be 
    \[
    \widehat{L}_X = (\widehat{X}_1,\ldots,\widehat{X}_q) \quad \mbox{and} \quad \widehat{L}_Y = (\widehat{Y}_1,\ldots,\widehat{Y}_q).
    \]
    To establish the estimate \eqref{eqn:GH_cost_bound}, for each $m = 1,\ldots,q$, we define a correspondence $C_m \in \mathcal{C}(\widehat{X}_m,\widehat{Y}_m)$ by setting
    \[
    C_m \coloneqq C_{i_m,j_m} \in \mathcal{C}(X_{i_m},Y_{j_m}) = \mathcal{C}(X_{i_m,j_m},Y_{i_m,j_m}) = \mathcal{C}(\widehat{X}_m,\widehat{Y}_m).
    \]
    Then
    \[
    \max_{m,n} \dis(C_m,C_n) = \max_{(i_m,j_m),(i_n,j_n) \in D} \dis(C_{i_m,j_m},C_{i_n,j_n}),
    \]
    which proves the desired inequality. 
    
    The proof that $\mathsf{GH}_{st}(\mathcal{X},\mathcal{Y}) \geq \overline{\mathsf{GH}}_{st}(\mathcal{X},\mathcal{Y})$ is similar: we show that, given $q$-stabilizations $\widehat{\mathcal{X}}$ and $\widehat{\mathcal{Y}}$, and correspondences $C_i \in \mathcal{C}(\widehat{X}_i,\widehat{Y}_i)$ for $i \in \{1,\ldots,q\}$, we can construct $D \in \mathcal{C}([k],[\ell])$ such that
    \begin{equation}\label{eqn:GH_cost_bound2}
     \max_{i,j = 1,\ldots,q} \dis(C_i,C_j) \geq \inf_{C_D} \max_{(i,j),(i',j') \in D} \dis(C_{i,j},C_{i',j'}).
    \end{equation}
    The construction is done essentially by running the previous construction in the reverse order, so we only sketch the ideas. Letting $\rho_X:[q] \to [k]$ and $\rho_Y:[q] \to [\ell]$ be the respective surjective maps from Definition \ref{def:stabilization}, we define $D$ by
    \[
    D = \{(i,j) \in [k] \times [\ell] \mid \exists \, m \in [q] \mbox{ such that } \rho_X(m) = i \mbox{ and } \rho_Y(m) = j\}.
    \]
    For each $(i,j) \in D$, we set $C_{i,j} = C_m$, where $m$ satisfies $\rho_X(m) = i$ and $\rho_Y(m) = j$. This defines a collection $C_D$ which can then use to establish \eqref{eqn:GH_cost_bound2}.
\end{proof}

    This equivalence gives the following useful corollary.

    \begin{corollary}\label{cor:GH_stab_realized}
        The infimum in \eqref{eqn:GH_stab} is achieved by some stabilizations $\widehat{\mathcal{X}}$ and $\widehat{\mathcal{Y}}$. In particular, if $\mathcal{X}$ is a $k$-labeled metric space and $\mathcal{Y}$ is an $\ell$-labeled metric space, then there is a minimizing $q$-stabilization with  $q \leq k\cdot \ell$. 
    \end{corollary}

    \begin{proof}
        By the same reasoning as that used in the proof of Lemma \ref{lem:optimal_correspondences}, for any choice of $D$, 
        \[
        \inf_{C_D} \max_{(i,j),(i',j') \in D} \dis(C_{i,j},C_{i',j'})
        \]
        is realized by a collection of correspondences $C_D$. Since the set $\mathcal{C}([k],[\ell])$ is finite, \eqref{eqn:GH_stab_correspondence} must be realized by some correspondence $D$. By the proof of Proposition \ref{prop:GH_equivalences}, this minimizing $D$ can be used to construct minimizing stabilizations whose labels have cardinality at most $|D| \leq k \cdot \ell$.  
    \end{proof}

    We are then able to define an appropriate notion of equivalence between labeled metric spaces with labels of arbitrary cardinality, and to then describe the metric properties of $\mathsf{GH}_{st}$. 

\begin{figure}
        \centering
        \includegraphics[width=0.95\linewidth]{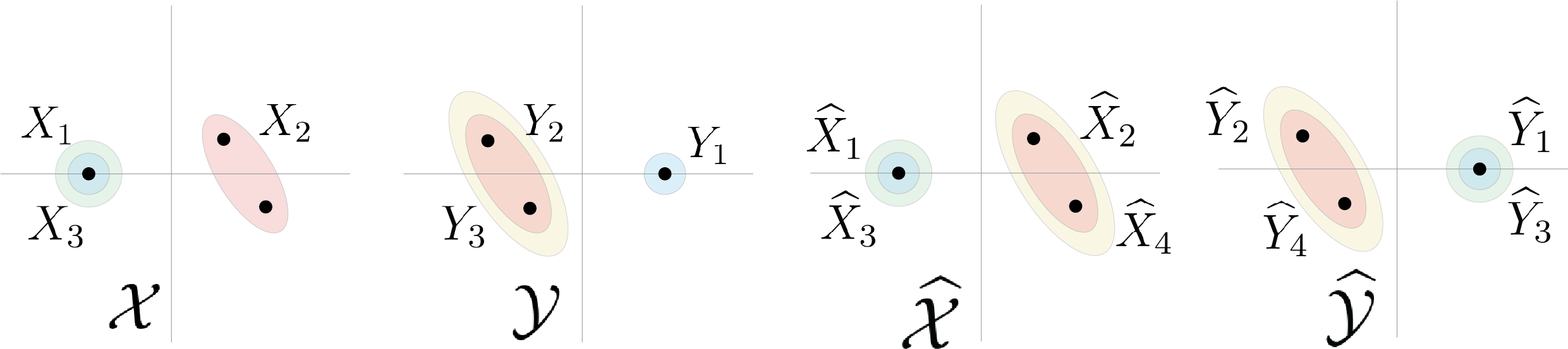}
        \caption{Examples of stabilization. The labeled metric spaces $\mathcal{X}$ and $\mathcal{Y}$ have repeated labels $X_1 = X_3$ and $Y_2 = Y_3$. These repeated labels  can't be perfectly matched between $\mathcal{X}$ and  $\mathcal{Y}$, so that  $\mathsf{GH}_{S_3}(\mathcal{X},\mathcal{Y}) > 0$. The spaces $\widehat{\mathcal{X}}$ and $\widehat{\mathcal{Y}}$ are 4-stabilizations of $\mathcal{X}$ and $\mathcal{Y}$, respectively, which are $4$-isomorphic, hence $\mathsf{GH}_{st}(\mathcal{X},\mathcal{Y}) = 0$.}
        \label{fig:isomorphism_examples2}
    \end{figure}
    
    \begin{definition}
        Labeled metric spaces $\mathcal{X}$ and $\mathcal{Y}$ are called \define{stably equivalent} if there exist $q$-stabilizations $\widehat{\mathcal{X}}$ and $\widehat{\mathcal{Y}}$ such that $\widehat{\mathcal{X}}$ and $\widehat{\mathcal{Y}}$ are $q$-isomorphic (in the sense of Definition \ref{def:k_isomorphic}). 
        \end{definition}
        
An example of stably equivalent labeled metric spaces is given in Figure \ref{fig:isomorphism_examples2}.

    \begin{theorem}\label{thm:GH_mod_stab}
        The labeled Gromov-Hausdorff distance modulo stabilizations defines a metric on the space of labeled metric spaces, considered up to stable equivalence. 
    \end{theorem}

    \begin{proof}
        Symmetry and non-negativity are obvious. If $\mathcal{X}$ and $\mathcal{Y}$ are stably equivalent, say via $q$-isomorphic stabilizations $\widehat{\mathcal{X}}$ and $\widehat{\mathcal{Y}}$, respectively, then 
        \[
        \mathsf{GH}_{st}(\mathcal{X},\mathcal{Y}) \leq \mathsf{GH}_q(\widehat{\mathcal{X}},\widehat{\mathcal{Y}}) = 0,
        \]
        where we have applied Theorem \ref{thm:k_GH} to get the last equality. Conversely, if $\mathsf{GH}_{st}(\mathcal{X},\mathcal{Y}) = 0$, then Corollary \ref{cor:GH_stab_realized} tells us that there exist $q$-stabilizations (for some $q$) $\widehat{\mathcal{X}}$ and $\widehat{\mathcal{Y}}$ with 
        \[
        0 = \mathsf{GH}_{st}(\mathcal{X},\mathcal{Y}) = \mathsf{GH}_q(\widehat{\mathcal{X}},\widehat{\mathcal{Y}}),
        \]
        whence $\widehat{\mathcal{X}}$ and $\widehat{\mathcal{Y}}$ must be $q$-isomorphic (once again, by Theorem \ref{thm:k_GH}). 
        
        It only remains to prove the triangle inequality; here, it is more convenient to use the (equivalent, by Proposition \ref{prop:GH_equivalences}) correspondence formulation $\overline{\mathsf{GH}}_{st}$. Let $\mathcal{X}$, $\mathcal{Y}$ and $\mathcal{Z}$ be labeled metric spaces, say with $|L_X|=k$, $|L_Y| = \ell$ and $|L_Z| = m$. Choose $D_{XY}$ and $C_{D_{XY}}$ which realize $\overline{\mathsf{GH}}_{st}(\mathcal{X},\mathcal{Y})$ and  $D_{YZ}$ and $C_{D_{YZ}}$ which realize $\overline{\mathsf{GH}}_{st}(\mathcal{Y},\mathcal{Z})$. We denote the elements of $C_{D_{XY}}$ as 
        \[
        C_{D_{XY}} = (C^{XY}_{i,j} \in \mathcal{C}(X_i,Y_j))_{(i,j) \in D_{XY}}
        \]
        and we use a similar convention for the elements of $C_{D_{YZ}}$. 
        We then compose the  correspondences $D_{XY}$ and $D_{YZ}$ to obtain 
        \[
        D \coloneqq \{(i,j) \in [k] \times [m] \mid \exists \, n \in [\ell] \mbox{ s.t. } (k,n) \in D_{XY} \mbox{ and } (n,j) \in D_{YZ}\} \in \mathcal{C}([k],[m]).
        \]
        Going forward, for each $(i,j) \in D$, we fix an $n \in [\ell]$ satisfying the defining condition. For each $(i,j) \in D$, we define
        \[
        C_{i,j} \coloneqq \{(x,z) \mid \exists \, y \in Y \mbox{ s.t. } (x,y) \in C^{XY}_{i,n} \mbox{ and } (y,z) \in C^{YZ}_{n,j}\}
        \]
        (with $n$ as in the definition of $D$). Then, for any $(i,j),(i',j') \in D$, a standard argument shows that   
        \[
            \dis(C_{i,j},C_{i',j'}) \leq \dis(C^{XY}_{i,n},C^{XY}_{i',n'}) + \dis(C^{YZ}_{n,j},C^{YZ}_{n',j'}) 
        \]
        (cf. \cite[Lemma 4.1]{gulecture}). We have 
        \begin{align*}
            \mathsf{GH}_{st}(\mathcal{X},\mathcal{Z}) = \overline{\mathsf{GH}}_{st}(\mathcal{X},\mathcal{Z}) &\leq \frac{1}{2} \max_{(i,j),(i',j')} \dis(C_{i,j},C_{i',j'}) \\
            &\leq \frac{1}{2} \max_{(i,j),(i',j') \in D} \dis(C^{XY}_{i,n},C^{XY}_{i',n'}) + \dis(C^{YZ}_{n,j},C^{YZ}_{n',j'}) \\
            &\leq \frac{1}{2} \max_{(i,n),(i',n') \in D_{XY}} \dis(C^{XY}_{i,n},C^{XY}_{i',n'}) + \frac{1}{2} \max_{(n,j),(n',j') \in D_{YZ}} \dis(C^{YZ}_{n,j},C^{YZ}_{n',j'}) \\
            &= \overline{\mathsf{GH}}_{st}(\mathcal{X},\mathcal{Y}) + \overline{\mathsf{GH}}_{st}(\mathcal{Y},\mathcal{Z}) = \mathsf{GH}_{st}(\mathcal{X},\mathcal{Y}) + \mathsf{GH}_{st}(\mathcal{Y},\mathcal{Z}),
        \end{align*}
        and this completes the proof.
    \end{proof}

\subsection{Lower Bounds on Labeled Gromov-Hausdorff Distance}

We now establish a simple lower bound on the labeled Gromov-Hausdorff distance, which will be used to prove a stability of a certain topological invariant of labeled data in \S\ref{sec:stability}. We focus on the setting of registered labels, as it will be most useful for our purposes below. In the following, fix $k$-labeled metric spaces $\mathcal{X}$ and $\mathcal{Y}$. 

\begin{definition}\label{def:q_diameter}
    For $Q \subset [k]$, the \define{$Q$-diameter} of $\mathcal{X}$ is 
    \[
    \mathrm{diam}_{Q}(\mathcal{X}) \coloneqq \mathrm{diam}(\cup_{i \in Q} X_i) = \sup_{x,x' \in \cup_{i \in Q} X_i} d_X(x,x').
    \]
\end{definition}

Our goal is to prove the following estimate.

\begin{proposition}\label{prop:GH_k_estimate}
    The $k$-labeled Gromov-Hausdorff distance is lower bounded as 
    \[
    \frac{1}{2} \max_{Q \subset [k]} |\mathrm{diam}_{Q}(\mathcal{X}) - \mathrm{diam}_{Q}(\mathcal{Y})| \leq \mathsf{GH}_k(\mathcal{X},\mathcal{Y}).
    \]
\end{proposition}

Observe that if $k=1$, then this proposition reduces to the corresponding known estimate for (classical) Gromov-Hausdorff distance---see~\cite[Theorem 3.4]{memoli2012some}. The proof will follow from two preliminary lemmas, each of which gives an additional lower bound on $\mathsf{GH}_k$, so that they are potentially of independent interest. For the first lemma, observe that we can  consider  $\mathcal{X}$ and $\mathcal{Y}$ as metric spaces by forgetting their label structures; abusing notation, we denote these underlying metric spaces as $\mathcal{X}$ and $\mathcal{Y}$, respectively.

\begin{lemma}\label{lem:GH_k_estimate_1}
For any $k$-labeled metric spaces $\mathcal{X}$ and $\mathcal{Y}$, we have 
\[
\mathsf{GH}(\mathcal{X},\mathcal{Y}) \leq \mathsf{GH}_k(\mathcal{X},\mathcal{Y}).
\]
\end{lemma}

\begin{proof}
    Given $C_i \in \mathcal{C}(X_i,Y_i)$ for all $i \in [k]$, define a set $C \subset X \times Y$ by 
    \[
    (x,y) \in C \Leftrightarrow (x,y) \in C_i \mbox{ for some $i$}.
    \]
    Since the label sets cover $X$ and $Y$ respectively, it is easy to see that $C \in \mathcal{C}(X,Y)$. Moreover,
    \begin{equation}\label{eqn:distortion_bound}
    \mathrm{dis}(C) \leq \max_{i,j} \mathrm{dis}(C_i,C_j),
    \end{equation}
    and this is enough to prove the lemma. To verify \eqref{eqn:distortion_bound}, observe that for any $(x,y),(x',y') \in C$, we have that $(x,y) \in C_i$ and $(x',y') \in C_j$ for some $i$ and $j$, so that 
    \[
    |d_X(x,x') - d_Y(y,y')| \leq \mathrm{dis}(C_i,C_j).
    \]
    This completes the proof.
\end{proof}

Given a subsequence $\mathcal{I} = (i_1,i_2,\ldots,i_\ell)$ of $(1,2,\ldots,k)$, let $\mathcal{X}_\mathcal{I} = (X_\mathcal{I},d_{X_\mathcal{I}},L_{X_\mathcal{I}})$ denote the $\ell$-labeled metric space with 
\begin{align*}
    X_\mathcal{I} &= X_{i_1} \cup X_{i_2} \cup \cdots \cup X_{i_\ell}, \\
    d_{X_\mathcal{I}} &= d_X|_{X_\mathcal{I} \times X_\mathcal{I}}, \mbox{ and} \\
    L_{X_\mathcal{I}} &= (X_{i_1},X_{i_2},\ldots,X_{i_\ell}).
\end{align*}
This allows us to state the second lemma.

\begin{lemma}\label{lem:GH_k_estimate_2}
For any $k$-labeled metric spaces $\mathcal{X}$ and $\mathcal{Y}$, and any subsequence $\mathcal{I}$ of size $\ell \leq k$, we have
    \[
\mathsf{GH}_\ell(\mathcal{X}_\mathcal{I},\mathcal{Y}_\mathcal{I}) \leq \mathsf{GH}_k(\mathcal{X},\mathcal{Y}).
    \]
\end{lemma}

\begin{proof}
    Let $C_i \in \mathcal{C}(X_i,Y_i)$ for all $i \in [k]$. Then 
    \[
    \max_{i_a,i_b \in \mathcal{I}} \mathrm{dis}(C_{i_a},C_{i_b}) \leq \max_{i,j \in [k]} \mathrm{dis}(C_i,C_j),
    \]
    and the claim follows by infimizing over correspondences $(C_1,\ldots,C_k)$. 
\end{proof}

\begin{proof}[Proof of Proposition \ref{prop:GH_k_estimate}]
For each choice of $Q \subset [k]$ of size $|Q|=\ell$, let $\mathcal{I}_Q$ denote the associated subsequence of $(1,\ldots,k)$. Then we have 
\begin{align}
    \frac{1}{2} |\mathrm{diam}_{Q}(\mathcal{X}) - \mathrm{diam}_{Q}(\mathcal{Y})| &\leq \mathsf{GH}\big(\mathcal{X}_{\mathcal{I}_Q}, \mathcal{Y}_{\mathcal{I}_Q}\big) \label{eqn:GH_k_estimate_1} \\
    &\leq \mathsf{GH}_\ell\big(\mathcal{X}_{\mathcal{I}_Q}, \mathcal{Y}_{\mathcal{I}_Q}\big) \label{eqn:GH_k_estimate_2} \\
    &\leq \mathsf{GH}_k\big(\mathcal{X}, \mathcal{Y}\big) \label{eqn:GH_k_estimate_3},
\end{align}
where \eqref{eqn:GH_k_estimate_1} follows by the classical Gromov-Hausdorff estimate~\cite[Theorem 3.4]{memoli2012some}, \eqref{eqn:GH_k_estimate_2} follows by Lemma \ref{lem:GH_k_estimate_1}, and  \eqref{eqn:GH_k_estimate_3} follows by Lemma \ref{lem:GH_k_estimate_2}. 
\end{proof}
    
\subsection{Gromov-Hausdorff Distance for Chromatic Metric Spaces}\label{ss:chromatic}

We conclude this section with a discussion of cocurrent and indepedent work by Draganov et al.~\cite{draganov2025gromov}, which examines a related metric, whose definition we now recall. A \define{chromatic metric space} $(X,d_X, \mathcal{L}_X)$ consists of a metric space $(X,d_X)$ and a \define{coloring map} $\mathcal{L}_X:X \rightarrow \mathbb{N}$, where $\mathbb{N}$ is conceptualized as a set of \textit{colors}. In the following, let $(X,d_X,\mathcal{L}_X), (Y,d_Y,\mathcal{L}_Y)$ be chromatic metric spaces. Given a family $\Sigma$ of subsets  of $\mathbb{N}$, a map $f:X \rightarrow Y$  is called \define{$\Sigma$-constrained} if $f(\mathcal{L}_X^{-1}(\sigma))\subset \mathcal{L}_Y^{-1}(\sigma)$ for any $\sigma\in \Sigma$. 

\begin{example}
    A  $\Sigma$-constrained map with $\Sigma=\{\{0\}, \{1,2\}\}$, where $0,1,2$ represent red, green, blue, respectively, prescribes that red points in $X$ must be mapped to red points in $Y$, while the set of green and blue points in $X$ must be mapped to the set of green and blue points in $Y$; that is, each green point  is mapped to either a green or blue point, as long as the mapping preserves the combined green-blue color classes.
\end{example}

 The \define{$\Sigma$-constrained Gromov-Hausdorff distance} is defined as
\begin{equation}\label{C-GH}
   \mathsf{GH}_\Sigma((X,d_X, \mathcal{L}_X), (Y,d_Y, \mathcal{L}_Y))\coloneqq \frac{1}{2}\inf_{\substack{f:X\rightarrow Y, g: Y\rightarrow X \\ \Sigma\textrm{-constrained}}}\max\{\dis(f), \dis(g), \codis(f,g)\},
\end{equation}
where $\dis$ and $\codis$ are defined as in Definition \ref{GH def2}. We now give a result which precisely compares our labeled metric space framework to the chromatic metric space framework of~\cite{draganov2025gromov}. Its statement requires some additional terminology. Given a constraint set $\Sigma$, we say that a chromatic metric space $(X,d_X,\mathcal{L}_X)$ is \define{consistent} with $\Sigma$ if
\begin{enumerate}
    \item for each $x \in X$ there exists $\sigma \in \Sigma$ such that $\mathcal{L}_X(x) \in \sigma$ (i.e., the image of $\mathcal{L}_X$ is contained in $\bigcup_{\sigma \in \Sigma} \sigma$), and
    \item for each $\sigma \in \Sigma$, there exists $x \in X$ such that $\mathcal{L}_X(x) \in \sigma$.
\end{enumerate}

\begin{proposition}\label{prop:chromatic}
Fix a finite constraint set $\Sigma = \{\sigma_j\}_{j=1}^k$. Given a chromatic metric space $(X,d_X,\mathcal{L}_X)$  which is consistent with $\Sigma$, we construct a $k$-labeled metric space
\begin{equation}\label{eqn:tau_map}
\tau(X,d_X,\mathcal{L}_X) \coloneqq (X,d_X,L_X = (X_1,\ldots,X_k)), \quad \mbox{where} \quad X_j = \mathcal{L}_X^{-1}(\sigma_j).
\end{equation}
Then, given two $\Sigma$-consistent chromatic metric spaces, 
\[
\mathsf{GH}_\Sigma\big((X,d_X,\mathcal{L}_X),(Y,d_Y,\mathcal{L}_Y) \big) = \mathsf{GH}_k\big(\tau(X,d_X,\mathcal{L}_X),\tau(Y,d_Y,\mathcal{L}_Y) \big).
\]
\end{proposition}

\begin{proof}
    Observe that the definition of consistency implies that the construction of $L_X$ in \eqref{eqn:tau_map} actually yields  valid labels (i.e., a finite cover by nonempty subsets). We turn to proving the equivalence of the distances---in the proof, we consider the mapping version $\widetilde{\mathsf{GH}}_k$, which is equivalent to $\mathsf{GH}_k$, by Proposition \ref{prop:equality_of_GH}. First, given  $\Sigma$-constrained maps $f:X \to Y$ and $g:Y \to X$, we construct sequences $\phi = (\phi_1,\ldots,\phi_k)$ and $\psi = (\psi_1,\ldots,\psi_k)$ by 
    \[
    \phi_i = f|_{X_i} \quad \mbox{and} \quad \psi_i = g|_{Y_i}.
    \]
    Since $f,g$ are $\Sigma$-constrained, we have that these maps are of the form $\phi_i:X_i \to Y_i$ and $\psi_i:Y_i \to X_i$. Moreover, it is straightforward to show from the definitions that 
    \[
    \dis(f) \leq \max_{i,j} \dis(\phi_i,\phi_j),
    \]
    with similar bounds for the other distortion functionals. It follows that $\Sigma$-constrained GH is bounded from above by $k$-GH. Similarly, given sequences of maps $\phi$ and $\psi$, one can construct $\Sigma$-constrained maps $f$ and $g$. Going in this direction is not as straightforward, as the domains of the $\phi_i$'s (respectively, $\psi_i$'s) are not necessarily disjoint, but this can be done by following an approach similar to the proof of Theorem \ref{thm:k_GH}. Moreover, the resulting maps have required distortion bounds of the form
    \[
    \max_{i,j} \dis(\phi_i,\phi_j) \leq \dis(f).
    \]
    This construction gives the remaining inequality.
\end{proof}

The interpretation of this result as follows. Under the restriction to  constraint sets $\Sigma$ of cardinality $k$, the proposition says that the space of chromatic metric spaces which are consistent with $\Sigma$ isometrically embeds (via $\tau$) into the space of $k$-labeled metric spaces, with respect to $\mathsf{GH}_\Sigma$ and $\mathsf{GH}_k$. In this sense, the labeled metric space framework generalizes the chromatic metric space framework, in that it is able to handle more general structures, such as metric spaces with multiple labels. On the other hand, the chromatic metric space framework is more natural for handling situations such as matchings which are only \textit{partially constrained}---this would correspond to the situation where the image of $\mathcal{L}_X$ is not contained in $\bigcup_{\sigma \in \Sigma} \sigma$, which violates the consistency assumption required for our isometric embedding result. Note that the consistency assumption 1 above can be realized by adding the image of $\mathcal{L}_X$ to $\Sigma$, which enables our framework to handle the partially constrained scenario in the chromatic setting.

\section{Persistent Homology of Labeled Datasets}\label{sec:persistent_homology}

In this section, we introduce a notion of persistent homology which is designed to capture the structure of labeled data. We then show that our notion of persistent homology is stable with respect to the labeled Gromov-Hausdorff distance introduced in the previous section. 

\subsection{\texorpdfstring{$P$-modules and Interleaving Distance}{P-modules and Interleaving Distance}}\label{sec:def P-repr}

We begin by recalling a standard construction from persistent homology theory.

\begin{definition}[$P$-module]\label{def:P-module}
    Let $(P, \preccurlyeq)$ be a partially ordered set and let $\mathbb{F}$ be a field. A \textbf{$P$-module} is a functor $V:P \to \mathrm{Vect}_\mathbb{F}$ from the poset category $P$ to the category of vector spaces over $\mathbb{F}$. Explicitly, $V$ defines:
    \begin{itemize}
        \item for each $p \in P$, an $\mathbb{F}$-vector space $V(p)$; 
        \item for each relation $p \preccurlyeq p'$, a linear map $V(p \preccurlyeq p'):V(p) \to V(p')$, such that compositions of these maps are compatible with the poset structure.
    \end{itemize}
    
    We use $\mathrm{Mod}_P$ to denote the \define{category of $P$-modules}. The objects of this category are $P$-modules and the morphisms are natural transformations. That is, a morphism $F:V \to V'$ is a collection of linear maps $F_p\colon V(p) \to V'(p)$ (one for each $p\in P$) such that, for any relation $p\preccurlyeq p'$, we have a commutative  diagram
\begin{center}
       \begin{tikzcd}
V(p) \arrow[r, "F_p"] \arrow[d, "V(p \preccurlyeq p')"'] & V'(p) \arrow[d, "V'(p \preccurlyeq p')"] \\
V(p') \arrow[r, "F_{p'}"]                               & V'(p') \,.                                
\end{tikzcd} 
\end{center}
We mildly abuse notation and write $V \in \mathrm{Mod}_P$ to indicate that $V$ is an object in the category.
\end{definition}

\begin{remark}
    Using the terminology of~\cite{BubenikdeSilvaScott2015}, a $P$-module is an example of a \textit{generalized persistence module}. We have chosen our terminology for simplicity and to emphasize the role of the poset in the construction.
\end{remark}

In the present work, we are interested  posets of the form $\R \times P$, where $\R$ is the real numbers taken with the usual ordering and $(P,\preccurlyeq)$ is some fixed, finite poset (to be specified later). We abuse notation and continue to use $\preccurlyeq$ for the product poset structure on $\R \times P$; that is, $(r,p) \preccurlyeq (r',p')$ if and only if $r \leq r'$ and $p \preccurlyeq p'$.

\begin{example}\label{ex:mixup} In \cite{mixup}, Wagner, Arustamyan, Wheeler, and Bubenik study the relationship between filtered simplicial complexes $L\subseteq K$ using, in our language, modules over the poset $\R \times P$, where $P = \{L,K\}$, with $L \preccurlyeq K$.  Specifically, they consider finite discrete filtrations $L_1\hookrightarrow L_2\hookrightarrow\cdots L_n$ and $K_1\hookrightarrow K_2\hookrightarrow\cdots K_n$,\footnote{In practice, each $L_i$ or $K_i$ is a $\VR$-complex for some filtration parameter determined by $i$; see \S\ref{sec:stability}.} with $\mathrm{inc}:L_i\hookrightarrow K_i$ and satisfying further technical conditions (see \cite[\S3]{mixup}). The generalized persistence module they consider is given by the functor $H_j(\bullet_i)$, with $\bullet=L, K$ and linear maps given by the maps induced on homology by inclusion of complexes.

 The authors study the ``mixup barcode": for each bar $(b,d)$ in the barcode of the $H_j(L_i)$-submodule, there is a corresponding cycle $\alpha$ in $C_j(L_b)$; the mixup barcode is the set of triples $(b,d',d)$ associated to each such $\alpha$, where $d'$ is the smallest filtration value for which the map
 \[
 H_j(L_b)\to H_j(K_{d'})
 \]
 sends $[\alpha]$ to zero. In the language of \cite{CSEHM}, the intervals $(b,d')$ correspond to the barcode of the persistence module whose vector spaces are the images of the inclusion-induced maps $H_j(L_i)\to H_j(K_i)$ and whose linear maps are those of the $H_j(K_i)$ module, restricted to the images of the $H_j(L_i)$ module. We explain in Example \ref{ex:mixupL} how our generalized persistence landscape recovers the mixup barcode; together with Theorem \ref{thm:stability_of_landscapes} this proves a stability result for mixup barcodes in the case of the Vietoris-Rips filtration.
 \end{example}

The product poset $\R \times P$ comes with an additional shifting structure, due to its $\R$-factor, as defined below. 

\begin{definition}[$\epsilon$-Shift]
    For $\epsilon \geq 0$, the \define{$\epsilon$-shift map} $\R \times P \to \R \times P$ is given by $(r,p) \mapsto (r + \epsilon, p)$. This induces a map $\mathrm{Mod}_{\R \times P} \to \mathrm{Mod}_{\R \times P}$, denoted $V \mapsto V_\epsilon$, where
    \[
    V_\epsilon(r,p) \coloneqq V(r + \epsilon, p).
    \]
\end{definition}

The shifting structure leads to a commonly used construction in persistent homology called \textit{interleaving}, originally going back to~\cite{chazal2009proximity}. Indeed, it is not hard to show that the collection of all $\epsilon$-shift maps defines a \textit{superlinear family of translations}, in the sense of~\cite{BubenikdeSilvaScott2015}, or that the induced family of maps is a \textit{strict flow} on $\mathrm{Mod}_{\R \times P}$, in the sense of~\cite{de2018theory}. By the general theory of~\cite{BubenikdeSilvaScott2015,de2018theory}, the shift maps therefore induce an extended pseudometric---i.e., a metric-like structure which can take the value $+\infty$ and which may assign distance zero to distinct points---on (the objects of) $\mathrm{Mod}_{\R \times P}$, called the \textit{interleaving distance}, denoted 
\[
    \mathsf{Int}: \mathrm{Mod}_{\R \times P}  \times  \mathrm{Mod}_{\R \times P} \to \R \cup \{+\infty\}. 
\]

As we will need to work with the interleaving distance in detail, let us recall the general construction from~\cite{BubenikdeSilvaScott2015} for this special case.

\begin{definition}[Interleaving Distance]\label{def:epsilon-interleaving}
For $\epsilon \geq 0$, we say that $(\R \times P)$-modules $V$ and $V'$ are \define{$\epsilon$-interleaved} if there exist $(\R \times P)$-module morphisms $\phi:V \to V'_\epsilon$ and $\psi:V' \to V_\epsilon$ such that the diagrams
\[
\resizebox{13cm}{!}{
\input{tikz/interleaving_2_tikz}
}\]
commute for every $(r,p)\in \R \times P$. The \define{interleaving distance} between  $V$ and $V'$ is defined as 
\[\mathsf{Int}(V,V') \coloneqq \inf\{\epsilon\in [0,\infty)|\, V,V'\text{ are }\epsilon\text{-interleaved}\},\]
where $\inf \emptyset \coloneqq +\infty$.
\end{definition}

\subsection{Stability for Point Clouds with Registered Labels}\label{sec:stability}

Recall that the \define{Vietoris-Rips complex} at a scale $r\geq 0$ of a metric space $(X,d)$, denoted $\VR_r(X)$, is the simplicial complex whose $k$-simplices are cardinality $k+1$ subsets of $X$ of diameter less than $r$:
\[
\VR_r(X) \coloneqq \{\sigma\subset X|\, |\sigma|<\infty\text{ and }\mathrm{diam}(\sigma)\leq r\}\,.
\]
The boundary of a simplex is the set of its proper subsets.

We now specialize to the posets $\R \times P$ considered in the previous subsection to the case where $P = P_k$, the \define{power set of $[k]$},  considered as a poset whose relations given by inclusions. The reason is that these posets are tied to the structure of labeled datasets, as we now describe.

\begin{definition}\label{def:labeled_vietoris_rips}
    Let $\mathcal{X} = (X,d_X,L_X)$, $L_X = (X_1,\ldots,X_k)$ be a $k$-labeled metric space. For a nonnegative integer $j$, the associated \define{labeled persistent homology} of $\mathcal{X}$, denoted $\mathrm{LPH}_j(\mathcal{X})$, is the $(\R \times P_k)$-module defined by 
    \[
    \mathrm{LPH}_j(\mathcal{X})(r,p) \coloneqq \mathrm{H}_j(\VR_r(\cup_{i\in p} X_i)),
    \]
    where the righthand side denotes degree-$j$ simplicial homology.
\end{definition}

In this subsection, we show that labeled persistent homology is a stable invariant with respect to the interleaving distance and the labeled Gromov-Hausdorff distances,  beginning with the case when labels are registered. Other versions of labeled Gromov-Hausdorff distances are treated in the following subsections.

\begin{theorem}[Gromov-Hausdorff Stability for Registered Labels]\label{thm:stability registered}
    Let $\mathcal X$ and $\mathcal Y$ be $k$-labeled metric spaces and fix a nonnegative integer $j$. The associated labeled persistent homology modules satisfy
    \[
    \mathsf{Int}(\mathrm{LPH}_j(\mathcal{X}),\mathrm{LPH}_j(\mathcal{Y}))\leq 2 \cdot  \mathsf{GH}_k(\mathcal{X},\mathcal{Y}).
    \]
\end{theorem}

This generalizes the famous Gromov-Hausdorff stability result of~\cite{chazal2009gromov}, which is recovered from our result by setting $k=1$ (in which case labeled metric spaces, labeled Gromov-Hausdorff distance, and labeled persistent homology modules all reduce to their classical versions). Our proof technique is based on proofs of similar results, e.g.,~\cite[Proposition 15]{chowdhury2018functorial},~\cite[Theorem 5.10]{memoli2019primer},~\cite[Theorem 5]{needham2024stability}. The only technical difference is the need to more carefully construct certain maps between Vietoris-Rips complexes, using the strategy of the proof of Theorem \ref{thm:k_GH}.

\begin{proof}
Let $\epsilon=2 \cdot \mathsf{GH}_k(\mathcal{X}, \mathcal{Y})$. 
To prove the theorem, we will show the existence of an $\epsilon$-interleaving between $\mathrm{LPH}_j(\mathcal{X})$ and $\mathrm{LPH}_j(\mathcal{Y})$.
We do so by constructing maps $\phi_{r,p}\colon \cup_{i\in p} X_i \to \cup_{i\in p} Y_i$ and $\psi_{r,p}\colon \cup_{i\in p} Y_i \to \cup_{i\in p} X_i$, for each $(r,p) \in \R \times P_k$, fitting in to the following diagrams on Vietoris-Rips complexes: 
\[
\resizebox{15cm}{!}{
\input{tikz/VR_diagrams_1}
}\]
and
\[
\resizebox{15cm}{!}{
\input{tikz/VR_diagrams_2}
},\]
where $\mathrm{inc}$ denotes the corresponding inclusion maps. We then show that these diagrams commute ``up to homotopy'' (to be defined precisely later), which then implies the commutativity of the corresponding diagrams on homology,  completing the proof.

By Lemma~\ref{lem:optimal_correspondences}, there are tuples of maps $(\phi_i\colon X_i \to Y_i)_i$ and $(\psi_i \colon Y_i \to X_i)_i$ that realize $\mathsf{GH}_k(\mathcal{X}, \mathcal{Y})$.
That is, for all $i,j$, we have that $\dis(\phi_i,\phi_j)\leq \epsilon$,  $\dis(\psi_i,\psi_j)\leq \epsilon$, and $\codis(\phi_i,\psi_j)\leq \epsilon$.
Using the strategy of the proof of Theorem \ref{thm:k_GH}, define a map $\phi_p\colon \cup_{i\in p} X_i \to \cup_{i\in p} Y_i$ by $\phi_p|_{X_i\backslash\cup_{j<i}{X_j}}=\phi_i$, and similarly, a map $\psi_p\colon \cup_{i\in p} Y_i \to \cup_{i\in p} X_i$ by $\psi_p|_{Y_i\backslash\cup_{j<i}{Y_j}}=\psi_i$.
By construction,  $\dis(\phi_p)\leq \epsilon$, $\dis(\psi_p)\leq \epsilon$, and $\codis(\phi_p, \psi_p)\leq \epsilon$.
We define the maps $\phi_{r,p}$ and $\psi_{r,p}$ in the diagrams above to be the maps induced by $\phi_p, \psi_p$ on Vietoris-Rips complexes.
These are indeed well defined since for any $x_1, x_2\in \cup_{i\in p} X_i$ we have that $d_{Y}(\phi_p(x_1),\phi_p(x_2))\leq d_{X}(x_1,x_2)+\epsilon$ and, similarly, for $y_1, y_2\in \cup_{i\in p} Y_i$ we have $d_{X}(\psi_p(y_1),\psi_p(y_2))\leq d_{Y}(y_1,y_2)+\epsilon$.

Recall that the maps $f,g\colon K \to L$ between simplicial complexes $K$ and $L$ are said to be \textbf{contiguous} if for every $\sigma\in K$ the set $f(\sigma)\cup g(\sigma)$ is a simplex in $L$. Contiguous maps induce homotopic maps on geometric realizations $|K|$ and $|L|$~\cite[Section 3.5]{spanier2012algebraic} and equal maps on homology~\cite[Theorem 12.5]{munkres2018elements}. Thus, the proof is complete once we show that the following maps are contiguous:
\begin{itemize}
    \item[(1)] $\mathrm{inc}\circ \phi_{r,p}$ and $ \phi_{r',p'}\circ \mathrm{inc}$
    \item[(2)] $\mathrm{inc}\circ \psi_{r,p}$ and $ \psi_{r',p'}\circ \mathrm{inc}$
    \item[(3)]  $\mathrm{inc}$ and $\psi_{r+\epsilon,p}\circ \phi_{r,p}$
    \item[(4)] $\mathrm{inc}$ and $\phi_{r+\epsilon,p}\circ \psi_{r,p}$
\end{itemize}
We will show contiguity for (1) and (3), and (2) and (4) are done analogously.

Consider $\sigma \in \VR_r(\cup_{i\in p} X_i)$ and $x_1,x_2 \in \sigma$, not necessarily distinct. 
To see that the maps in (1) are contiguous, we want to show that $d_Y(\mathrm{inc}\circ \phi_{r,p}(x_1), \phi_{r',p'}\circ \mathrm{inc}(x_2))\leq r'+\epsilon$.
Since $\mathrm{diam}(\sigma) \leq r$, we also have that $d_X(x_1, x_2)\leq r$.
By construction, since $p\preccurlyeq p'$ and since $x_1, x_2 \in \cup_{i\in p} X_i$, we have that $\phi_{p'}(x_i)=\phi_{p}(x_i)$, $i=1,2$.
Since $\dis(\phi_p)\leq \epsilon$, we have that $|d_X(x_1, x_2) - d_Y(\phi_{p}(x_1), \phi_{p'}(x_2))|\leq \epsilon$, which implies
\[d_Y(\mathrm{inc}\circ \phi_{r,p}(x_1), \phi_{r',p'}\circ \mathrm{inc}(x_2))=d_Y(\phi_{p}(x_1), \phi_{p'}(x_2))< d_X(x_1, x_2) +  \epsilon=r+\epsilon \leq r'+\epsilon \,.\]

To see that the maps in (3) are contiguous, we want to show that $d_X(\mathrm{inc}(x_1), \psi_{r+\epsilon,p}\circ \phi_{r,p}(x_2)) \leq r+2\epsilon$.
Since $\codis(\psi_p, \phi_p) \leq \epsilon$, we have that
\begin{align*}
d_X(\mathrm{inc}(x_1), \psi_{r+\epsilon,p}\circ \phi_{r,p}(x_2))&=d_X(x_1, \psi_p\circ \phi_p(x_2)) \\
&\leq \epsilon + d_Y(\phi_p(x_1), \phi_p(x_2)) \\
&\leq  2\epsilon + d_X(x_1, x_2) \leq 2\epsilon + r\,, 
\end{align*}
where the second inequality follows from $\dis(\phi_p)\leq \epsilon$.
\end{proof}

\subsection{Stability for Permuted Labels}

We now work toward generalizing Theorem \ref{thm:stability registered} to the more generalized settings of unregistered labeled metric spaces considered in \S\ref{sec:GH distance for labeled}. First consider the case of permuted $k$-labeled metric spaces and the Gromov-Hausdorff distance modulo relabeling,  $\mathsf{GH}_{S_k}$. Our strategy is to define an associated interleaving distance and show that it is stable with respect to $\mathsf{GH}_{S_k}$. 

The symmetric group $S_k$ acts on $\mathrm{Mod}_{\R \times P_k}$ as follows. First define an action on $P_k$: for $p = \{p_1,\ldots,p_\ell\} \in P_k$ and $\sigma \in S_k$, define
\[
\sigma \cdot p \coloneqq \{\sigma(p_1),\ldots,\sigma(p_\ell)\} \in P_k.
\]
Now, given $V \in \mathrm{Mod}_{\R \times P_k}$ and $\sigma \in S_k$, define $\sigma \cdot V$ by 
\[
\sigma \cdot V (r,p) \coloneqq V(r, \sigma \cdot p).
\]

\begin{proposition}
    The $S_k$-action on $\mathrm{Mod}_{\R \times P_k}$ defined above is by isometries.
\end{proposition}

\begin{proof}
    Given a module morphism $\phi:V \to V'$ and $\sigma \in S_k$, we get an induced module morphism $\sigma \cdot \phi: \sigma \cdot V \to \sigma \cdot V'$ defined by 
    \[
    (\sigma \cdot \phi)_p \coloneqq \phi_{\sigma \cdot p}.
    \]
    One can then check that if morphisms $\phi,\psi$ define an $\epsilon$-interleaving of $V$ and $V'$, then $\sigma \cdot \phi, \sigma \cdot \psi$ define an $\epsilon$-interleaving of $\sigma \cdot V$ and $\sigma \cdot V'$. It follows easily that 
    \[
    \mathsf{Int}(\sigma \cdot V, \sigma \cdot V') = \mathsf{Int}(V,V').
    \]
\end{proof}

It follows from the previous result that the following notion of distance between $(\R \times P_k)$-modules is well-defined.

\begin{definition}[Interleaving Distance for Permuted Classes]
    The \define{interleaving distance modulo relabeling} between $(\R \times P_k)$-modules $V$ and $V'$ is 
    \[
    \mathsf{Int}_{S_k}(V,V') \coloneqq \min_{\sigma \in S_k} \mathsf{Int}(\sigma \cdot V, V').
    \]
\end{definition}

As a corollary of Theorem \ref{thm:stability registered}, we obtain a stability result for labeled persistent homology for permuted classes.

\begin{corollary}[Gromov-Hausdorff Stability for Permuted Classes]\label{cor:stability_unregistered_classes}
    Let $\mathcal{X}$ and $\mathcal{Y}$ be $k$-labeled metric spaces and let $j$ be a nonnegative integer. Then
  \[
    \mathsf{Int}_{S_k}(\mathrm{LPH}_j(\mathcal{X}),\mathrm{LPH}_j(\mathcal{Y}))\leq 2 \cdot  \mathsf{GH}_{S_k}(\mathcal{X},\mathcal{Y}).
    \]
\end{corollary}

\begin{proof}
    This follows easily from the observation that the labeled persistent homology map is equivariant with respect to the $S_k$-actions; that is,
    \[
    \mathrm{LPH}_j(\sigma \cdot \mathcal{X}) = \sigma \cdot \mathrm{LPH}_j(\mathcal{X}).
    \]
    From this, Theorem \ref{thm:stability registered} implies that, for any $\sigma \in S_k$,
    \[
    \mathsf{Int}(\sigma \cdot \mathrm{LPH}_j(\mathcal{X}), \mathrm{LPH}_j(\mathcal{Y})) = \mathsf{Int}( \mathrm{LPH}_j(\sigma \cdot\mathcal{X}), \mathrm{LPH}_j(\mathcal{Y})) \leq 2 \cdot \mathsf{GH}_k(\sigma \cdot \mathcal{X},\mathcal{Y}).
    \]
    The claim follows by minimizing over $S_k$.
\end{proof}

\subsection{Stability for Unregistered Labels}

Finally, we consider the most general situation of labeled metric spaces $\mathcal{X}$ and $\mathcal{Y}$ with unregistered labels, of potentially unequal size. In this setting, the relevant GH variant is the GH distance modulo stabilization, $\mathsf{GH}_{st}$. Following the idea of the previous subsection, we introduce a related variant of interleaving distance.

Suppose that $\mathcal{X}$ is a $k$-labeled metric space. Recall from Definition \ref{def:stabilization} that a $q$-stabilization $\widehat{\mathcal{X}}$ of $\mathcal{X}$ involves a surjective map of sets $\rho:[q] \to [k]$. Such a map induces a poset morphism $\widetilde{\rho}: P_q \to P_k$, with $\widetilde{\rho}(\{p_1,\ldots,p_\ell\}) = \{\rho(p_1),\ldots,\rho(p_\ell)\}$. In turn, given an $(\R \times P_q)$-module $V$, this poset morphism induces an $(\R \times P_k)$-module, denoted $V \circ \widetilde{\rho}$, defined by 
\[
(V \circ \widetilde{\rho})(r,p) = V(r,\widetilde{\rho}(p)).
\]

\begin{definition}
    Let $V$ be an $(\R \times P_k)$-module and let $q$ be an integer with $q \geq k$. A \define{$q$-stabilization of $V$} is an $(\R \times P_q)$-module of the form $V \circ \widetilde{\rho}$, for some surjective map $\rho:[q] \to [k]$. 

    Let $V'$ be an $(\R \times P_\ell)$-module. For $\epsilon \geq 0$, we say that $V$ and $V'$ are \define{stably $\epsilon$-interleaved} if there exist $q$-stabilizations of $V$ and $V'$ which are $\epsilon$-interleaved. The \define{interleaving distance modulo stabilization} between $V$ and $V'$ is
    \[
    \mathsf{Int}_{st}(V,V') \coloneqq \inf\{\epsilon \geq 0 \mid V,V' \mbox{ are stably $\epsilon$-interleaved}\}.
    \]
\end{definition}

We note that we can rephrase the definition of $\mathsf{Int}_{st}$ as
\[
\mathsf{Int}_{st}(V,V')=\inf_{\rho, \rho'} \mathsf{Int}(V\circ \widetilde{\rho}, V'\circ \widetilde{\rho'}),
\]
where the infimum is over maps $\rho:[q] \to [k]$ and $\rho':[q] \to [\ell]$, for all choices of $q$.

We define a category $\lim_{k} \mathrm{Mod}_{\R \times P_k}$ as the directed colimit of the family $(\mathrm{Mod}_{\R \times P_k})_{k \in \mathbb{Z}_{\geq 0}}$. Concretely, each object is an $(\R \times P_k)$-module, for some $k$. A morphism from $V \in \mathrm{Mod}_{\R \times P_k}$ to $V' \in \mathrm{Mod}_{\R \times P_{k'}}$ consists of $q$-stabilizations $V \circ \widetilde{\rho}$ and $V' \circ \widetilde{\rho'}$ and an $(\R \times P_{q})$-module morphism between them.

\begin{proposition}
    The stabilized interleaving distance $\mathsf{Int}_{st}$ defines an extended pseudometric on $\lim_{k} \mathrm{Mod}_{\R \times P_k}$.
\end{proposition}
\begin{proof}
    It is clear from the definition that $\mathsf{Int}_{st}$ is non-negative and symmetric.
    It is left to show the triangle inequality. 
    For that, we will show that if $V \in \mathrm{Mod}_{\R \times P_k}$ and $V' \in \mathrm{Mod}_{\R \times P_{k'}}$ are stably $\epsilon_1$-interleaved and $V'$ and $V'' \in \mathrm{Mod}_{\R \times P_{k''}}$ are stably $\epsilon_2$-interleaved, then $V$ and $V''$ are stably $(\epsilon_1+\epsilon_2)$-interleaved, from which the claim follows.
    
    Suppose $\rho_{1,1}\colon [q_1]\to [k]\,, \rho_{1,2}\colon [q_1]\to [k']$ are such that $V\circ \widetilde{\rho_{1,1}}$ and $V'\circ \widetilde{\rho_{1,2}}$ are $\epsilon_1$-interleaved, and $\rho_{2,1}\colon [q_2]\to [k']\,, \rho_{2,2}\colon [q_2]\to [k'']$ are such that $V'\circ \widetilde{\rho_{2,1}}$ and $V''\circ \widetilde{\rho_{2,2}}$ are $\epsilon_2$-interleaved.
    For $q_3=q_1 \cdot q_2$, we construct maps $\rho\colon [q_3] \to [k]\,, \rho''\colon [q_3] \to [k'']$ as follows.
     First, let $\pi_1\colon [q_3] \to [q_1]$ and $\pi_2\colon [q_3] \to [q_1]$ be defined, respectively, by $\pi_1(i)= (i-1)\, \mathrm{mod}\, q_1 + 1$ and $\pi_2(i)= (i-1)\, \mathrm{mod}\,q_2 + 1$; finally, set $\rho = \rho_{1,1} \circ \pi_1$ and $\rho'' = \rho_{2,2} \circ \pi_2$.
    Since $V\circ \widetilde{\rho_{1,1}}$ and $V'\circ \widetilde{\rho_{1,2}}$ are $\epsilon_1$-interleaved, we have that $V\circ \widetilde{\rho}$ and $V'\circ \widetilde{\rho_{1,2}\circ \pi_1}$ are also $\epsilon_1$-interleaved.
    Similarly, we have that $V'\circ \widetilde{\rho_{2,1}\circ \pi_2}$ and $V''\circ \widetilde{\rho''}$ are  $\epsilon_2$-interleaved.
It follows from the triangle inequality of the (usual, not stabilized) interleaving distance  (see~\cite[Prop.~3.11]{BubenikdeSilvaScott2015}) that $V\circ \widetilde{\rho}$ and $V''\circ \widetilde{\rho''}$ are $(\epsilon_1+\epsilon_2)$-interleaved. 
\end{proof}

We conclude with a final corollary of Theorem \ref{thm:stability registered}, which applies in the setting of GH distance modulo stabilizations.

\begin{corollary}[Gromov-Hausdorff Stability Modulo Stabilizations]\label{cor:stability_modulo_stabilizations}
    Let $\mathcal{X}$ and $\mathcal{Y}$ be  labeled metric spaces, with unregistered labels of potentially different sizes, and let $j$ be a nonnegative integer. Then
  \[
    \mathsf{Int}_{st}(\mathrm{LPH}_j(\mathcal{X}),\mathrm{LPH}_j(\mathcal{Y}))\leq 2 \cdot  \mathsf{GH}_{st}(\mathcal{X},\mathcal{Y}).
    \]
\end{corollary}

\begin{proof}
    Let $\mathcal{X}$ be a $k$-labeled metric space and $\mathcal{Y}$ an $\ell$-labeled metric space.
    Similar to the proof of Corollary \ref{cor:stability_unregistered_classes}, this boils down to an equivariance property: if $\widehat{\mathcal{X}}$ is a $q$-stabilization, with associated surjective map $\rho:[q] \to [k]$, then 
    \[
    \mathrm{LPH}_j(\mathcal{X}) \circ \widetilde{\rho} = \mathrm{LPH}_j(\widehat{\mathcal{X}}).
    \]
    The result then follows from the definitions of $\mathsf{Int}_{st}$ and $\mathsf{GH}_{st}$, and from Theorem~\ref{thm:stability registered}. Indeed, for stabilizations $\widehat{\mathcal{X}}$, with map $\rho\colon [q] \to [k]$, and $\widehat{\mathcal{Y}}$, with map $\rho'\colon [q] \to [\ell]$, we have 
    \[
    \mathsf{Int}(\mathrm{LPH}_j(\mathcal{X}) \circ \widetilde{\rho},\mathrm{LPH}_j(\mathcal{Y}) \circ \widetilde{\rho}') = \mathsf{Int}(\mathrm{LPH}_j(\widehat{\mathcal{X}}),\mathrm{LPH}_j(\widehat{\mathcal{Y}}) )  \leq 2 \cdot  \mathsf{GH}_{q}(\widehat{\mathcal{X}}, \widehat{\mathcal{Y}}).
    \]
    The corollary follows by infimizing over stabilizations.
\end{proof}

\subsection{Comparison to Prior Work}

We mention here two related results. Each estimates the bottleneck distance between single-parameter persistence modules also associated to an inclusion or mapping structure.

In \cite{CSEHM}, the authors prove a stability result for image persistence that bounds the difference in the bottleneck distance of the image, kernel, and cokernel persistence modules of an inclusion map for two different sub-level set filtrations on the same underlying topological space $X$; the upper bound is the sup-norm distance between the filtrations. This result could potentially be used to derive stability for mixup barcodes when comparing different filtrations on the same metric space.

The recent work of Draganov et al.~\cite{draganov2025gromov} proves a more general stability result for six different persistence modules derived from a map between two chromatic metric pairs (image, kernel, and cokernel, as in \cite{CSEHM}, as well as three others); see \S\ref{ss:chromatic}. The authors bound the bottleneck distance between the modules---which can now differ in underlying metric space in addition to filtration---using their $\Sigma$-constrained Gromov-Hausdorff distance. Therefore, although our Gromov-Hausdorff distance defined in \S\ref{sec:GH distance for labeled} is a generalization of the $\Sigma$-constrained Gromov-Hausdorff distance, in the sense of Proposition \ref{prop:chromatic}, our stability result is not a generalization of theirs. Their result applies to more general maps, more general subspaces and quotient spaces derived from those maps, and more general filtrations than ours does.

Note that in Example \ref{ex:mixupL} we explain how to recover the landscape of the image persistence module for the inclusion map using our methods, so the above results for the image persistence module (and using the Vietoris-Rips filtration) are a possible site of overlap. However, while the bottleneck distance equals the interleaving distance for single-parameter persistence, our interleaving distance compares $(\R\times P)$-modules rather than $\R$-modules, so it is not immediately obvious how to use the result of \ref{ss:chromatic} to prove ours, or vice versa.

\section{Landscape Representations of Labeled Datasets}\label{sec:landscapes}

This section introduces a more computationally-friendly approach to working with the $(\R \times P_k)$-modules that we studied theoretically in the previous subsection. These are based on the \textit{persistence landscape} construction of Bubenik~\cite{Bubenik15}. We treat theoretical aspects of our generalization in this section. 

\subsection{Generalized Persistence Landscapes}

To use persistent homology techniques in practice on real data, one must extract computationally convenient statistics from the somewhat unwieldy module structures. First consider the setting of \textit{single parameter persistence}, where the poset $P$ in Definition \ref{def:P-module} is $\R$ (with the usual ordering). Under very mild conditions, a single parameter persistence module $V:\R \to \mathrm{Vect}_\mathbb{F}$ is  characterized up to isomorphism by an invariant called a \textit{barcode} or \textit{persistence diagram}~\cite{crawley2015decomposition}. Indeed, these representations of one-parameter modules---which essentially encode the summands of a module when it is expressed as a direct sum of indecomposable modules---are perhaps the most familiar objects in topological data analysis. Unfortunately, for more general choices of posets $P$, complete invariants of this form for $P$-modules do not exist~\cite{gabriel1972unzerlegbare,carlsson2007theory}. 

On the other hand, a certain representation of a one-parameter persistence module, called a \textit{persistence landscape}~\cite{Bubenik15}, does generalize to provide an incomplete, but nonetheless useful, invariant of a general module. We now review the classical construction. Let $V:\R\to\mathrm{Vect}_\mathbb{F}$ be a single-parameter persistence module.  
Given $a\leq b\in\R$, the \textbf{Betti number} $\beta_V^{a,b}$ is the rank
\[
\beta_V^{a,b}\coloneqq\rank\big(V(a\leq b)\big).
\]
This leads to the following definition, first introduced in~\cite{Bubenik15}.

\begin{definition}[{\cite{Bubenik15}}]\label{def:1land}
The \textbf{persistence landscape} of a persistence module $V:\R\to\Vect_\mathbb{F}$ is the function $\lambda:\mathbb{Z}_{\geq 0}\times\R\to[0,+\infty]$, given by
\[
\lambda_V(n,r)\coloneqq\sup\{\epsilon\geq0\,|\,\beta_V^{r-\epsilon,r+\epsilon}\geq n\},
\]
where $\sup \emptyset \coloneqq 0$.
\end{definition}

The persistence landscape encodes the full rank information of the module, which is known to characterize the module up to isomorphism~\cite[Theorem 5]{carlsson2007theory}. This particular encoding embeds $\mathrm{Mod}_\R$ (or, rather, the subclass of modules with finite support) into a Banach space, and this embedding is Lipschitz stable with respect to interleaving distance~\cite[Theorem 13]{Bubenik15}.

In~\cite{Vipond20}, Vipond generalized the persistence landscape construction to the setting of $\R^d$-modules, where we consider $\R^d$ as a poset with the product order: $(x_1,\ldots,x_d) \preccurlyeq (y_1,\ldots,y_d)$ if and only if $x_i \leq y_i$ for all $i$. This generalized invariant is also shown to be Lipschitz stable~\cite[Theorem 30]{Vipond20}. We recall the definition below.

\begin{definition}[{\cite{Vipond20}}]\label{def:multiparameter_landscape}
The \textbf{multiparameter persistence landscape} of a persistence module $V:\R^d\to\Vect_\mathbb{F}$ is the function $\lambda:\mathbb{Z}_{\geq 0}\times\R^d\to[0,+\infty]$, given by
\[
\lambda_V(n,\mathbf{x})\coloneqq\sup\{\epsilon\geq0\,|\,\beta_V^{\mathbf{x}-\mathbf{h},\mathbf{x}+\mathbf{h}}\geq n \mbox{ for all $\mathbf{h}$ such that  $\mathbf{0} \preccurlyeq \mathbf{h}$ and $\|\mathbf{h}\|_\infty \leq \epsilon$}\}.
\]
Here, we use $\mathbf{x}$, $\mathbf{0}$ and $\mathbf{h}$ to denote elements of $\R^d$,  $\|\cdot\|_\infty$ denotes the standard $\ell_\infty$-norm on $\R^d$, and 
\[
\beta_V^{\mathbf{x}-\mathbf{h},\mathbf{x}+\mathbf{h}} \coloneqq \rank\big(V(\mathbf{x}-\mathbf{h} \preccurlyeq \mathbf{x} + \mathbf{h})\big).
\]
\end{definition}

From Vipond's forumulation, it becomes straightforward to generalize the landscape construction to persistence modules over more general choices of poset $P$. To this end, fix a poset $(P,\preccurlyeq)$. Let $V:P\to\mathrm{Vect}_\mathbb{F}$ be a $P$-module. For $a,b\in P$ with $a \preccurlyeq b$, we define the associated \define{Betti number}
\[
\beta_V^{a,b}\coloneqq\rank(V(a\preccurlyeq b)).
\]
To give a definition of a generalized landscape, we need a certain notion of distance on $P$. 

\begin{definition}
    An \textbf{extended quasimetric} on a set $X$ is a function $d\colon X\times X\to[0,+\infty]$ satisfying the metric  axioms:
    \begin{enumerate}[label=(\arabic*)]
        \item $d(x,y)\geq0$, and $d(x,y)=0$ if and only if $x=y$;
        \item $d(x,z)\leq d(x,y)+d(y,z)$.
    \end{enumerate}
\end{definition}

We provide relevant examples of extended quasimetrics on posets in the following subsection. Assuming $P$ has been endowed with such a structure, we arrive at our main definition.

\begin{definition}\label{def:2} The \textbf{$P$-landscape} or \textbf{generalized persistence landscape} of a $P$-module $V\colon P \to \mathrm{Vect}_\mathbb{F}$, with respect to an extended quasimetric $d$ on $P$, is the function $\lambda_V:\mathbb{Z}_{\geq 0} \times P\to[0,+\infty]$, given by
    \[
    \lambda_V(n,x)\coloneqq\sup\{\epsilon\geq0\;|\;\beta_V^{a,b}\geq n \mbox{ for all $a,b \in P$ satisfying } a\preccurlyeq x\preccurlyeq b, \; d(a,x) \leq\epsilon, \; d(x,b)\leq\epsilon\}.
    \]
\end{definition}

\begin{remark}
    When the underlying poset needs to be emphasized, we will use the notation $\lambda_V^P$, instead of $\lambda_V$. To avoid overloading notation, we do not indicate the dependency on the choice of extended quasimetric. In the following, this should always be clear from context.
\end{remark}

\begin{remark}\label{rmk:multiparameter_landscape}
    It is implicit in the name of this construction that this should generalize the multiparameter persistence landscape described above. This is not immediately clear from the definition, but we prove this implicit claim below in Proposition \ref{prop:mp=gp}. 
\end{remark}

\subsection{Examples of \texorpdfstring{$\R \times P$}{R x P}-Modules}

We now focus on modules of the form $\R \times P$, in light of our focus on these modules in \S\ref{sec:persistent_homology}. Fix a poset $(P,\preccurlyeq)$ (e.g., the poset $P_k$ consisting of the power set of $[k]$ with its inclusion order). As above, we abuse notation and use $\preccurlyeq$ for the product partial order on $\R \times P$. Below, we denote elements of $\R \times P$ using vector notation as $\mathbf{x} = (r,x)$, with $r \in \R$ and $x \in P$. We now provide examples of relevant extended quasimetrics on the poset $\R \times P$. The first example assumes the existence of a quasimetric on $P$; later we give specific examples of such structures.

\begin{example}[Product Quasimetrics]\label{ex:product_quasimetrics}
    Given an extended quasimetric $d_P$ on $P$, we can use standard metric product constructions to extend it to $\R \times P$. For example, the \define{sum (extended) quasimetric} is defined on $\bx = (r,x)$ and $\by = (s,y)$ as 
    \[
        (\bx, \by) \mapsto d_P(x, y) + |r - s|.
    \]
    Similarly, the \define{max (extended) quasimetric} is 
    \[
    (\bx, \by) \mapsto \max\{d_P(x, y),|r - s|\}.
    \]
\end{example}

We now show that the generalized persistence landscape really does generalize the multiparameter persistence landscape (see Remark \ref{rmk:multiparameter_landscape}). Let $V: \R^d \to \mathrm{Vect}_\mathbb{F}$. We use $\lambda_V$ for its multiparameter persistence landscape (Definition \ref{def:multiparameter_landscape}). On the other hand, setting $P = \R^{d-1}$, we can consider $V$ as an $(\R \times P)$-module. Taking the metric $d_P$ on $P$ to be the $\ell_\infty$-distance (i.e., sup-norm distance), we endow $\R \times P$ with the max quasimetric (i.e., sup-norm distance on $\R^d \approx \R \times \R^{d-1}$). We denote the associated generalized persistence landscape as $\lambda_V^{\R \times P}$. This yields the following equivalence.

\begin{prop}\label{prop:mp=gp}
    The multiparameter persistence landscape $\lambda_V$ and generalized persistence landscape $\lambda_V^{\R \times P}$ are equal.
\end{prop}

\begin{proof}
    First we argue that $\lambda_V^{\R \times P} \geq \lambda_V$. Fix $n$ and $\bx$, and suppose $\lambda_V(n, \bx) = \epsilon$. If $\epsilon = 0$, then, since $\lambda_V^{\R \times P}(n, \bx) \geq 0$, then the claim is trivially true. Suppose $\epsilon > 0$. Choose arbitrary $\ba, \bb$ such that $\ba \preccurlyeq \bx \preccurlyeq \bb$, $\|\bx - \ba\|_\infty < \epsilon$, and $\|\bb - \bx\|_\infty < \epsilon$. Let $h = \max (\|\bx - \ba\|_{\infty}, \|\bx -\bb\|_{\infty})$ and set $\bh \coloneqq (h,\dots, h) \in \R^d$. By the definition of $\bh$, we have that $\bx -\bh \preccurlyeq \ba$ and $\bb \preccurlyeq \bx + \bh$. Then, by the definition of a persistence module on a poset, we can factor $V(\bx -\bh\preccurlyeq\bx +\bh)$ as
    \[
    V(\bx -\bh) \rightarrow V(\ba) \rightarrow V(\bb) \rightarrow V(\bx +\bh).
    \]
    Because $\|h\|_\infty < \epsilon = \lambda_V(n, \bx)$, we know that the rank of $V(\bx-\bh \preccurlyeq \bx+\bh)$ is at least $n$, and therefore that the rank of $V(\ba \preccurlyeq \bb)$ as also at least $n$. This shows that $\lambda_V^{\R \times P}(n, \bx) \ge \lambda_V(n, \bx)$, and because $n$ and $\bx$ were arbitrary this also shows that $\lambda_V^{\R \times P} \ge \lambda_V$.

    Next, we argue that $\lambda_V^{\R \times P} \le \lambda_V$. As above, we fix $n$ and $\bx$, but now suppose $\lambda_V^{\R \times P}(n, \bx) = \epsilon$. We similarly can assume that $\epsilon > 0$. Consider any $\bh$ with $\bzero \preccurlyeq \bh$ and $\|\bh\|_{\infty}<\epsilon$. Then set $\ba\coloneqq \bx-\bh \preccurlyeq \bx \preccurlyeq \mathbf{b} \coloneqq \bx+\bh$, and observe that $\|\bx - \ba\|_{\infty} < \epsilon$ and  $\|\bx - \bb\|_{\infty} < \epsilon$. Thus, because $\ba$ and $\bb$ fit the definition of $\lambda_V^{\R \times P}$, the rank of $V(\bx - \bh \preccurlyeq \bx + \bh) = V(\ba\preccurlyeq\bb)$ is at least $n$. This shows that $\lambda_V(n, \bx) \ge \lambda_V^{\R \times P}(n, \bx)$, and again because $n$ and $\bx$ were arbitrary, we have that $\lambda_V \ge \lambda_V^{\R\times P}$ in general.
\end{proof}

Let us now describe some extended quasimetrics which are relevant to our setting. In particular, we are interested in the case where the poset $P$ is finite. 

\begin{example}[Geodesic and Ultrametric Distances]\label{ex:gud}
    Let $(P,\preccurlyeq)$ be a finite poset. We consider the \textit{Hasse diagram}~\cite{birkhoff1940lattice} of $P$ as a directed graph: the nodes of the graph are the elements of $P$ and there is a directed edge from $x$ to $y$ if $x \preccurlyeq y$ and there is no other point $z$, distinct from $x$ and $y$, such that $x \preccurlyeq z \preccurlyeq y$. A \define{weighting} of $P$ is a choice of nonnegative weight $w(e)$ for each directed edge $e$ in the Hasse diagram. For $x,y \in P$ with $x \preccurlyeq y$, a \define{path} from $x$ to $y$ is a sequence $f = (e_1,\ldots,e_\ell)$ of directed edges in the Hasse diagram such that the source of $e_1$ is $x$ and the target of $e_\ell$ is $y$; we use $\mathcal{P}_{x,y}$ to denote the set of all paths from $x$ to $y$. The set of paths may be empty even if $x \preccurlyeq y$ and, moreover, in the case that $x \preccurlyeq y$ does not hold, we set $\mathcal{P}_{x,y} = \emptyset$. 
    
    The weightings described above lead to two useful extended quasimetrics on $P$. First, we define the \define{geodesic distance} on $P$ to be the extended quasimetric 
    \[
    (x,y) \mapsto \inf_{f \in \mathcal{P}_{x,y}} \sum_{e \in f} w(e)
    \]
    (maintaining the convention that $\inf \emptyset = +\infty$). Similarly, we define the \define{ultrametric distance} on $P$ to be the extended quasimetric 
    \[
    (x,y) \mapsto \inf_{f \in \mathcal{P}_{x,y}} \max\{w(e) \mid e \in f\}.
    \]

    Given a choice of weighting, one obtains an extended quasimetric on $\R \times P$ by applying the product constructions of Example \ref{ex:product_quasimetrics} to the associated geodesic or ultrametric distance. Distances of this form will be used in theory and numerical experiments below.

    The choice of a weighting scheme is generally data-driven. Some possible choices are:
    \begin{enumerate}
        \item (Constant Weights) The weight for each directed edge in the Hasse diagram is chosen to be some constant $w$, which can be tuned in a supervised way for a given task.
        \item (Point Clouds) Suppose that $\mathcal{X} = (X,d_X,L_X)$ is a finite $k$-labeled metric space and consider the poset $P_k$. An edge in the Hasse diagram is given by an inclusion $Q \hookrightarrow Q'$ of subsets of $[k]$, and one can define the weight of this edge to be $\mathrm{diam}_{Q'}(\mathcal{X})$ (Definition \ref{def:q_diameter}). Such a choice is theoretically justified by Corollary \ref{cor:4GH}.
    \end{enumerate}
\end{example}

\begin{example}[$P$-landscapes and mixup barcodes]\label{ex:mixupL}
 The $P$-landscape generalizes the notion of image sub-bars from \cite{mixup}; see Example \ref{ex:mixup}. Given point clouds $X_1$ and $X_2$, we form the \textbf{image persistence module} by taking the persistence module whose vector spaces are images of the maps $\mathrm{inc}_*$ induced on homology of the Vietoris-Rips complex by the inclusion $X_1\subseteq X$ and whose linear maps are the restrictions of the linear maps for the $X$ module to the images of the $X_1$ module.\footnote{In the language of Example \ref{ex:mixup}, $\VR_i(X_1)=L_i$ and $\VR_i(X)=K_i$; the Vietoris-Rips filtraton is always finite, with maximum value given by the diameter of the set, and can be discretized by restricting to critical filtration values, i.e., those at which the Betti numbers change.} 
 Let $\lambda_\mathrm{im}$ denote the usual single-parameter landscape derived from the image persistence module. As noted in \cite[Observation~1]{mixup}, the image sub-bars equal the barcode of the image persistence module, which is itself equivalent to $\lambda_\mathrm{im}$. Therefore, in order to obtain the mixup barcode of $X_1\subseteq X=X_1\cup X_2$, it is enough to calculate $\lambda_\mathrm{im}$ along with the barcode and/or single parameter persistence landscape of $X_1$.

 We may compute $\lambda_\mathrm{im}$ from the $\R\times P$-landscape of $P=X_1\to X$ with $d_P(X_1,X)=0$. In this case we use the simplified notation
 \[
 V(r,S)=\mathrm{LPH}_j(r,S)=\begin{cases}
     \mathrm{LPH}_j(\mathcal{X})(r,\{1\})&\text{ if }S=X_1,
     \\\mathrm{LPH}_j(\mathcal{X})(r,\{1,2\})&\text{ if }S=X.
 \end{cases}
 \]
 To show that
     \begin{equation}
     \lambda_\mathrm{im}(n,r)=\lambda_V(n,(r,X_1)),
     \end{equation}
     it is enough to show that for all $n, r$, and $\epsilon$, the ranks of
     \[
     g:\mathrm{inc}_*(\mathrm{LPH}_j(r-\epsilon,X_1))\to\mathrm{inc}_*(\mathrm{LPH}_j(r+\epsilon,X_1))
     \]
     and
     \[
     h:\mathrm{LPH}_j(r-\epsilon,X_1)\to \mathrm{LPH}_j(r+\epsilon,X)
     \]
     are equal. (Here we have used the fact that $d_P(X_1,X)=0$ to compute the $\epsilon$-ball about $(r,X_1)$ in $\R\times P$.) The latter 
     follows from the fact that $h$ factors as $g\circ\mathrm{inc}_*$.

Note that when $(j,n)=(0,1)$, the above discussion may not hold, depending on the convention chosen. For any point cloud and all large enough $r$, the vector space $\mathrm{LPH}_0(r,S)$ is one-dimensional, and yet we want $\lambda$ to only take finite values, requiring us to choose an arbitrary $r$ past which we declare $V(r,\cdot)\equiv0$. Our choice in \S\ref{s:experiments} is to set $\mathrm{LPH}_0(r,S)=0$ for all $r$ greater than or equal to the diameter of $S$. In that case, our argument above applies so long as $r+\epsilon$ is strictly less than the diameter of $X_1$. Alternately, one could consider reduced homology, which would effectively increase the rank requirement from $\beta_V^{a,b}\geq n$ to $\beta_V^{a,b}\geq n+1$ in Definition \ref{def:2}, thus removing the information of $\mathrm{LPH}_0$.

We may compute the single-parameter persistence landscape of $X_1$ by setting all distances in $d_P$ to be greater than half the diameter of $X_1$ (see Lemma \ref{lem:dPrange}) and taking the restriction of the generalized persistence landscape to $X_1$.

Now, because the generalized peristence landscape is stable with respect to the interleaving and therefore Gromov-Hausdorff distances (see \S\ref{s:landstab} below), we have shown that, when applied to the Vietoris-Rips filtration, mixup barcodes are also stable with respect to these distances.
 \end{example}

\subsection{Stability of $(\R \times P)$-Landscapes}\label{s:landstab}

This subsection establishes a stability result for our generalized landscapes. Before stating the theorem, we set up notation.

Fix a poset $(P,\preccurlyeq)$ and consider the product poset $\R \times P$. Let $V$ and $V'$ be $(\R \times P)$-modules. Suppose that $P$ has been endowed with extended quasimetrics $d_P$ and $d_P'$---the point here is that we allow the choice of metric on $P$ to be dependent on the modules. Let $d$ and $d'$ both be either the sum quasimetrics or the max quasimetrics on $\R \times P$, in the sense of Example \ref{ex:product_quasimetrics}. Let $\lambda_V$ and $\lambda_{V'}$ denote, respectively, the persistence landscapes associated to these choices. Below, we use $\|\cdot\|_\infty$ to generically denote the sup-norm on any relevant function space (here, the norm is allowed to take the value $+\infty$). 

\begin{theorem}\label{thm:stability_of_landscapes}
    With the notation above, the $(\R \times P)$-landscapes satisfy
    \[
    \|\lambda_V - \lambda_{V'}\|_\infty \leq \mathsf{Int}(V,V') + \|d_P - d_P'\|_\infty.
    \]
\end{theorem}

The proof generally follows that of~\cite[Theorem 30]{Vipond20}, with some additional work required to account for the differences in the poset metrics.

\begin{proof}
    Suppose that $V$ and $V'$ are $\epsilon$-interleaved via $(\R \times P)$-module morphisms $\phi:V \to V'_\epsilon$ and $\psi:V' \to V_\epsilon$. Fix $(n,r,x) \in \mathbb{Z}_{\geq 0} \times \R \times P$. We want to show that 
    \begin{equation}\label{eqn:landscape_stability}
    |\lambda_V(n,(r,x)) - \lambda_{V'}(n,(r,x))| \leq \epsilon + \|d_P - d_P'\|_\infty.
    \end{equation}
    For simplicity, we assume, without loss of generality, that 
    \[
    \lambda_V(n,(r,x)) \geq  \lambda_{V'}(n,(r,x)).
    \]
    Moreover, we may assume that 
    \[
    \lambda_V(n,(r,x)) \geq \epsilon + \|d_P - d_P'\|_\infty,
    \]
    as the desired claim follows immediately otherwise. 

    Let $(s,a), (t,b) \in \R \times P$ be arbitrary elements satisfying 
    \[
    (s,a) \preccurlyeq (r,x) \preccurlyeq (t,b)
    \]
    and
\begin{equation}\label{eqn:landscape_stability_condition}
    d'((s,a),(r,x)), d'((r,x),(t,b)) \leq \lambda_V(n,(r,x)) - \epsilon - \|d_P - d_P'\|_\infty.
    \end{equation}
    Then the desired inequality \eqref{eqn:landscape_stability} follows if we show that
    \begin{equation}\label{eqn:betti_number}
    \beta_{V'}^{(s,a),(t,b)} \geq n.
    \end{equation}
    To this end, consider the diagram
\[\begin{tikzcd}
	{V(s-\epsilon,a)} && {V(s+\epsilon,a)} && {V(t+\epsilon,b)} \\
	& {V'(s,a)} && {V'(t,b)}
	\arrow[from=1-1, to=1-3]
	\arrow["{\phi_{s-\epsilon,a}}"', from=1-1, to=2-2]
	\arrow[from=1-3, to=1-5]
	\arrow["{\psi_{s,a}}"', from=2-2, to=1-3]
	\arrow[from=2-2, to=2-4]
	\arrow["{\psi_{t,b}}"', from=2-4, to=1-5]
\end{tikzcd}\]
The diagram commutes, by the interleaving assumption. Observe that 
\begin{itemize}
    \item $(s-\epsilon,a) \preccurlyeq (r,x) \preccurlyeq (t + \epsilon, b)$, and 
    \item $d((s-\epsilon,a),(r,x)), \, d((r,x),(t+\epsilon,b)) \leq \lambda_V(n,(r,x))$.
\end{itemize}
The first point is clear, but the second requires some justification (we only prove one of the inequalities here, with the other being analogous):
\begin{align}
d((s-\epsilon,a),(r,x)) &\leq d((s-\epsilon,a),(s,a)) + d((s,a),(r,x)) \label{eqn:stability_proof_1} \\
&\leq \epsilon + d'((s,a),(r,x)) + \|d_P - d_P'\|_\infty \label{eqn:stability_proof_2} \\
&\leq \epsilon + \big(\lambda_V(n,(r,x)) - \epsilon - \|d_P - d_P'\|_\infty \big)  + \|d_P - d_P'\|_\infty \label{eqn:stability_proof_3} \\
&= \lambda_V(n,(r,x)). \nonumber
\end{align}
The inequalities in this chain are justified as follows:
\begin{itemize}
    \item \eqref{eqn:stability_proof_1} is the triangle inequality for the quasimetric.
    \item \eqref{eqn:stability_proof_2} first uses the fact that $d((s-\epsilon,a),(s,a)) \leq \epsilon$, by construction of the sum/max quasimetrics. Secondly, it uses the fact that 
    \begin{equation}\label{eqn:norm_bound}
    d((s,a),(r,x)) \leq  d'((s,a),(r,x)) + \|d_P - d_P'\|_\infty.
    \end{equation}
    To establish this, we need to consider the two quasimetric constructions (sum versus max) separately. If $d$ and $d'$ are sum quasimetrics, then we easily deduce that  
    \begin{align*}
    d((s,a),(r,x)) - d'((s,a),(r,x)) &= |s-r| + d_P(a,x) - |s-r| - d_P'(a,x) \\
    &= d_P(a,x) - d_P'(a,x) \leq \|d_P - d_P'\|_\infty,
    \end{align*}
    from which the inequality \eqref{eqn:norm_bound} follows. On the other hand, if $d$ and $d'$ are max quasimetrics, then we have 
    \begin{align*}
    d((s,a),(r,x)) - d'((s,a),(r,x)) &= \max\{|r-s|,d_P(a,x)\} - \max\{|r-s|,d_P'(a,x)\} \\
    &= \left\{
    \begin{array}{rl}
    0 & \mbox{if $|r-s| \geq d_P(a,x),\, d_P'(a,x)$} \\
    d_P(a,x) - d_P'(a,x) & \mbox{if $|r-s| \leq d_P(a,x),\, d_P'(a,x)$} \\
    |r-s| - d_P'(a,x) & \mbox{if $d_P(a,x) \leq |r-s| \leq d_P'(a,x)$} \\ 
    d_P(a,x) - |r-s| & \mbox{if $d_P(a,x) \geq |r-s| \geq d_P'(a,x)$}
    \end{array}\right\}\\
    &\leq \left\{
    \begin{array}{rl}
    0 & \mbox{if $|r-s| \geq d_P(a,x),\, d_P'(a,x)$} \\
    d_P(a,x) - d_P'(a,x) & \mbox{if $|r-s| \leq d_P(a,x),\, d_P'(a,x)$} \\
    0 & \mbox{if $d_P(a,x) \leq |r-s| \leq d_P'(a,x)$} \\ 
    d_P(a,x) - d_P'(a,x) & \mbox{if $d_P(a,x) \geq |r-s| \geq d_P'(a,x)$}
    \end{array}\right\}\\
    &\leq \|d_P-d_P'\|_\infty,
    \end{align*}
    so that \eqref{eqn:norm_bound} once again follows.
    \item \eqref{eqn:stability_proof_3} applies the assumption \eqref{eqn:landscape_stability_condition}.
\end{itemize}
From these points, it follows that the map across the top of our commutative diagram is at least rank $n$. This is only possible if the rank along the bottom of the diagram is also at least $n$, i.e., that \eqref{eqn:betti_number} holds. This completes the proof.
\end{proof}

We now give some immediate corollaries of Theorem \ref{thm:stability_of_landscapes}. The following statements use the notation set at the beginning of this subsection.

\begin{corollary}
    Suppose that $d_P = d_P'$. Then 
    \[
    \|\lambda_V - \lambda_{V'}\|_\infty \leq \mathsf{Int}(V,V').
    \]
    If the modules $V$ and $V'$ arise as labeled persistent homology modules of labeled metric spaces $\mathcal{X}$ and $\mathcal{X}'$, with registered labels, then
    \[
    \|\lambda_V - \lambda_{V'}\|_\infty \leq 2 \cdot \mathsf{GH}_k(\mathcal{X},\mathcal{X}').
    \]
\end{corollary}

\begin{proof}
    This is a direct application of Theorems \ref{thm:stability_of_landscapes} and \ref{thm:stability registered}.
\end{proof}

Recall that, taking $P = \R^{n-1}$ endowed with $\ell_\infty$-distance, and taking the max quasimetric on $\R \times P$, our generalized persistence landscape $\lambda_V^{\R \times P}$ is equivalent to Vipond's multiparameter persistence landscape $\lambda_V$ (see Proposition \ref{prop:mp=gp}). In this setting, the following is immediate.

\begin{corollary}
    For $\R^n$-modules $V$ and $V'$, their multiparameter persistence landscapes satisfy 
    \[
    \|\lambda_V - \lambda_{V'}\|_\infty \leq \mathsf{Int}(V,V').
    \]
\end{corollary}

\begin{remark}\label{rem:interleaving)_comparisons}
    The previous corollary does not recover Vipond's stability result \cite[Theorem 30]{Vipond20}, as the interleaving distances appearing in each result are not the same. In \cite{Vipond20}, the standard flow structure on $\R^n$ is used to define interleaving distance: namely, the flow is generated by shifting in the direction of the \textit{all ones vector} $(1,1,\ldots,1) \in \R^n$. On the other hand, our interleaving distance is generated by shifts in the direction of the first standard basis vector $(1,0,\ldots,0) \in \R^n$, which is a nonstandard (and somewhat degenerate) choice.
\end{remark}

One can also choose the metrics $d_P$ and $d_P'$ in a data-driven way, while maintaining Gromov-Hausdorff stability. Suppose that $V$ and $V'$ are labeled persistent homology modules for $k$-labeled metric spaces $\mathcal{X}$ and $\mathcal{X}'$, with registered labels. We define the distances $d_P$ and $d_{P}'$ as follows. Each weighted edge in the Hasse diagram for $P$ is given by an inclusion of one subset of $Q \subset [k]$ into a larger one $Q' \subset [k]$; we assign the weight of this directed edge to be the $Q'$-diameter of this subset of labels, in the sense of Definition \ref{def:q_diameter}, and let $d_P$ denote the induced ultrametric distance (see Example \ref{ex:gud}), with $d_{P}'$ defined similarly. 

\begin{corollary}\label{cor:4GH}
    With the setup described in the previous paragraph, 
    \[
    \|\lambda_V - \lambda_{V'} \|_\infty \leq 4 \cdot \mathsf{GH}_k(\mathcal{X},\mathcal{X}').
    \]
\end{corollary}

\begin{proof}
    By the construction of $d_P$ and $d_P'$, and by Proposition \ref{prop:GH_k_estimate}, we have 
    \[
    \|d_P - d_P'\|_\infty \leq \max_{{Q} \subset [k]} |\mathrm{diam}_{Q}(\mathcal{X}) - \mathrm{diam}_{Q}(\mathcal{X}') | \leq 2 \cdot \mathsf{GH}_k(\mathcal{X},\mathcal{X}').
    \]
    The result then follows by Theorems \ref{thm:stability registered} and \ref{thm:stability_of_landscapes}.
\end{proof}

\begin{remark}
    The corollaries in this subsection illustrate some choices of metrics $d_P$ which are principled in the sense that they induce interpretable stability bounds. In practice, we also employ choices of $d_P$ which give better empirical results, even if they are not as theoretically sound as those considered in the corollaries; in any case, Theorem \ref{thm:stability_of_landscapes} applies to these alternative choices. 
\end{remark}

\subsection{Basic Properties of $(\R \times P)$-Landscapes}\label{s:theory}

Next we turn to proving theoretical properties of $(\R\times P)$-landscapes, which are required for calculations (both by hand and implicitly in our code). First we address the general properties of $(\R\times P)$-landscapes.

\begin{lemma}\label{lem:properties}
    The generalized persistence landscape of an $(\R\times P)$-module $V$ has the following properties:
    \begin{enumerate}
        \item $\lambda_V(n,\bx)\geq0$.
        \item $\lambda_V(n,\bx)\geq\lambda_V(n+1,\bx)$.
        \item $\lambda_V(n,\bx) = \lambda_V (n, (r, x))$ is 1-Lipschitz in $r$ when $d$ is the sum quasimetric and $d_P$ is any extended quasimetric. 
    \end{enumerate}
\end{lemma}
\begin{proof}
    We note that the first two properties follow immediately from the definition (for all $P$-landscapes, not just for $(\R\times P)$-landscapes). To show that $\lambda_V(n, (r, x))$ is 1-Lipschitz in $r$, we begin by choosing $\bx = (r_x, z), \by = (r_y, z)$ (so $\bx$ and $\by$ differ only in their $\R$-coordinate), and $n$, and wish to show that 
    \begin{equation}
    \label{eq:landscape_lipschitz_condition}
        |\lambda_V(n, \by) - \lambda_V(n, \bx)| \le d(\bx, \by).
    \end{equation}    
    Without loss of generality, assume that $\lambda_V(n, \by) \ge \lambda_V(n, \bx)$, so we can rearrange the above equation to the equivalent condition
    \[
        \lambda_V(n, \bx) \ge \lambda_V(n, \by) - d(\bx, \by).
    \]
    If $d(\bx, \by) \ge \lambda_V(n, \by)$, this holds trivially by condition 1, so we must only consider when $d(\bx, \by) < \lambda_V(n, \by)$. 
    
    First, assume that $\bx \preccurlyeq \by$. Let $\epsilon = \lambda_V(n, \by) - d(\bx, \by)$ and choose any $\ba \preccurlyeq \bx \preccurlyeq \bb$ such that $d(\ba, \bx) \le \epsilon$ and $d(\bx, \bb) \le \epsilon$. We wish to show that for all such $\ba$ and $\bb$, $\beta_V^{\ba, \bb} \ge n$, which will show that $\lambda_V(n, \bx) \ge \epsilon = \lambda_V(n, \by) - d(\bx, \by)$, which will complete the proof.
    
    By the triangle inequality, because $\ba \preccurlyeq \bx \preccurlyeq \by$, we have that  
    \begin{equation}
    \label{eqn:one_lipschitz_triangle_ineq}
        d(\ba, \by) \le d(\ba, \bx) + d(\bx, \by) \le (\lambda_V(n, \by) - d(\bx, \by)) + d(\bx, \by) = \lambda_V(n, \by).
    \end{equation}
    On the other hand, we have three cases for the relationship for $\bb$ and $\by$: $\bb \preccurlyeq \by$, $\by \preccurlyeq \bb$, or $\bb$ and $\by$ are not comparable. To handle all of these together, we wish to find a $\bc$ such that $\by \preccurlyeq \bc$, $\bb \preccurlyeq \bc$, and $d(\by, \bc) \le \lambda_V(n, \by)$. We will label the coordinates of $\bb$ as $(r_b, z_b)$ and choose $\bc = (r_c, z_c) = (\max(r_y, r_b), z_b)$. Because $\bb$ and $\bc$ share the same $P$-coordinate and $r_c \ge r_b$ we have that $\bb \preccurlyeq \bc$. On the other hand, because $\bx \preccurlyeq \bb$, we know that $z_b \preccurlyeq z$ in $P$. Thus, because $r_y \le r_c$, we have that $\by \preccurlyeq \bc$. 
    
    Finally, we need to show that $d(\by, \bc) \le \lambda_V(n, \by)$. If $r_b \ge r_y$, then we have that $\bc = (r_b, z_b) = \bb$, so we can use a similar triangle inequality argument as in Equation \ref{eqn:one_lipschitz_triangle_ineq} to show that $d(\by, \bc) = d(\by, \bb) \le \lambda_V(n, \by)$. Otherwise, if $r_b < r_y$, then $\bc = (r_y, z_b)$. However, we observe that 
    \[
        d((r_x, z), (r_x, z_b)) \le d((r_x, z), (r_b, z_b)) = d(\bx, \bb) \le \epsilon = \lambda_V(n, \by) - d(\bx, \by) \le \lambda_V(n, \by).
    \]
    However, because of the additive structure of our sum quasimetric, we have that
    \[
        d(\by, \bc) = d((r_y, z), (r_y, z_b)) = d((r_x, z), (r_x, z_b)) \le \lambda_V(n, \by).
    \]
    Therefore, we have that both $\ba$ and $\bc$ are within distance $\lambda_V(n, \by)$ of $\by$ with $\ba \preccurlyeq \by \preccurlyeq \bc$ so we know that $\beta_V^{\ba, \bc} \ge n$. However, because $\ba \preccurlyeq \bb \preccurlyeq \bc$, we can factor the map $V(\ba) \rightarrow V(\bc)$ (which we know is rank at least $n$) as $V(\ba) \rightarrow V(\bb) \rightarrow V(\bc)$, which means that $\beta_V^{\ba, \bb} \ge n$ as well.

    The case when $\by \preccurlyeq \bx$ is very similar. However, in this case we have that $\by \preccurlyeq \bx \preccurlyeq \bb$, andcan establish that $d(\bx, \bb) \le \lambda_V(n, \by)$ directly via the triangle inequality, and our three cases regard the relationship between $\ba$ and $\by$. In this case, if $\ba = (r_a, z_a)$, we define $\bc = (\min(r_y, r_a), z_a)$; by following the same logic as in the $\bx \preccurlyeq \by$ case with some inequalities flipped, we can show that $\bc \preccurlyeq \by$, $\bc \preccurlyeq \ba$, and $d(\bc, \by) \le \lambda_V(n, \by)$. These three facts combine to establish that $\beta_V^{\bc, \bb} \ge n$, and because $\bc \preccurlyeq \ba \preccurlyeq \bb$, we can factor $V(\bc) \rightarrow V(\bb)$ as $V(\bc) \rightarrow V(\ba) \rightarrow V(\bb)$, establishing $\beta_V^{\ba, \bb} \ge n$ as well.
\end{proof}

We now turn to Proposition \ref{prop:supinf}, which helps to simplify our computations and is a key component of our code. See also \cite[Prop.~22]{Vipond20}. To do so, we require a few new notions, which we discuss in the following remark.

\begin{remark}\label{rmk:paths}
Let $V:P\to\mathrm{Vect}_\mathbb{F}$ be a $P$-module.
\begin{enumerate} 
    \item Each path $f$ in $P$ induces a single-parameter persistence module, which we will denote by $f^*V$: it is the composition of functors $f^*V = V\circ \mathrm{inc}$, where $\mathrm{inc}:f\hookrightarrow P$ is the inclusion map.
    \item Each path from $\ba$ to $\bb$ determines a linear map $V(\ba)\to V(\bb)$; by functoriality of the $P$-module, they all have rank $\beta_V^{a,b}$. Therefore we do not need to discuss paths when calculating $\lambda_P$ in general, but we do need them to compare $\lambda_{\R\times P}$ to single-parameter landscapes, as we do in the following proposition.
\end{enumerate}
\end{remark}

Let $\mathcal{F}_{\bx}$ denote the \textbf{set of all paths} in $\R\times P$ through $\bx$.
\begin{prop}\label{prop:supinf}
    Let $V:\R\times P\to\mathrm{Vect}_\mathbb{F}$ be a $P$-module. For any $\bx \in \R \times P$, $f\in\mathcal{F}_{\bx}$, and $k \in \mathbb{Z}$, let $\lambda_{f^*V}(n,u)$
    be the single-parameter persistence landscape of $f^*V$ (see Remark \ref{rmk:paths}). Then for any extended quasimetric $d$ on $\R\times P$,
    \[
    \lambda_V(n,\bx) = \inf\{\lambda_{f^*V}(n, \bx) \, | \, f \in \mathcal{F}_{\bx} \}.
    \]
\end{prop}

\begin{proof}
    Let $\epsilon_0 \le \inf\{\lambda_{f^*V}(k, \bx) \, | \, f \in \mathcal{F}_{\bx} \}$. Then choose $\ba, \bb \in \R \times P$ such that $\ba \preccurlyeq \bx \preccurlyeq \bb$ and $d(\ba, \bx) \le \epsilon_0$ and $d(\bx, \bb) \le \epsilon_0$. Let $f_{\ba,\bb}$ be any path in $\mathcal{F}_{\bx}$ that passes through $\ba$, $\bx$, and $\bb$. Because we chose $\ba$ and $\bb$ to be within $\epsilon_0$ of $x$ in $P$ and $\epsilon_0 \le \lambda_{f_{\ba, \bb}^*V}(n, \bx)$, we know that $\beta_V^{\ba, \bb} \ge n$. Thus for any $\epsilon_0 \le \inf \{\lambda_{f^*V}(n, \bx) \ | \, f \in \mathcal{F}_{\bx}\}$, we know that $\beta_V^{\ba,\bb}=\beta_{f^*V}^{\ba, \bb} \ge n$ for all $\ba \le \bx \le \bb$ such that $d(\ba, \bx) \le \epsilon_0$ and $d(\bx, \bb) \le \epsilon_0$. Thus $\lambda_V(n,\bx)\geq\epsilon_0$ for all $\epsilon_0\leq\inf\{\lambda_{f^*V}(n, \bx)\,|\,f\in\mathcal{F}_{\bx}\}$, so we may conclude $\lambda_V(n,\bx) \ge \inf\{\lambda_{f^*V}(n, \bx) \, | \, f \in \mathcal{F}_{\bx}\}$.

    On the other hand, choose a path $f \in \mathcal{F}_{\bx}$ and $\ba, \bb \in \R \times P$ on $f$ such that $d(\ba, \bx) \le \lambda_V(n,\bx)$ and $d(\bx, \bb) \le \lambda_V(n,\bx)$. Then $\beta_{f^*V}^{\ba,\bb}=\beta_V^{\ba, \bb} \ge n$, so $\lambda_{f^*V}(n, \bx) \ge \lambda_V(n,\bx)$. Thus, $\lambda_V(n,\bx) \le \lambda_{f^*V}(n, \bx)$ for all such $f$, hence $\lambda_V(n,\bx) \le \inf\{\lambda_{f^*V}(n, \bx) \, | \, f \in \mathcal{F}_{\bx}\}$.
\end{proof}

\subsection{Related work}

In addition to the aforementioned work in~\cite{Vipond20}, which defines and studies landscape functions for modules over $\R^n$, there are a number of other recent developments for the notion of landscapes.

In~\cite{KimMemoli2021generalizedrank}, the authors propose a new invariant of modules over arbitrary posets called the generalized rank invariant.
In~\cite{Dey2024computinggeneralizedrank}, the authors propose an algorithm for computing the generalized rank invariant for 2-parameter persistence modules using zig-zag persistence modules. 
This algorithm is then used for the computation of generalized rank invariant landscapes, proposed in~\cite{Xin2023GRIL}. 
Compared to the landscapes defined by Vipond~\cite{Vipond20}, which are defined using the usual notion of rank invariant and are computed over rectangles, the generalized rank invariant landscapes are defined over a generalization of a rectangle, which the authors refer to as an \textit{interval}.
The authors of~\cite{Xin2023GRIL} show that these generalized landscapes are equivalent to the generalized rank invariant over a covering of $\R^2$ by a set of intervals. 

In very recent works~\cite{flammer2024spatiotemporalpersistencelandscapes, dey2025quasi}, the authors define variants of landscapes for modules over a product of $\R$ and a zig-zag poset $\mathbb Z\mathbb Z$, which can be of interest for studying time-series data. 
In~\cite{flammer2024spatiotemporalpersistencelandscapes}, the authors use the notion of the generalized ranked invariant computed over a rectangular covering of $\R\times \mathbb Z\mathbb Z$ , whereas in~\cite{dey2025quasi}, the authors look at a different covering set, the elements of which are called \textit{worms}.

In comparison to the above works, our definition of landscape is based on a usual (not a generalized) rank invariant.
We are also further focusing on the landscapes for the posets $\R\times P$, where $P$ is a Boolean subset poset.

\section{Computational Aspects and Numerical Experiments}\label{s:experiments}

This section is split into two parts: \S\ref{ss:algorithms}, in which we explain and justify our computational methods for generalized persistence landscapes, and \S\ref{ss:numexs}, in which we present several examples in order of increasing complexity. Throughout this section, the poset $P$ is the Boolean subset poset $P_k$, the extended quasimetric $d_P$ on $P_k$ is the geodesic distance, and the extended quasimetric on $\R\times P_k$ is the sum quasimetric.

\subsection{Computational Aspects}\label{ss:algorithms}

Our computational pipeline relies on Proposition \ref{prop:supinf}, which reduces the problem of computing $\lambda_V(n,(r,p))$ to computing the minimum of a set of single-parameter landscape values. Two difficulties now arise:
\begin{enumerate}
    \item The complexity of computing $\lambda_V$ (equivalent to computing the number of paths passing through a given point $(r,p)\in\R\times P_k$ by Proposition \ref{prop:supinf}) is exponential in $k$. In practice, we restrict to a subposet $D=Z\times P_k\subseteq\R\times P_k$ and interpolate between the values of the landscape calculated at points in $D$ using only paths that turn at points in $D$ (see Definition \ref{def:lamD} below) to approximate those of $\lambda_V$. Therefore, we also want to choose $|Z|$ as small as possible, since our calculation is polynomial in $|Z|$.\footnote{Let $P_k$ be the poset of subsets of $[k]$ (including $\emptyset$), with $\preccurlyeq$ given by inclusion. Let $L$ be a maximal chain in $P_k$, and let $Z$ be a finite subposet of $\R$. The number of different maximal chains in $Z\times L$ is ${|Z|+|L|-2\choose |L|-1}={|Z|+k-1\choose k}$: we have $|L|=k+1$ because $L$ is maximal, and we are counting the number of down-right paths starting at the top left and ending at the bottom right of a $|Z|\times|L|$ dot grid. The number of maximal chains $L$ is $k!$ (from the set $[k]$ there are $n$ choices; from each of those choices, there are $k-1$, etc.), therefore the number of paths in $Z\times P$ is $k!{|Z|+k-1\choose k}=(|Z|+k-1)\cdots(|Z|)=O(|Z|^k)$.
    
    }

    The interpolation issue also arose for Vipond: see \cite[A.3]{Vipond20}. However, verifying that the values which we have computed for $\lambda_V$ at points in $D$ are correct is a new issue. This is because computing $\lambda_V(n,(r,p))$ requires us to take an infimum over the set of all paths in $\R\times P_k$ through $(r,p)$, yet to speed computation we are only considering a subset of paths. 
    We explain why we still obtain the correct values for $\lambda_V$ at points in $D$ in Lemma \ref{lem:sepdisc}, and we bound the error at points not in $D$ in Lemma \ref{lem:discerror} (which is analogous to Vipond's work).
    
    \item Standard tools for computing single-parameter persistence over $\R$ (barcodes) and landscapes rely on the fact that bars correspond to continuous families of vector spaces.
    Since we are mainly interested in the case when $P$ is an inclusion poset, which is finite, there is no family of vector spaces that corresponds to the relations $p\preccurlyeq p'$, for $p,p'\in P_k$.
    However, based on the application, we want to set the distance between $p$ and $p'$ to be non-zero, $d_P(p,p')>0$.
    See Remark~\ref{rmk:flats}. 
    In this case, we obtain the correct landscape values by extending the relevant bars, that is, as further explained in Remark \ref{rmk:algo}~3(c).
\end{enumerate}

\begin{definition}
    A \textbf{discretization} of $\R\times P_k$ is a poset $D=Z\times P_k$, where $Z$ is a finite subposet of $\R$. The inclusion $D\subseteq\R\times P_k$ induces weights on the edges of $D$ and hence the geodesic distance as in Example \ref{ex:gud}.
\end{definition}

Each discretization $D$ of $\R\times P_k$ induces a new landscape function $\lambda_V^D$.

\begin{definition}\label{def:lamD}
    Let $V$ be an $\R\times P_k$-module and let $D$ be a discretization of $\R\times P_k$. For $\bx\in D$, define
    \[
    \lambda_V^D(n,\bx)=\inf\{\lambda_{\bar f}(n,\bx)\,|\,f\in\mathcal{F}_\bx^D\},
    \]
    Here $f$ is a path through $\bx$ in $D$, $\mathcal{F}_\bx^D$ is the set of all such paths, and $\bar f$ is the path in $\R\times P$ with $\bar f|_D=f$. We say $\bar f$ is the \textbf{closure} of $f$. The function $\lambda_{\bar f}$ is the single-parameter landscape of $\bar f^*V$.
\end{definition}

To compute this in practice, there a number of algorithmic choices made, which are described below.
\begin{remark}\label{rmk:algo} 
    To compute $\R\times P_k$-landscapes, we follow these steps:
    \begin{enumerate}
        \item We begin by computing the Vietoris-Rips filtered simplicial complex $\VR(S)$ for every $S$ in the set $\mathcal{S}$ of individual classes and all possible unions of classes.
        \item Next, we choose a discretization $D$ of $\R \times P_k$. For the purposes of examples in this paper, we choose the finite subposet of $\R$ used to define $D$ to be a set of evenly spaced real values between the minimum and maximum filtration values found in $\VR(S)$ for all $S \in \mathcal{S}$.
        \item For each $\bx$ in $D$ and for each path $f \in \mathcal{F}^{D}_{\bx}$, use the following steps to compute $\lambda_{\bar{f}}(n, \bx)$:
        \begin{enumerate}
            \item Create a filtered simplicial complex $C_f$ for $f$. We know that $f$ starts in some single class $X_0 \in P_k$ and there exists some subsequence $\{m_i\}_{i=1}^{n}$ of $[n]$ and a resulting sequence of triples $(\alpha_i, X_{m_i}, w_i)_{i = 1}^n$ where at filtration value $\alpha_i$, $f$ crosses from the union $X_{m_0} \cup X_{m_1} \cup \dots \cup X_{m_{i-1}}$ to the union $X_{m_0} \cup X_{m_1} \cup \dots \cup X_{m_{i-1}} \cup X_{m_i}$, and the edge in $P_k$ from $X_{m_0} \cup \dots \cup X_{m_{i-1}}$ to $X_{m_0} \cup \dots \cup X_{m_i}$ has weight $w_i$. Note that here we are utilizing the correspondence between our covering sets $X_i \in \mathcal{X}$ that define our classes and the source nodes in $P_k$. This correspondence follows on to other elements in $P_k$: the join of the poset elements $X_i$ and $X_j$ corresponds to the union $X_i \cup X_j$, and we will primarily use the latter notation here.
            \begin{enumerate}
                \item Start with $C_f = \VR(X_{m_0})$.
                \item For the first triple $(\alpha_1, X_{m_1}, w_1)$, we first create a modified copy $\overline{\VR}(X_{m_0} \cup X_{m_1})$ of $\VR(X_{m_0} \cup X_{m_1})$: for each simplex $\sigma$ in $\VR(X_{m_0} \cup X_{m_1})$, if it has filtration value $\operatorname{filt}(\sigma) < \alpha_1$, set $\operatorname{filt}(\sigma) = \alpha_1 + w_1$, and if $\operatorname{filt}(\sigma) \ge \alpha_1$, set $\operatorname{filt}(\sigma) = \operatorname{filt}(\sigma) + w_1$. For each simplex $\gamma \in C_f$, if $\operatorname{filt}(\gamma) > \alpha_1$, set $\operatorname{filt}(\gamma) = \operatorname{filt}(\gamma) + w_1$. Finally, for every simplex $\delta \in \overline{\VR}(X_{m_0} \cup X_{m_1})$, insert $\delta$ into $C_f$ if and only if $\delta \not \in C_f$. This creates a filtered simplicial complex that looks like $X_{m_0}$ up to filtration value $\alpha_1$, has a ``gap'' from $\alpha_1$ to $\alpha_1 + w_1$, then looks like a shifted version of $X_{m_0} \cup X_{m_1}$ for higher filtration values.
                \item The steps for the remaining triples are similar, but must account for the gaps of length $w_i$ that we are inserting into the filtration values for $C_f$. For a given $i$, we again create a modified copy $\overline{\VR}(X_{m_0} \cup \dots \cup X_{m_i})$ of $\VR(X_{m_0} \cup \dots \cup X_{m_i})$: for each simplex $\sigma \in \VR(X_{m_0} \cup \dots \cup X_{m_i})$, if $\operatorname{filt}(\sigma) < \alpha_i + \sum_{j = 1}^{i - 1}w_j$, set $\operatorname{filt}(\sigma) = \alpha_i + \sum_{j=1}^{i} w_j$ and if $\operatorname{filt}(\sigma) \ge \alpha_i + \sum_{j=1}^{i-1}w_j$ set $\operatorname{filt}(\sigma) = \operatorname{filt}(\sigma) + \sum_{j=1}^{i} w_j$. For every simplex $\gamma \in C_f$, if $\operatorname{filt}(\gamma) > \alpha_i + \sum_{j = 1}^{i - 1} w_j$, set $\operatorname{filt}(\gamma) = \operatorname{filt}(\gamma) + w_i$. Finally, for each simplex $\delta \in \overline{\VR}(X_{m_0} \cup \dots \cup X_{m_i})$ insert $\delta$ into $C_f$ if and only if $\delta \not \in C_f$.
            \end{enumerate}
            \item Compute the persistence barcode $B(C_f)$ for $C_f$.
            \item\label{item:c} The ``gaps'' of lengths $w_i$ that we inserted correspond to edges in the poset that do not have any vector spaces along them (unlike the edges in the $\R$ direction, which are dense with infinitely many vector spaces). For landscape computations, this means that the requirement that the rank of maps be at least $n$ is vacuously true along that edge; to emulate this effect while still making use of existing computational tools for one-dimensional landscapes over $\R$, we modify the computed barcode directly as follows. For each triple $(\alpha_i, X_{m_i}, d_i)$ and each bar $(a, b) \in B(C_f)$, if $a = \alpha_i + \sum_{j = 1}^{i} w_j$, replace the bar $(a, b)$ with $(a - w_i, b)$. That is, if a bar starts precisely at the filtration value corresponding to the path arriving at a new union, we extend that bar back to the beginning of the edge the path just traversed.
            \item Compute the one-dimensional landscape $\lambda_{B(C_f)}(n, y)$ from the barcode $B(C_f)$.
            \item To find the path landscape value $\lambda_{\bar{f}}(n, \bx)$, we first coordinatize $\bx = (r, p)$ and note that $p = (X_{m_0} \cup \dots \cup X_{m_i}) \in P_k$ for some $i$. Then we have that $\lambda_{\bar{f}} (k, \bx) = \lambda_{\bar{f}}(k, (r, p)) = \lambda_{B(C_f)}(k, r + \sum_{j=1}^{i-1} w_j)$, as we simply need to adjust our filtration value for the gaps we introduced.
        \end{enumerate}
        \item For a finite discretization $D$ and finite poset $P_k$, there are only finitely many paths $f$ through a point $\bx$, so to find the $\R\times P_k$-landscape value for $\bx$, we simply take the minimum of the path landscape values; that is,
        \[
            \lambda^D_V(n, \bx) = \min_{f \in \mathcal{F}_\bx^D} \lambda_{\bar{f}}(n, \bx).
        \]
    \end{enumerate}
    In practice, to improve the computational efficiency, we leverage the fact that each path passes through many points in $D$ by first finding all paths through $D$ and computing their path landscapes as described above, then simply looking up values from the appropriate paths when computing the $\R\times P_k$-landscape value at a particular point in the discretization.
\end{remark}

\begin{remark}
    In step 3(d) of the algorithm described in the previous remark, it is important to note that the difference between the maximum and minimum filtration values used when computing the one-dimensional landscape $\lambda_{B(C_f)}(k, x)$ is in general larger than the difference between max and min filtration values used for $D$. In particular, this difference is expanded by the sum of the relevant edge lengths $w_i$ from $P_k$. When those edge lengths are significantly larger than the original difference between max and min filtration values used for $D$, the one-dimensional discretization that TDA tools (such as the software package gudhi) use to compute one-dimensional landscapes must be refined sufficiently to be fine enough to give good results on the scale of the discretization $D$. As an example, if the filtration range of $D$ is from $0$ to $2$, but the sum of the $d_i$ is $999$, then a one-dimensional discretization with $1002$ points will only yield useful values at $0, 1,$ and $2$ in $D$ and have to be interpolated in between. In such a case, we would want to refine the one-dimensional discretization to have on the order of 10 to 100 thousand points.
\end{remark}


We now prove that the methods of Remark \ref{rmk:algo} calculate $\lambda_V$ correctly at points in $D$. This result (as well as the next one) hold for all $P$, not just $P_k$, and are stated in that generality.

\begin{lemma}\label{lem:sepdisc}
    Let $V$ be an $(\R\times P)$-module and let $D=Z\times P$ be a discretization of $\R\times P$. Then
    \[
    \lambda_V^D(n,\bx)=\lambda_V(n,\bx)
    \]
    for all $\bx\in D$.
\end{lemma}
\begin{proof}
    Because $\mathcal{F}^D_\bx\subseteq\mathcal{F}_\bx$, by Proposition \ref{prop:supinf} we immediately have
    \[
    \lambda_V^D(n,\bx)\geq\lambda_V(n,\bx).
    \]
    Therefore, our task is to prove the reverse inequality: that there are no $\ba\preccurlyeq\bx\preccurlyeq\bb$ with $d(\ba,\bx)$ or $d(\bx,\bb)$ strictly less than $\lambda_V^D(n,\bx)$ and for which $\beta_V^{\ba,\bb}<n$.

    Assume otherwise. Let $\delta=\lambda_V^D(n,\bx)$. There is a path in $D$ from $\bx=(r_x,x)$ whose closure passes through $\bb=(r_b,b)$: let $f_P$ be any path in $P$ from $x$ to $b$ centered at $x$. Then $\{r_x\}\times f_P\in\mathcal{F}^D_\bx$ because $D$ is a product $Z\times P$. Now let $s$ be the smallest element of $Z$ which is greater than or equal to $r_b$, and concatenate the (appropriate translation of the) path
    \[
    t\mapsto(t,b), \quad r_x\leq t\leq s
    \]
    to the end of $\{r_x\}\times f_P$. If no such $s$ exists, use the bounds $r_x\leq t<\infty$. This is the closure of a path in $D$. Similarly, the point $\ba$ is contained in the closure of a path in $D$. Concatenating both paths provides us with a path $f$ in $D$ whose closure $\bar f$ is centered at $\bx$ and passes through both $\ba$ and $\bb$.

    By assumption, both $\ba$ and $\bb$ are strictly contained in the $\delta$-ball around $\bx$ and $\beta_V^{\ba,\bb}<n$. But that contradicts the fact that $\lambda_{\bar f}(n,\bx)\geq \delta$, since $f\in\mathcal{F}^D_\bx$.
\end{proof}

Our code does not actually compute $\lambda_V(n,\bx)$ for $\bx\not\in D$, but our plots do indicate those values by the straight-line interpolation from $(r,\lambda_V^D(n,(r,p))$ to $(r',\lambda_V^D(n,(r',p))$: for $(r,p), (r',p)\in D$ with $r<r'$ such that $r'$ covers $r$ (i.e., there are no other $z\in Z$ satisfying $r<z<r'$) and any $r<s<r'$, define our approximation $\check\lambda_V$ by
\[
\check\lambda_V(n,(s,p))\coloneqq\frac{\lambda_V^D(n,(r',p))-\lambda_V^D(n,(r,p))}{r'-r}(s-r)+\lambda_V^D(n,(r,p)).
\]

Graphically, the function $\check\lambda_V$ takes values on the diagonal of a potential rectangle of true values for $\lambda_V$, as explained in the proof of the next lemma.

\begin{lemma}\label{lem:discerror} Let $V$ be an $(\R\times P)$-module and let $D=Z\times P$ be a discretization of $\R\times P$. Then for all $(s,p)\in\R\times P$,
\[
|\lambda_V(n,(s,p))-\check\lambda_V(n,(s,p))| \leq r'-r,
\]
where $r<s<r'$ and $r<r'$ are consecutive in $Z$ (i.e., there are no other $z\in Z$ satisfying $r<z<r'$).
\end{lemma}

\begin{remark}
    Compare this bound to \cite[\S A.3]{Vipond20}, setting $i=1$. Our result provides the analogous (i.e., when generalizing from $\R^{d-1}$ to $P$) global bound on the difference between $\lambda_V$ and $\check\lambda_V$ by taking the maximum over all covering pairs $r<r'$ in $Z$.
\end{remark}

\begin{proof}
First note that by the $1$-Lipschitz property (Lemma~\ref{lem:properties}), we have
    \[
    -1\leq \frac{\partial}{\partial r}\,\lambda_V(n,(s,p))\leq 1.
    \]
    Now let $r<r'\in Z$, and let $l=\lambda_V(n,(r,p))$ and $l'=\lambda_V(n,(r',p))$. The largest possible value for $\lambda_V(n,(s,p))$ for any $r<s<r'$ is achieved by the $\lambda$-values on the $s\lambda$-plane given by the minimum of the line of slope $1$ through $(r,l)$ and the line of slope $-1$ through $(r',l')$. Analogously, the lowest value is achieved by the maximum of the line of slope $-1$ through $(r,l)$ and the line of slope $1$ through $(r',l')$.

    The interpolation $\check\lambda_V$ takes values on the diagonal between $(r,l)$ and $(r',l')$. Therefore, for $r<s<r'$, the difference
    \[
    |\lambda_V(n,(s,p))-\check\lambda_V(n,(s,p))|
    \]
    is at most the maximum of the vertical difference between the diagonal and each boundary of the rectangle drawn in Figure \ref{fig:bounds}. Those differences are necessarily all less than the vertical height of the rectangle. A rectangle whose sides have slope $\pm1$ has equal vertical height and horizontal width (that is, its bounding box is a square): see Figure \ref{fig:bounds}.

    \begin{figure}[h]
\centering
    \begin{tikzpicture}
\draw[black, very thick, dashed] (-2,0) -- (0,2);
\draw[black, very thick, dashed] (0,2) -- (4,-2);
\draw[black, very thick] (4,-2) -- (2,-4);
\draw[black, very thick] (2,-4) -- (-2,0);
\filldraw[black] (-2,0) circle (2pt) node[anchor=east]{$(r,l)$};
\filldraw[black] (4,-2) circle (2pt) node[anchor=west]{$(r',l')$};
\draw[black, thick, dotted] (-2,0) -- (4,-2) node[midway, below]{$\check\lambda_{\R\times P}(k,(s,p))$};
\draw[black] (-2,-4) -- (-2,2);
\draw[black] (-2,2) -- (4,2);
\draw[black] (4,2) -- (4,-4);
\draw[black] (4,-4) -- (-2,-4);
\end{tikzpicture}
\caption{The thick dashed line indicates the upper bound of $\lambda_V(n,(s,p)$ between $(r,l)$ and $(r',l')$, while the thick solid line indicates its lower bound. The dotted diagonal line indicates the plotted value $\check\lambda_V(n,(s,p))$ for $r<s<r'$. The difference between $\check\lambda_V$ and the actual value of $\lambda_V$ is bounded by the vertical height of the bounding box (depicted in thin lines); because all angles are right angles, the triangles between the bounding box and the rectangle are right isosceles, and so the bounding box is a square.}
\label{fig:bounds}
\end{figure}
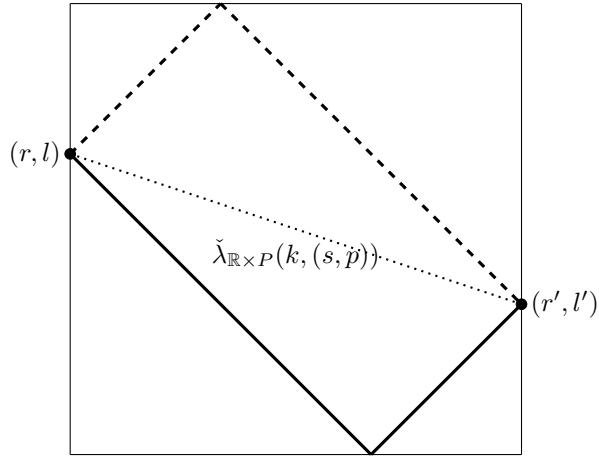

    Finally, the horizontal width of the bounding box is precisely $|r'-r|$, as desired.
\end{proof}

\subsection{Numerical Examples}\label{ss:numexs}

In this section we present several examples, organized by complexity of the point cloud $X$. For simplicity, we focus on the case $k=2$ and $X\subseteq\R^2$ (except in \S\ref{ss:realcode}). We start with first principles in \S\ref{ss:byhand}, where we describe several $(\R\times P_2)$-landscapes computed directly from Definition \ref{def:2}. Then we move into a more abstract discussion of the ways in which generalized persistence landscapes and single-parameter persistence landscapes differ: see Remark \ref{rmk:flats} and Lemma \ref{lem:dPrange}. In \S\ref{ss:syntheticcode} we turn to more complicated data. For these calculations, we rely on our code as presented in \S\ref{ss:algorithms}. Example \ref{ex:synthetic} illustrates Remark \ref{rmk:flats} and Lemma \ref{lem:dPrange} using several simple examples with $|X_1|=|X_2|=4$; the relevant features are very easy to identify. Example \ref{ex:noise} illustrates the robustness of the labeled persistentce landscape to noise. Finally, in \S\ref{ss:realcode},  we perform a relatively straightforward experiment on the MNIST dataset, a commonly-used real dataset used in machine learning.

As we focus on the case $k=2, X\subseteq\R^2$, we use the notation
\[
\mathrm{LPH}_j(r,S)=\begin{cases}
    \mathrm{LPH}_j(\mathcal{X})(r,\{1\}) & \text{ if } S=X_1,
    \\\mathrm{LPH}_j(\mathcal{X})(r,\{2\}) & \text{ if } S=X_2,
    \\\mathrm{LPH}_j(\mathcal{X})(r,\{1,2\})&\text{ if }S=X_1\cup X_2=X
\end{cases}
\]
to denote the components of the labeled persistent homology module of $(X,d_X,L_X)$, where $d_X$ is the Euclidean distance restricted to $X$ and $L_X=(X_1,X_2)$, with $X_1\cup X_2=X$ and $X_1\cap X_2=\emptyset$. 
Furthermore, because it is clumsy to repeat each of the three cases above in every instance we need to refer to the correspondence between elements of $P_2$ and subsets of $X$, we occasionally abuse notation and indicate elements of $\R\times P_2$ by $(r,S)$ rather than $(r,p)$, where $p$ is the element of $P_2$ corresponding to $S$ as above.

\subsubsection{Direct Calculations}\label{ss:byhand}

We first present an example of a generalized persistence landscape computed by hand, together with the corresponding output from our code. The details (as well as two simpler examples) are deferred to Appendix \ref{app:calc}. Then, in Remark \ref{rmk:flats}, we list the three basic ways an $(\R\times P_2)$-landscape restricted to either $X_1, X_2$, or $X$, respectively, can differ from the single-parameter landscape of either $X_1, X_2$, or $X$, respectively. 
Finally, Lemma \ref{lem:dPrange} provides an estimate on those $d_P$ for which certain $\lambda_V|_{\cup_{i\in p}X_i}$ must equal their original single-parameter persistence landscapes, giving us an estimate for the $d_P$ for which the differences explained in Remark \ref{rmk:flats} might arise.

\begin{example}\label{ex:byhand1}

Let $X_1=\{(0,0),(0.4,0),(1,0)\}$ and $X_2=\{(0,1),(0.4,1),(1,1)\}$. See Figure \ref{fig:AB}.

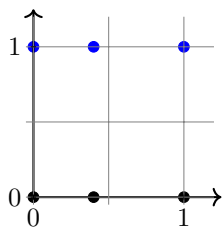
\begin{figure}[h]
\centering
    \begin{tikzpicture}
\filldraw[black] (0,0) circle (2pt);
\filldraw[black] (0.8,0) circle (2pt);
\filldraw[black] (2,0) circle (2pt);
\filldraw[blue] (0,2) circle (2pt);
\filldraw[blue] (0.8,2) circle (2pt);
\filldraw[blue] (2,2) circle (2pt);
\draw[thick,->] (0,0) -- (2.5,0);
\draw[thick,->] (0,0) -- (0,2.5);
\draw[step=1cm,gray,very thin] (-.1,-.1) grid (2.4,2.4);
   \draw (0 cm,1pt) -- (0 cm,-1pt) node[anchor=north] {$0$};
   \draw (2 cm,1pt) -- (2 cm,-1pt) node[anchor=north] {$1$};
    \draw (1pt,0 cm) -- (-1pt,0 cm) node[anchor=east] {$0$};
    \draw (1pt,2 cm) -- (-1pt,2 cm) node[anchor=east] {$1$};
\end{tikzpicture}
\caption{We have depicted $X_1$ in black and $X_2$ in blue.}
\label{fig:AB}
\end{figure}

    A diagram indicating the image of the $(\R\times P_2)$-module $V(r,S)=\mathrm{LPH}_0(r,S)$ is below (omitting $\R\times\{\emptyset\}$). We have used the notation $V^I_S\coloneqq V(r,S)$, where $r\in I$, and $I$ is an open or half-open interval between values at which the rank of $\mathrm{LPH}_0(r,S)$ changes for $S$ equal to any of $X_1, X_2$, or $X$ (the $V^I_S$ are therefore isomorphic for all $r\in I$). We have also removed the labels $V^{(-\infty,0)}_S$ and $V^{[d,\infty)}_S$ for space, where $d$ is the diameter of $S$, since by our convention both are zero.
    \[
    \begin{tikzcd}
        0 \arrow[r] \arrow[d] & V^{[0,0.4)}_{X_1}=\mathbb{F}^3 \arrow[r, "\mathbf{V}_{X_1}^{0.4}"] \arrow[d, "\begin{pmatrix}I_3\\0\end{pmatrix}"] & V^{[0.4,0.6)}_{X_1}=\mathbb{F}^2 \arrow[d, "\begin{pmatrix}I_2\\0\end{pmatrix}"] \arrow[r, "\mathbf{V}_{X_1}^{0.6}"] & V^{[0.6,1)}_{X_1}=\mathbb{F} \arrow[r] \arrow[d, "\begin{pmatrix}1\\0\end{pmatrix}"] & 0 \arrow[d]\arrow[r] & 0 \arrow[d]
        \\0 \arrow[r] & V^{[0,0.4)}_{X}=\mathbb{F}^6 \arrow[r, "\mathbf{V}_{X}^{0.4}"] & V^{[0.4,0.6)}_{X}=\mathbb{F}^4 \arrow[r, "\mathbf{V}_{X}^{0.6}"] & V^{[0.6,1)}_{X}=\mathbb{F}^2 \arrow[r, "\mathbf{V}_{X}^1"] & V^{[1,\sqrt{2})}_{X}=\mathbb{F} \arrow[r] & 0
        \\0 \arrow[r] \arrow[u] & V^{[0,0.4)}_{X_2}=\mathbb{F}^3 \arrow[r, "\mathbf{V}_{X_2}^{0.4}"] \arrow[u, "\begin{pmatrix}0\\I_3\end{pmatrix}"] & V^{[0.4,0.6)}_{X_2}=\mathbb{F}^2 \arrow[r, "\mathbf{V}_{X_2}^{0.6}"] \arrow[u, "\begin{pmatrix}0\\I_2\end{pmatrix}"] & V^{[0.6,1)}_{X_2}=\mathbb{F} \arrow[r] \arrow[u, "\begin{pmatrix}0\\1\end{pmatrix}"] & 0 \arrow[u]\arrow[r] & 0 \arrow[u]
    \end{tikzcd}
    \]
    We have set $\mathbf{V}^{\min I'}_S\coloneqq V((r,S)\preccurlyeq(r',S))$ for $r\in I, r'\in I'$ with $\sup I=\min I'$. These maps are:
    \begin{align*}
        \mathbf{V}_{X_1}^{0.4}=\mathbf{V}_{X_2}^{0.4}&=\begin{pmatrix}1&1&0\\0&0&1\end{pmatrix}
        \\\mathbf{V}_{X_1}^{0.6}=\mathbf{V}_{X_2}^{0.6}&=\begin{pmatrix}1&1\end{pmatrix}
        \\\mathbf{V}_{X}^{0.4}&=\mathbf{V}_{X_1}^{0.4}\oplus \mathbf{V}_{X_2}^{0.4}
        \\\mathbf{V}_{X}^{0.6}&=\mathbf{V}_{X_1}^{0.6}\oplus \mathbf{V}_{X_2}^{0.6}
        \\\mathbf{V}_{X}^1&=\begin{pmatrix}1&1\end{pmatrix}.
    \end{align*}

    Let $\lambda_S(n,r)\coloneqq\lambda_V(n,(r,S))$. Using the method outlined in Appendix \ref{app:calc}, we calculate:
    \begin{align*}
        \lambda_{X_1}(1,r)=\lambda_{X_2}(1,r)&=\min\{r,1-r\},
        \\\lambda_{X_1}(2,r)=\lambda_{X_2}(2,r)&=\min\{r,0.6-r\},\\
        \lambda_{X_1}(3,r)=\lambda_{X_2}(3,r)&=\min\{r,0.4-r\},
    \end{align*}
    and all other $\lambda_{X_1}(n,r)$ and $\lambda_{X_2}(n,r)$ are zero. Furthermore, setting $d_1=d_P(\{1\},\{1,2\})$ and $d_2=d_P(\{2\},\{1,2\})$, we have
    \begin{align*}
        \lambda_{X}(1,r)&=\min\{\sqrt{2}-r,r\},\\
        \lambda_{X}(2,r)&=\begin{cases}
        \min\{1-r,r,\max\{d_1,0.6-r\},\max\{d_2,0.6-r\}\} & \text{ if $r\in[0,0.6)$, and}
        \\\min\{1-r,d_1,d_2\} & \text{ if $r\in[0.6,1)$, and }
        \\0 & \text{ else,}\end{cases}\\
        \lambda_{X}(3,r)&=\begin{cases}
        \min\{0.6-r,r,\max\{d_1,0.4-r\},\max\{d_2,0.4-r\}\} & \text{ if $r\in[0,0.4)$, and}
        \\\min\{0.6-r,d_1,d_2\} & \text{ if $r\in[0.4,0.6)$, and }
        \\0 & \text{ else,}\end{cases}\\
        \lambda_{X}(4,r)&=\begin{cases}\min\{d_1,d_2,0.6-r,r\} & \text{ if $r\in[0,0.6)$, and }
        \\0 & \text{ else,}
        \end{cases}\\
        \lambda_{X}(5,r)&=\lambda_{X}(6,r)=\begin{cases} \min\{d_1,d_2,0.4-r,r\} &\text{ if $r\in[0,0.4)$, and}
        \\0 & \text{ else.}\end{cases}
    \end{align*}
    All other $\lambda_{X}(n,r)$ are zero. 
    See Figure \ref{fig:ex3} for the output of our code when $d_1=d_2=0.1$.
    
    \begin{figure}[htp]
        \centering
        \begin{subfigure}[b]{0.28\textwidth}
    \centering
    \includegraphics[width=\textwidth]{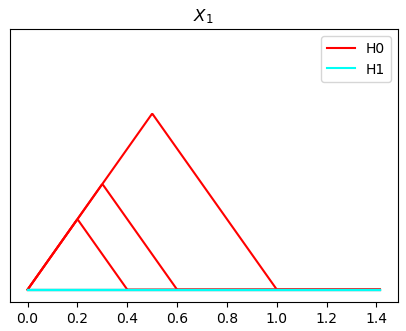}
\end{subfigure}
\hspace{1cm}
\begin{subfigure}[b]{0.28\textwidth}
    \centering
    \includegraphics[width=\textwidth]{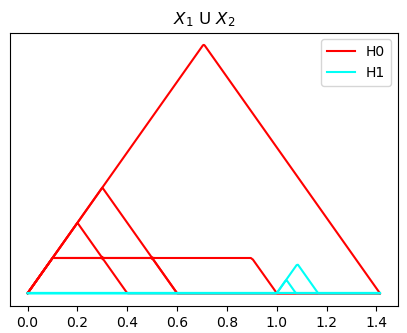}
\end{subfigure}
\hspace{1cm}
\begin{subfigure}[b]{0.28\textwidth}
    \centering
    \includegraphics[width=\textwidth]{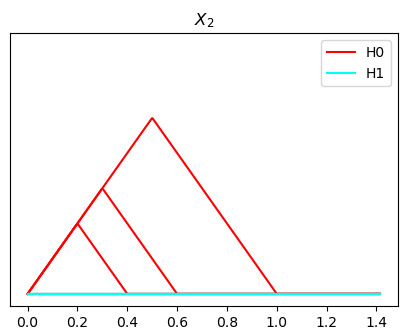}
\end{subfigure}
        \caption{From left to right, we depict the plots $\lambda_{X_1}, \lambda_{X_2}$, and $\lambda_{X}$ for the sets $X_1=\{(0,0),(0.4,0),(1,0)\}$ and $X_2=\{(0,1),(0.4,1),(1,1)\}$, and for the module $V(r,S)=\mathrm{LPH}_0(r,S)$. The plots of $\lambda_{X_i}(n,\cdot)$ with $n=1,2,3$ are the three obvious triangles; to obtain $\lambda_X(n,\cdot)$ for $n=1,\dots,6$, take the $n$-max of the three red triangles and \textit{two copies of the red trapezoid} (one for each $X_i$). The cyan graphs are relevant for the module $(r,S)\mapsto \mathrm{LPH}_1(r,S)$, which we have not calculated in the text.}
        \label{fig:ex3}
    \end{figure}
    
\end{example}

\begin{remark}[Interpretation of the flat regions]\label{rmk:flats} Using the same notation as in the previous example, there are three fundamental ways in which $\lambda_{X_1}, \lambda_{X_2}$, or $\lambda_{X}$, respectively, can differ from the single-parameter persistence landscape of $X_1, X_2$, or $X$, respectively, for a given rank $n$. We only analyze the simplest cases, when just one feature is involved.

Recall that the slope of a single-parameter persistence landscape is always either positive or negative one on the support of the function \cite{Bubenik15}. While our landscapes also contain regions of slope $\pm1$, they also contain flat regions, where $r\mapsto\lambda_V(n,(r,x))$ is nonzero and has slope zero. This occurs because the metric on $P_k$ is discrete, so it is possible for the set of $\epsilon$ over which we are taking the supremum used to define $\lambda_V$ to be constant in $r$. In the single-parameter case, this set always varies as $r$ varies, and standard computational methods for persistence landscapes rely on that fact. See Remark~\ref{rmk:algo}~(\ref{item:c}) for our workaround.

\begin{enumerate}
    \item There is an interval on which $\lambda_{X}$ has slope negative one, then slope zero, then slope negative one. Such an interval arises if, with $S=X_i$, we have $d_P({i},\{1,2\})<\epsilon$ and the rank of the map $V((r,\{i\})\preccurlyeq(r,\{1,2\}))$ is less than $n$, while simultaneously the ranks of the linear maps $V((r-\delta,\{1,2\})\preccurlyeq(r,\{1,2\}))$ for all $0<\delta<\epsilon$ are at least $n$. If $S=X_1$, the ``additional" rank in $V(r,\{1,2\})$ comes from cycles including points from $X_2$, and vice versa.
    
    In Example~\ref{ex:byhand1}, such a situation appears in the commutative diagram below, for $0.6\leq r<1$ and any $d_P(\{1\},\{1,2\})\leq 1-r$,
    \[
    \begin{tikzcd}[ampersand replacement=\&]
        V^{[0.4,0.6)}_{X_1}=\mathbb{F}^2 
        \arrow[d, "{\begin{psmallmatrix}I_2\\0\end{psmallmatrix}}"'] 
        \arrow[r, "{\begin{psmallmatrix}1 & 1\end{psmallmatrix}}"] 
        \& V^{[0.6,1)}_{X_1}=\mathbb{F} 
        \arrow[d, "{\begin{psmallmatrix}1\\0\end{psmallmatrix}}"] \arrow[r]
        \& 0 \arrow[d]
        \\V^{[0.4,0.6)}_{X}=\mathbb{F}^4 
        \arrow[r, "{\begin{psmallmatrix}1&1&0&0\\0&0&1&1\end{psmallmatrix}}"'] 
        \& V^{[0.6,1)}_{X}=\mathbb{F}^2 \arrow[r, "{\begin{psmallmatrix}1&1\end{psmallmatrix}}"']
        \& V^{[1,\sqrt{2})}_{X}=\mathbb{F}
    \end{tikzcd}
    \]
    as well as the analogous square for $X_2$. The map $V((r,X_1)\preccurlyeq(r,X))$ has rank one, meaning $\lambda_{X}(r)=\lambda_V(2,(r,X))\leq d_P(\{1\},\{1,2\})$. The other visible rank one map through $(r,X)$ is $V((r,X)\preccurlyeq(1,X))$, meaning $\lambda_{X}(2,r)\leq 1-r$. The slope zero region in $\lambda_{X}(2,r)$ is therefore at height $d_P(\{1\},\{1,2\})$, while the second slope negative one region is the graph of $1-r$. (We do not actually have enough information to claim this much in this small part of the diagram, because there is one more map of rank less than two: $V((s,X)\preccurlyeq(r,X))$. However, if $s\geq0$, then its rank is two.)
    
    In general, the height of the slope zero region in $\lambda_{X}$ is $d_P(\{1\},\{1,2\})$ or $d_P(\{2\},\{1,2\})$, depending on whether $\mathrm{LPH}_j(r,X_1)\to\mathrm{LPH}_j(r,X)$ or the analogous map from $\mathrm{LPH}_j(r,X_2)$ is the rank less than $n$ (if both are, then we take the minimum distance). Each slope negative one section has an $r$-intercept; the larger $r$-intercept is the death time in $X$.
    \item There is an interval on which $\lambda_{X_1}$ (or $\lambda_{X_2}$, respectively) has slope negative one, then slope zero, then slope negative one. Such an interval arises when there is a cycle which dies earlier in $X$ than in $X_1$ (or $X_2$, respectively). See Figure \ref{fig:3pts} for an example.
    
    As in the case above, the height is $d_P(\{1\},\{1,2\})$ (or $d_P(\{2\},\{1,2\})$, respectively), and the larger $r$-intercept of the negative slope region is the death time of the feature in $X_1$ (or in $X_2$, respectively).  The start of the slope zero region is the death time in $X$.

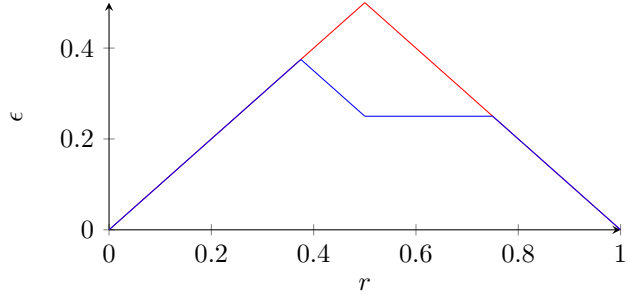
\begin{figure}[h!]
\centering
    \begin{tikzpicture}
\begin{axis}[
scale only axis = true,
width=0.45\textwidth,
height=0.2\textwidth,
    axis lines = left,
    xlabel = \(r\),
    ylabel = {\(\epsilon\)},
]
\addplot [
    domain=0:1/2, 
    samples=100, 
    color=red,
]
{x};
\addplot [
    domain=1/2:1, 
    samples=100, 
    color=red,
    ]
    {1-x};
\addplot [
    domain=0:3/8, 
    samples=100, 
    color=blue,
]
{x};
\addplot [
    domain=1/2:3/4, 
    samples=100, 
    color=blue,
    ]
    {1/4};
\addplot [
    domain=3/8:1/2, 
    samples=100, 
    color=blue,
    ]
    {3/4-x};
\addplot [
    domain=3/4:1, 
    samples=100, 
    color=blue,
    ]
    {1-x};
\end{axis}
\end{tikzpicture}
\caption{Let $X_1=\{0,2\}$ and $X_2=\{1\}$. In red we have plotted $\lambda_{X_1}^\R(2,r)$, the rank two single-parameter persistence landscape of $\mathrm{LPH}_0(X_1,r)$. In blue we have plotted $\left.\lambda_V(2,r)
\right|_{X_1}$, the rank two generalized persistence landscape of $\mathrm{LPH}_0(X,r)$ restricted to $X_1$, assuming $d_P(\{1\},\{1,2\})=1/4$ and using the logic outlined in Appendix \ref{app:calc}.}
\label{fig:3pts}
\end{figure}

    \item There is an interval on which $\lambda_{X}$ has slope one, then slope zero, then slope one. The underlying feature is a cycle that is born earlier in $X$ than in either $X_1$ or $X_2$. Again, the height is either $d_P(\{1\},\{1,2\})$ or $d_P(\{2\},\{1,2\})$. See Example~\ref{ex:synthetic} for many examples in $\mathrm{LPH}_1(X,r)$.
\end{enumerate}
In all three cases, we assume that the distances $d_P(\{1\},\{1,2\})$ and $d_P(\{2\},\{1,2\})$ are low enough for the slope zero regions to appear: see Lemma \ref{lem:dPrange} for more details.

Finally, note that it is not possible for $\lambda_{X_1}$ (or $\lambda_{X_2}$, respectively) to have slope one, then zero, then one, which would require a feature that is born earlier in $X_1$ (or $X_2$, respectively) than in $X$. Because $X_1\subseteq X$ (and $X_2\subseteq X$, respectively), every cycle in $X_1$ (respectively, $X_2$) is a cycle in $X$. There may not be a corresponding bar in the barcode of $X$, but only because that feature came to an early death, not because it was never born.
\end{remark}

Note that the flat regions discussed in Remark \ref{rmk:flats} all have height equal to either $d_P(\{1\},\{1,2\})$ or $d_P(\{2\},\{1,2\})$. Therefore, we only see flat regions in $\lambda_V$ when $d_P$ is small enough. This is summarized in the following simple lemma, which is stated only for the case $k=2$ for simplicity.

\begin{lemma}\label{lem:dPrange} Let $\lambda_S^\R$ denote the single-parameter persistence landscape of $H_j(\VR_r(S))$, and let $\lambda_S$ denote the restriction of the generalized persistence landscape of $V(r,S)=\mathrm{LPH}_j(r,S))$ to $S$.
\begin{enumerate}[label=(\roman*)]
    \item For $i=1$ or $2$, if $d_P(\{i\},\{1,2\})\geq||\lambda_{X_i}^\R(n,\cdot)||_\infty$, then $\lambda_{X_i}(n,(\cdot,X_i))=\lambda_{X_i}^\R(n,\cdot)$.
    \item If both $d_P(\{1\},\{1,2\})\geq||\lambda_{X}^\R(n,\cdot)||_\infty$ and $d_P(\{2\},\{1,2\})\geq||\lambda_{X}^\R(n,\cdot)||_\infty$, then $\lambda_{X}(n,(\cdot,X))=\lambda_{X}^\R(n,\cdot)$.
\end{enumerate}
    We may replace $||\lambda_S^\R(n,\cdot)||_\infty$ in all hypotheses above by $\frac{1}{2}\mathrm{diam}(S)$.
\end{lemma}
\begin{proof}
Recall that by Proposition~\ref{prop:supinf},
\[
\lambda_S(n,(r,S))=\inf\{\lambda_f(n,(r,S)) \, | \, f \in \mathcal{F}_{(r,S)} \},
\]
and note that if $f_S$ is the path $\R\times S$, then
\[
\lambda_{f_S}(n,(r,S))=\lambda_S^\R(n,r).
\]
To prove \textit{(i)}, let us compute $\lambda_{X_i}(n,(r,X_i))$. Every other path $f$ through $(r,\{i\})$ has at least one edge of the form $\{r'\}\times\{\{i\}\to \{1,2\}\}$. Assume $r<r'$ (the case of $r'<r$ is analogous). If $r+\lambda_{X_i}^\R(n,(r,X_i))<r'$, then $\lambda_f(n(r,X_i))=\lambda_{f_{X_i}}(n,(r,X_i))$. Otherwise, if
\[
\beta_V^{(r,\{i\}),(r',\{1,2\})}>n,
\]
then $d((r,\{i\}),(r',\{1,2\}))$ does not contribute to the supremum. Finally, if
\[
\beta_V^{(r,\{i\}),(r',\{1,2\})}<n,
\]
then $d((r,\{i\}),(r',\{1,2\}))$ does contribute to the supremum, but 
    \begin{equation}\label{eqn:dPd}
    \lambda_{f_{X_i}}(n,(r,X_i))\leq||\lambda_{X_i}^\R(n,\cdot)||_\infty\leq d_P(\{i\},\{1,2\})\leq d((r,\{i\}),(r',\{1,2\})),
    \end{equation}
    so we still have
    \[
    \lambda_{X_i}(n,(r,X_i))=\lambda_{f_{X_i}}(n,(r,X_i))=\lambda_{X_i}^\R(n,r).
    \]
    The last inequality in (\ref{eqn:dPd}) relies on the fact that we are using the sum quasimetric. 
    The proof of \textit{(ii)} is similar.

    To prove the final claim, note that we have
    \[
    \frac{1}{2}\mathrm{diam}(S)\geq||\lambda_S^\R(n,\cdot)||_\infty
    \]
    because if $\epsilon>\frac{1}{2}\mathrm{diam}(S)$ then either $V((r-\epsilon,S))$ or $V((r+\epsilon,S))$ is the zero vector space.\footnote{In the case $j=0, n=1, r\geq\frac{1}{2}\mathrm{diam}(S)$ this is a consequence of our conventions. See the discussion in Remark \ref{ex:mixupL}.}
\end{proof}

\subsubsection{Synthetic Data}\label{ss:syntheticcode}

\begin{figure}
    \centering
    \includegraphics[width=0.3\linewidth]{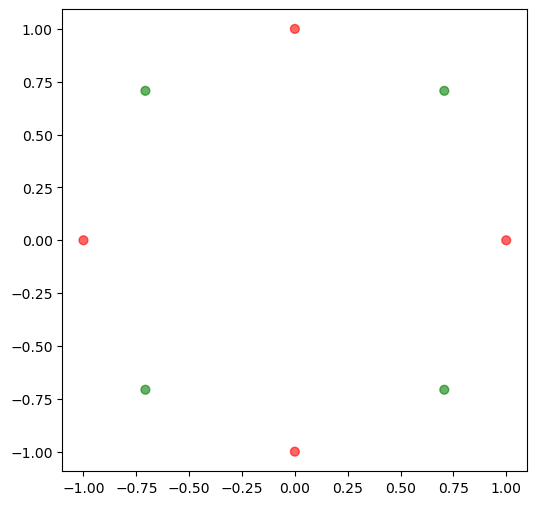}
    \caption{The 2-labeled metric space described in  Example \ref{ex:synthetic}.}
    \label{fig:4ptscirc}
\end{figure}

In practice, the relevant lower bound for $d_P$ at which we find $\lambda_S=\lambda_S^\R$ seems to be the Hausdorff distance between $X_1$ and $X_2$. Recall that for subsets $X_1$ and $X_2$ of the same metric space,
\[
\mathrm{diam}(X)\geq \mathsf{H}(X_1,X_2),
\]
where $\mathsf{H}$ is the Hausdorff distance
\[
\mathsf{H}(X_1,X_2)\coloneqq\max\left\{\sup_{a\in X_1}\inf_{b\in X_2}d(a,b),\sup_{b\in X_2}\inf_{a\in X_1}d(a,b)\right\}.
\]
Heuristically, the Hausdorff distance is the radius at which all balls centered at points in $X_1$ intersect $X_2$ and vice versa, which is why we can expect it to approximate the high end of the range of distances at which births and deaths in $X_1, X_2$, and $X$ affect each other. We illustrate this in Example \ref{ex:synthetic}.

\begin{example}\label{ex:synthetic}  

Let $X_1$ and $X_2$ be the green and red classes depicted in Figure \ref{fig:4ptscirc}, consisting of 4 points each, equidistributed around a circle of radius 1, rotated away from each other by $\pi/4$.

The Hausdorff distance from each class to the union is $H\coloneqq\sqrt{2-\sqrt{2}}\approx 1.53073$. Generalized landscapes were computed with the weights on the edges $\text{red}\preccurlyeq \text{(red}\cup \text{green)}$, $\text{green}\preccurlyeq \text{(red}\cup \text{green)}$ increasing from $0.2H$ to $0.9H$ in increments of $0.1H$. Since the situation for both classes in symmetric, we are displaying only the landscape for the red class in Figure \ref{fig:incrH}, with weights increasing from left to right.

\begin{figure}
\includegraphics[width=0.12\textwidth]{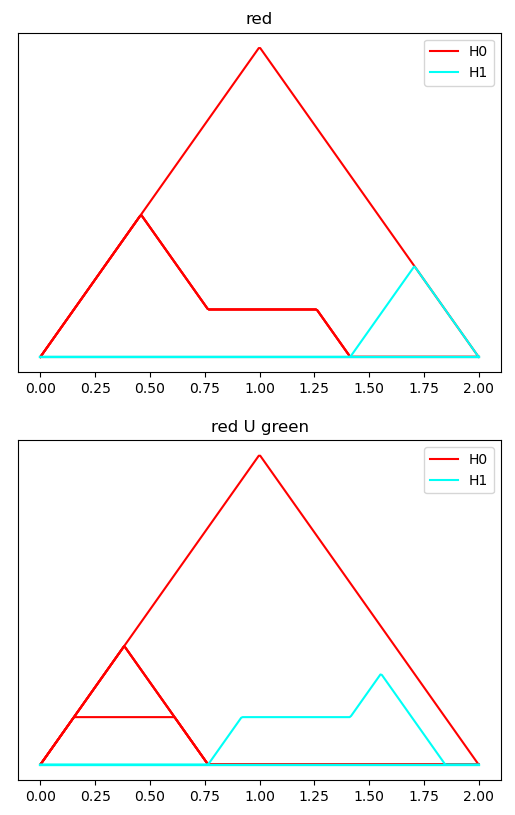}
\includegraphics[width=0.12\textwidth]{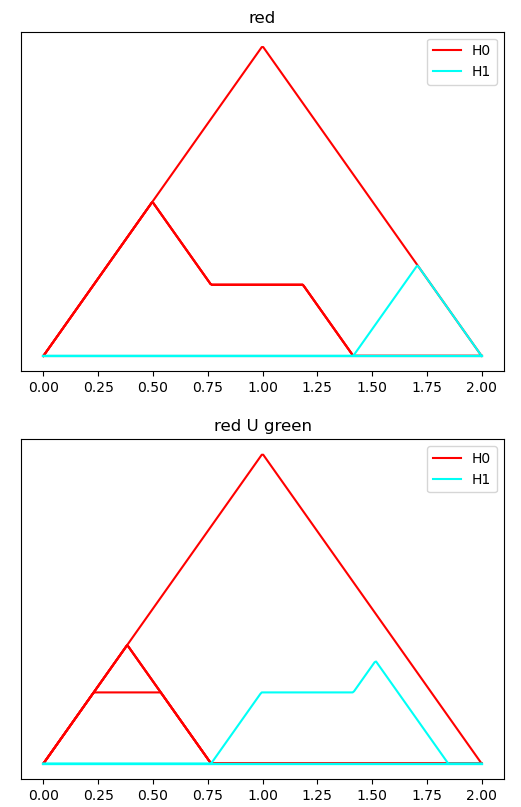}
\includegraphics[width=0.12\textwidth]{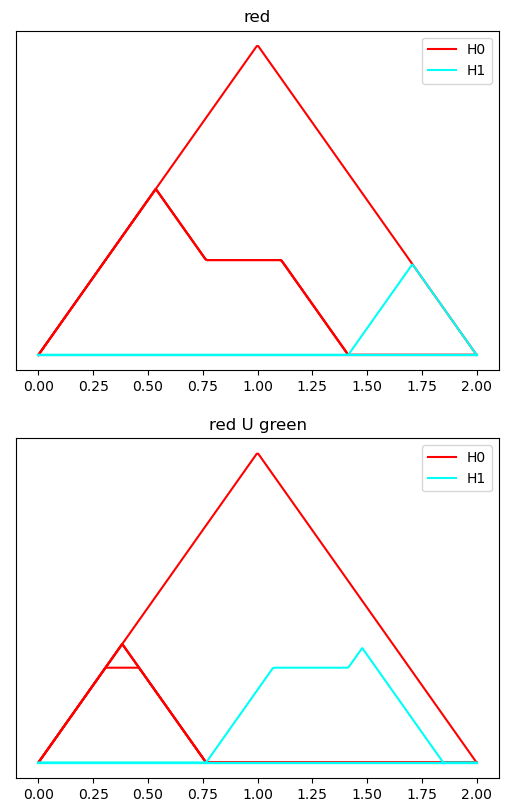}
\includegraphics[width=0.12\textwidth]{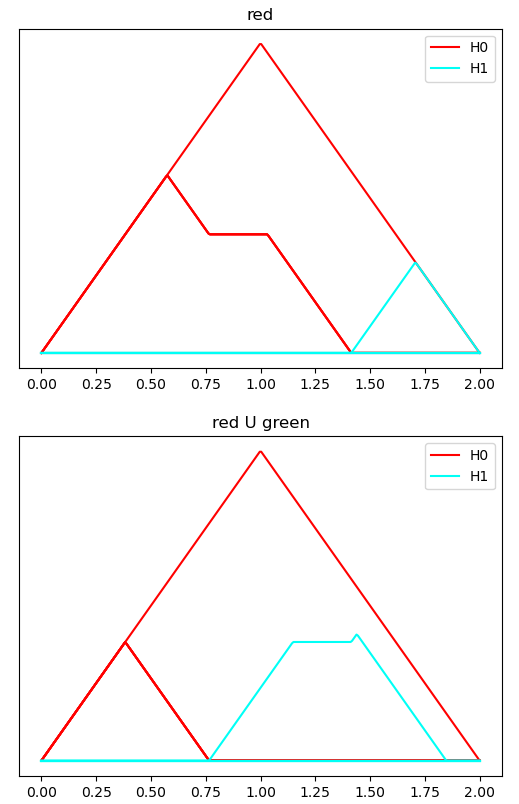}
\includegraphics[width=0.12\textwidth]{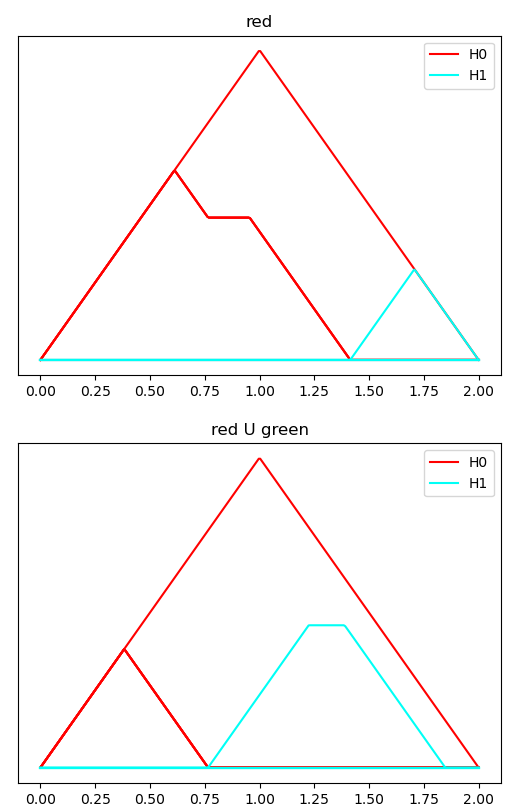}
\includegraphics[width=0.12\textwidth]{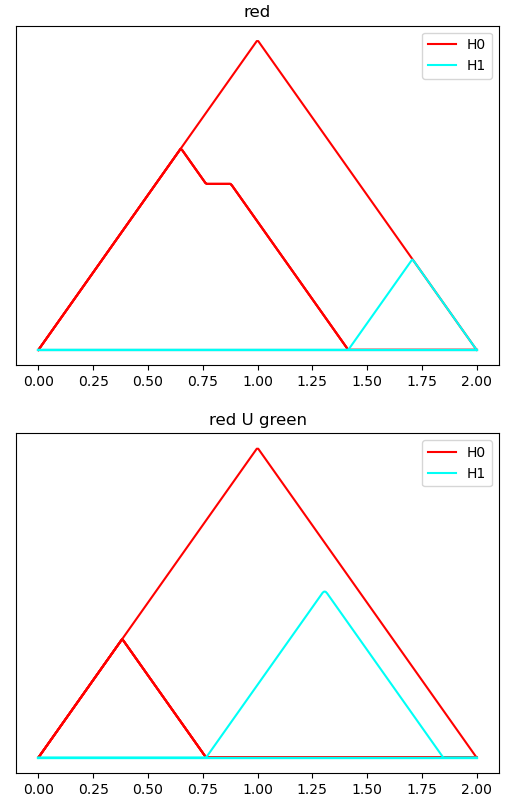}
\includegraphics[width=0.12\textwidth]{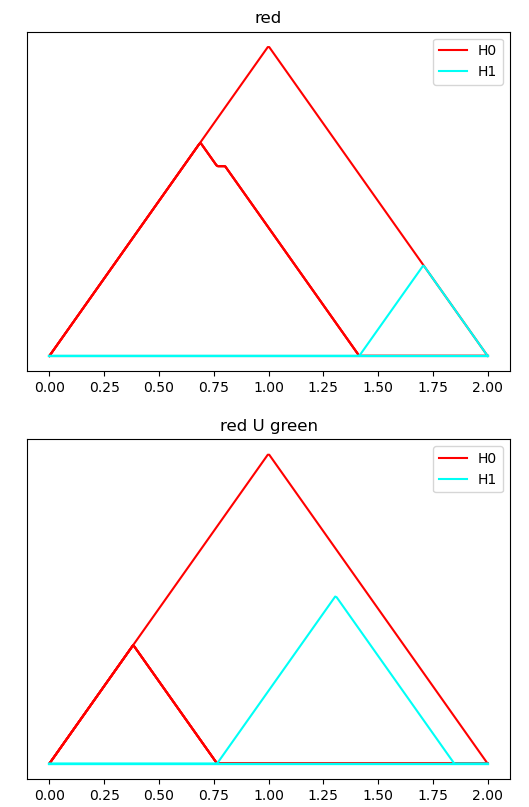}
\includegraphics[width=0.12\textwidth]{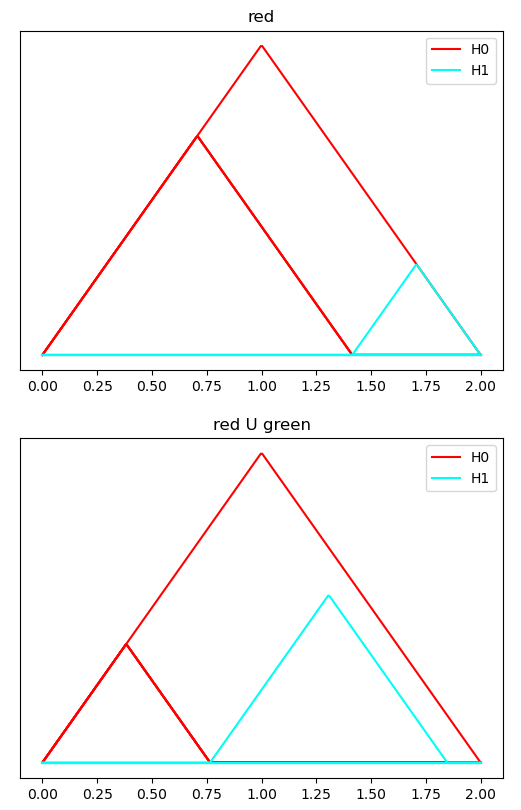}
\caption{Generalized landscapes for $\mathrm{LPH}_0$ (red) and $\mathrm{LPH}_1$ (cyan) of the red class and union of classes from Figure \ref{fig:4ptscirc}, with $d_{P_2}$ increasing from left to right by increasing all edge weights from one-fifth to nine-tenths of the Hausdorff distance between the red and green classes.}
\label{fig:incrH}
\end{figure}

We note that as the distance between each class and the union increases, the flat regions in both $\mathrm{LPH}_0$ and $\mathrm{LPH}_1$ appear at increasing height, until, for  sufficiently large distance between classes and the union, we see the landscape function only of the class (union, respectively) point cloud itself. This observation is in line with Lemma~\ref{lem:dPrange}.

\end{example}

Next we complicate the correspondence between landscape slopes and features in the labeled persistent homology of $X$ by adding noise to our point clouds.

\begin{example}\label{ex:noise}

In this example, we compare the effects of noise on the landscape function. In the first two-class point cloud, depicted in Figure~\ref{fig:noise}(a), the classes (red and green) have 15 points each and the points are equidistributed around a circle of radius 1, rotated away from each other by $\pi/15$. The second two-class point cloud, in Figure~\ref{fig:noise}(b), has a similar set up, but one of the classes (green) has noise added to it.

\begin{figure}
    \centering
\begin{subfigure}[b]{0.3\textwidth}
    \centering
    \includegraphics[width=\textwidth]{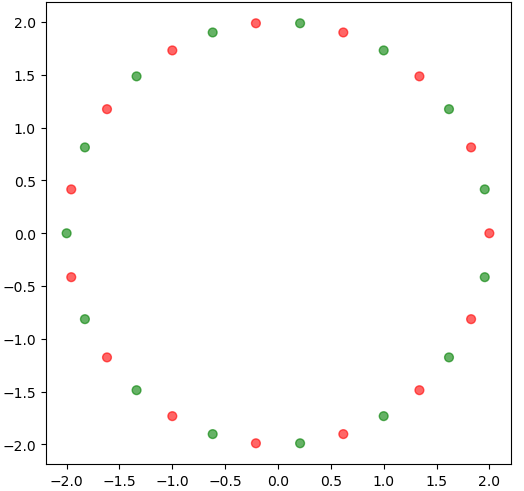}
    \caption{}
    \label{fig:noisea}
\end{subfigure}
\hspace{1cm}
\begin{subfigure}[b]{0.3\textwidth}
\centering
\includegraphics[width=\textwidth]{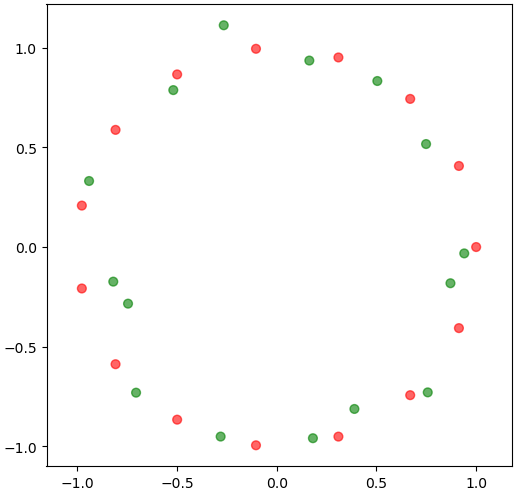}
\caption{}
\label{fig:noiseb}
\end{subfigure}
    \caption{The clean (in (a)) and noisy (in (b)) 2-labeled metric spaces described in Example \ref{ex:noise}.}
    \label{fig:noise}
\end{figure}

In Figure \ref{fig:noiseland} below we display the generalized landscapes of these point clouds. The weights on the edges from each individual class to the union in $P_2$ are set to be half of the Hausdorff distance from the individual class' point cloud to the union point cloud.

\begin{figure}[htp]
    \centering
    \begin{subfigure}[b]{0.3\textwidth}
    \centering
        \includegraphics[width=\textwidth]{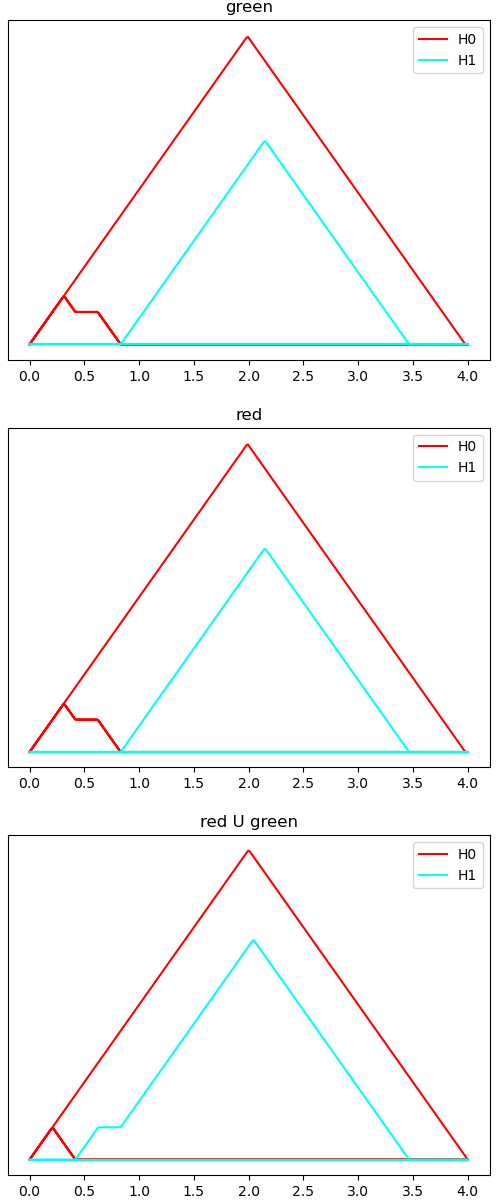}
        \caption{}
        \label{fig:noiselanda}
    \end{subfigure}
\hspace{1cm}
\begin{subfigure}[b]{0.3\textwidth}
\centering
\includegraphics[width=\textwidth]{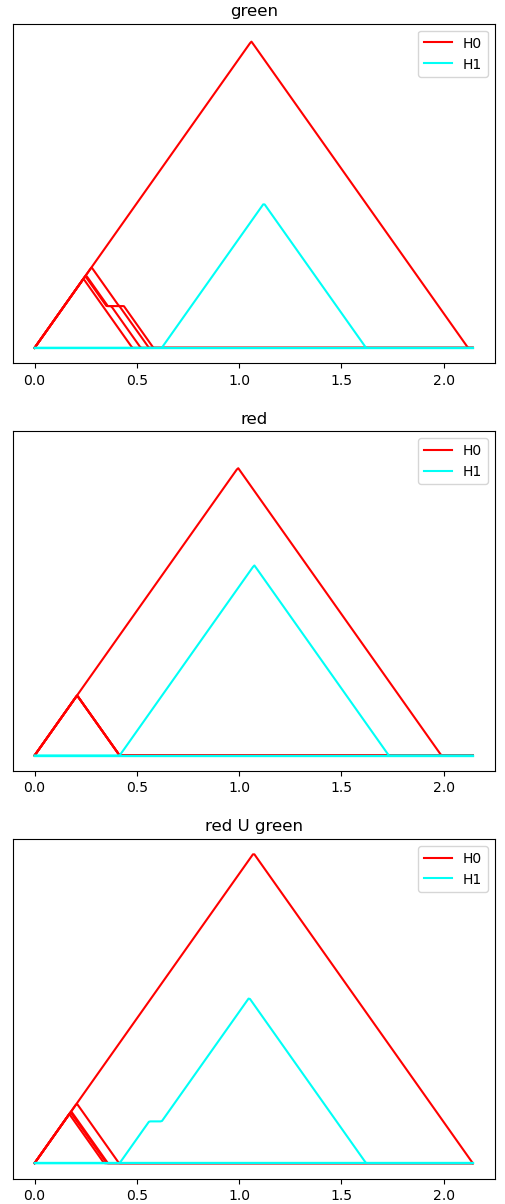}
\caption{}
\label{fig:noiselandb}
\end{subfigure}
\caption{Generalized landscapes for $\mathrm{LPH}_0$ (red) and $\mathrm{LPH}_1$ (cyan) of the green class, red class, and union of classes from Figure \ref{fig:noise} (again, the noisy example is on the right). The weight on an edge in $P_2$ equals half the Hausdorff distance between the sets corresponding to the endpoints of the edge. Note that when noise is added, the generalized landscape appears to indicate more $\mathrm{LPH}_0$ components; this is only because without noise, these components all have equal generalized landscape (thus their plots overlap in the output of our code) due to the symmetry of the point cloud.
}
    \label{fig:noiseland}
\end{figure}
In the case without noise (Figure \ref{fig:noiselanda}), the prominent cycle in $\mathrm{LPH}_1$ of the each individual class is born later than in the union of the classes (the birth times can be read off at the $r$-intercepts of the large cyan triangle indicating this cycle in each graph), leading to the flat region in the landscape of the union. The later birth happens because the union is more densely filled than either class individually. In the case with noise (Figure \ref{fig:noiselandb}), we can see that the flat region in $\mathrm{LPH}_1$ of the union arises because only the prominent cycle of the green class is born later. The flat region occurs at the height of the distance between the green circle and the union in the poset.

In the green class (in both cases), we observe an $\mathrm{LPH}_0$ cycle (a component) which dies earlier in the union than in the green class itself, leading to the flat region in the landscape of the class.
\end{example}

\subsubsection{Simple Real Example}\label{ss:realcode}

To test how these methods can be used on real-world data, we perform an ``entanglement'' experiment on MNIST data \cite{cohen2017emnist} similar to that performed in \cite{mixup}. The MNIST dataset found within the expanded EMNIST dataset \cite{cohen2017emnist} is a classic example dataset used throughout machine learning, and consists of $28 \times 28$ grayscale images of handwritten digits from 0 through 9. We choose 50 random digits from each class and perform a pairwise analysis on each of the 45 choices of pairs with the intent of seeing if we can detect which pairs of digits have their homology most affected by using class-aware landscapes, and if so, how.

In this experiment, we want to see how ``entangled'' each pair of digits are as point clouds in high-dimensional Euclidean space. Treating these grayscale images as flattened vectors means we lose some spatial adjacency information, but this flattening operation is common in machine learning, and is a reasonable first step in such an analysis, while more sophisticated distance measures could be used in follow-on studies if desired. 

The $28 \times 28$ images in the MNIST dataset are provided as flattened vectors in $\R^{784}$ so we begin by choosing a pair of digits. To obtain the class-aware landscape, we compute the $\R \times P_2$-landscape, then extract the union portion of that landscape for analysis. To compute the $\R \times P_2$-landscape, we choose to use inter-class distances equal to $0.1$ times the Hausdorff distance between the relevant point clouds. To obtain the class-naive landscape, we simply treat the point cloud of the samples for the two digits chosen as an unclassified point cloud and compute the standard landscape. To get a single numerical value that could represent the distance between these landscapes, we then compute the $L^2$ distance (or mean square error (MSE)) between the resulting landscapes (with the class-naive landscapes interpolated to be sampled at the same filtration values as the class-aware landscapes to ensure a fair comparison).

\begin{figure}[htp]
    \centering
    \includegraphics[width=0.5\linewidth]{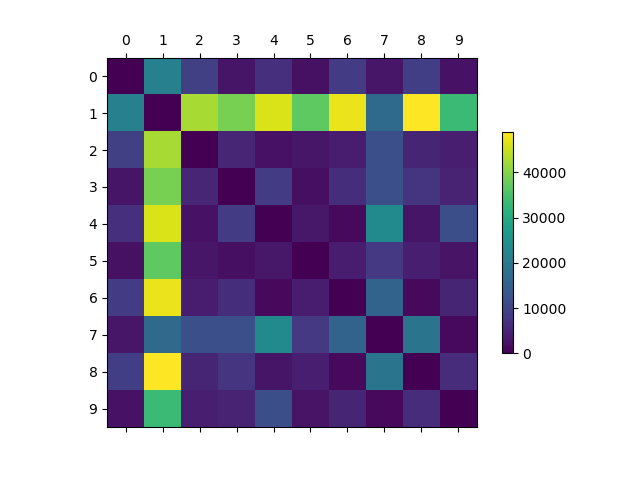}
    \caption{Pairwise MSE between the class-aware and class-naive landscapes for each pair of digits in the MNIST dataset.}
    \label{fig:mnist_pairwise_mse}
\end{figure}

The complete set of pairwise MSEs is shown in Figure \ref{fig:mnist_pairwise_mse}. One notable difference between our analysis and the similar analysis in \cite{mixup} is that because our construction is symmetric, the resulting MSE matrix is symmetric as well.

\begin{figure}[htp]
    \centering
    \includegraphics[width=0.8\linewidth]{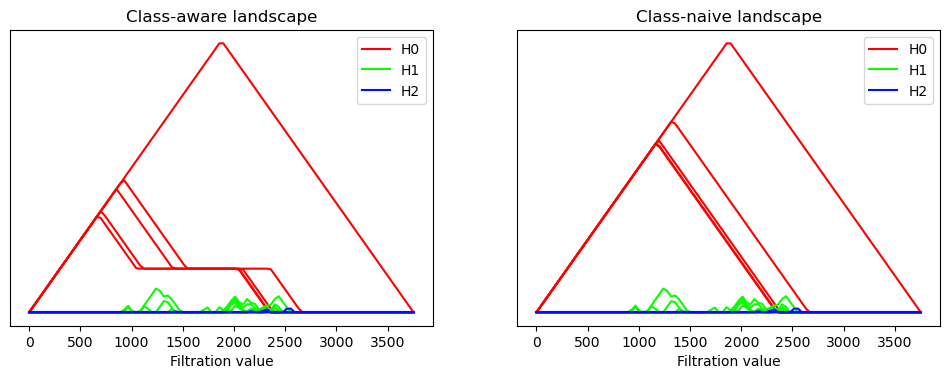}
    \caption{The class-aware and class-naive landscapes for the digits 1 versus 8. The MSE for this case is 49031.08.}
    \label{fig:mnist_1_vs_8}
\end{figure}

In analyzing Figure \ref{fig:mnist_pairwise_mse}, the most striking feature is that the MSE values for 1 vs.\ every other digit is much higher than the other pairwise differences, which is consistent between multiple runs of this analysis and multiple random samples of digits. To see why this might be, we can examine the relevant landscapes for the most extreme cases. In Figure \ref{fig:mnist_1_vs_8} we see the class-aware and class-naive landscapes for the 1 vs.\ 8 case, which has an MSE of $49031.08$, the highest amongst all comparisons. The most notable feature of these landscapes is that several of the $LPH_0$ features in the class-aware landscape descend to flat regions later in their lifetimes. Our interpretation of this is that in fact 1 and 8 are quite disentangled---there are multiple strong $LPH_0$ components that are present in the union, but those components are not present in one or the other of the single classes, so they die off at the fixed inter-class distance, which we would not expect to be true if the classes were well-mixed. The landscapes for other pairs involving the 1 digit exhibit similar (although slightly less extreme) behavior.

\begin{figure}[htp]
    \centering
    \includegraphics[width=0.8\linewidth]{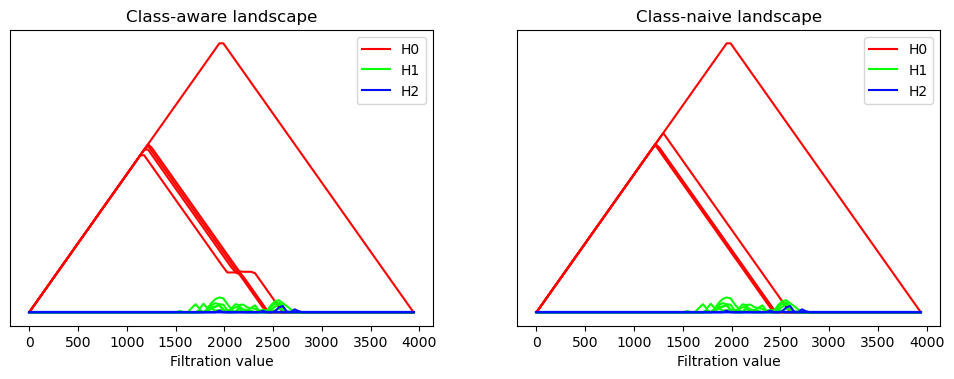}
    \caption{The class-aware and class-naive landscapes for the digits 4 vs. 6. The MSE for this case is 1182.22.}
    \label{fig:mnist_4_vs_6}
\end{figure}

To see this contrast, we also show the class-aware and class-naive landscapes for the $4$ vs $6$ case, which has an MSE of 1182.22, the lowest amongst all comparisons. In this case, while the class-aware landscape has a small amount of the same behavior (with a flat portion on the descending slope of the main $LPH_0$ components), it is far less extreme. This suggests that the homology for the digits separately is fairly similar to the homology for the union and that the point clouds for these classes are intermingled.

While this degree of interpretation is promising, the current computational scheme for $\R \times P_k$-landscapes leaves much to be desired in terms of speed. While it is exact at each point in the discretization chosen, the number of separate 1-dimensional landscape computations needed grows rapidly as the number of classes increases (see Subsection \ref{ss:algorithms} for more details) and the computation for each path increases in complexity as the number of points increases. While the latter problem can be addressed to some extent via algorithmic and coding improvements, the number of paths needed is a combinatorial requirement, and is why we only perform pairwise comparisons in these experiments on MNIST data. We view the exploration of finding approximating sets of paths as a promising avenue for future work.

\section{Discussion}

We view the main contributions of this paper to be  theoretical; namely, the formalization of the notion of labeled datasets, and the introduction of various metrics for comparing them, based on Gromov-Hausdorff distance and persistent homology. Our stability results relate these ideas and give a rigorous justification for the usage of persistent homology techniques in this setting. We also give a more tractably computable approach to comparing labeled datasets via our theoretically-grounded landscape invariants. However, as was observed in \S\ref{ss:realcode}, scalability of the present methods to larger datasets poses a serious issue, so that our current numerical pipeline should be viewed as an important first step toward computation in this setting. Future work will be especially focused on methodological and algorithmic improvements for overcoming this scalability issue. This will enable more serious applications of our framework to real data; in particular, an initial goal of this project (which has not yet been achieved) was to use these ideas to explore the topology of latent space embeddings, inspired by the recent paper~\cite{wayland2024mapping}. We are also interested in extending the Gromov-Hausdorff constructions of this paper to the measured setting, yielding a notion of Gromov-Wasserstein distances between labeled metric measure spaces.

\section*{Data Availability Statement}
The MNIST data used in \S\ref{ss:realcode} is available from the National Institute of Standards and Technology at \url{https://www.nist.gov/itl/products-and-services/emnist-dataset}. The code used to perform all computations, including example scripts, is available at \url{https://github.com/zyjux/classwise_tda}.

\section*{Acknowledgements}

This work originated at the American Mathematical Society (AMS) Mathematical Research Communities (MRC) workshop on \textit{Mathematics of Adversarial, Interpretable, and Explainable AI}, held in June 2024. We thank the AMS, as well as the workshop organizers and staff, for providing a stimulating and productive research environment. This material is based upon work supported by the National Science Foundation under Grant Number DMS 1916439. SL was partially supported by NSF grant DMS 2410140. TN was partially supported by NSF grants DMS 2324962 and CIF 2526630.  LV's work on this project was supported by two NSF grants: AI Institute grant RISE-2019758 (AI2ES) and CAIG grant RISE-2425923.
EL was partially supported by the Deutsche Forschungsgemeinschaft (DFG, German Research
Foundation) under Germany's Excellence Strategy – The Berlin Mathematics
Research Center MATH+ (EXC-2046/1, project ID: 390685689).

\bibliography{refs.bib}{}
\bibliographystyle{plain}

\appendix

\section{Worked Examples}\label{app:calc}

We present three fully computed examples with $k=2$. The third consists of the details needed to support Example \ref{ex:byhand1}. In each, we compute the entire ($\R\times P_2$)-module. Therefore, in order to depict it on the page, we restrict from $P_2$ to $P_2\setminus\{\emptyset\}=X_1\to X\leftarrow X_2$, and we lose no information because $\lambda(n,(r,\emptyset))\equiv0$.

For an interval (open, closed, or half-open) $I\subset\R$ that may only have critical filtration values (values at which Betti numbers change) at its endpoints, we use $\mathrm{LPH}_j(I,S)$ to indicate any of the isomorphic homology groups $\mathrm{LPH}_j(r,S)$ for a parameter $r\in I$.

Each of our examples follows the same outline:
\begin{enumerate}
    \item For $S=X_1, X_2$, and $X$, calculate the vector spaces $\mathrm{LPH}_j(I,S)$ for all half-open intervals $I$ starting and ending at the union of critical filtration values for $\VR(X_1), \VR(X_2)$, and $\VR(X)$ (and thus containing no critical filtration values for any of these three $\VR$-complexes in their interiors), as well as $I=(-\infty,0)$ and $I=[r_{max},\infty)$, where $r_{max}$ is the largest critical filtration value across all three complexes.
    \item Calculate all maps $\mathrm{LPH}_j(r,S)\to \mathrm{LPH}_j(r',S')$ induced by the inclusions of simplicial complexes $\VR_r(S)\subseteq\VR_{r'}(S')$ arising when $r\leq r'$ and $S\subseteq S'$.
    \item Arrange these vector spaces and maps into a poset indicating the image of the $\R\times P_2$-module $(r,S)\mapsto \mathrm{LPH}_j(r,\cdot)$. Here is where our concise notation $\mathrm{LPH}_j(I,S)$ (rather than $\mathrm{LPH}_j(r,S)$ for each $r\in I$) becomes useful.
    \item Calculate $\lambda_V(n,(r,S))$ for $S=X_1, X_2$, and $X$ in turn. Given $n$ and $S$, we do this by splitting our analysis based on the interval $I\ni r$, and for each $r$, we consider the minimal\footnote{Here, ``minimal" means that once we have analyzed a given map, we are free to ignore any maps factoring through it.} set of maps factoring through $\mathrm{LPH}_j(r,S)$ of rank less than $n$. Each such map induces a lower bound on the value of $\epsilon$ needed to contain both endpoints of (the preimage under the functor from $\R\times P$ to vector spaces of) that map within the $\epsilon$-ball in $\R\times P$ around $(r,S)$. We record these upper bounds, ignoring any repeats; the minimum upper bound is the value of $\lambda_V(n,(r,S))$.
\end{enumerate}

In the first example, we repeat and re-explain these steps, and fix notation used throughout; the logic of the second example is almost identical to the first; finally, the logic of the third example, which requires many more calculations, is summarized tersely.

\begin{example}[Two one-point classes; $H_0$]
Let $X_1=\{(0,0)\}$ and $X_2=\{(0,1)\}$. Then
    \[
    \mathrm{LPH}_0((-\infty,0),X_1)=0, \quad \mathrm{LPH}_0([0,\infty),X_1)=\mathbb{F},
    \]
    and the homology of the Vietoris-Rips complex of $X_2$ is the same. We also have
    \[
    \mathrm{LPH}_0((-\infty,0),X)=0, \quad \mathrm{LPH}_0([0,1),X)=\mathbb{F}^2, \quad \mathrm{LPH}_0([1,\infty),X)=\mathbb{F}.
    \]
    We now arrange these homology groups into a poset diagram indicating the image of the $\R\times P_2$-module $(r,S)\mapsto \mathrm{LPH}_0(r,S)$ and $(r,S)\preccurlyeq(r',S')$ to the inclusion-induced map on homology. For further conciseness, we use $V^I_S$ to denote any of the isomorphic homology groups $\mathrm{LPH}_0(r,S)$ for $r\in S$. We use $V^r_S$ to specify a single parameter $r$. We will also use the convention that, when $I\subseteq[d,\infty)$ for the diameter $d$ of $S$, we set $V^I_S=0$ (so that the landscape values for $r\geq d$ will be zero).
    \[
    \begin{tikzcd}
        V^{(-\infty,0)}_{X_1}=0 \arrow[r] \arrow[d] & V^{[0,1)}_{X_1}=0 \arrow[r, "id"] \arrow[d, "\begin{pmatrix}1\\0\end{pmatrix}"] & V^{[1,\infty)}_{X_1}=0 \arrow[d]
        \\V^{(-\infty,0)}_{X}=0 \arrow[r] & V^{[0,1)}_{X}=\mathbb{F}^2 \arrow[r, "(1 \;\; 1)"] & V^{[1,\infty)}_{X}=0
        \\V^{(-\infty,0)}_{X_2}=0 \arrow[r] \arrow[u] & V^{[0,1)}_{X_2}=0 \arrow[r, "id"] \arrow[u, "\begin{pmatrix}0\\1\end{pmatrix}"] & V^{[1,\infty)}_{X_2}=0 \arrow[u]
    \end{tikzcd}
    \]
    Recall that in the sum extended quasimetric, we have $d((r,S)\to (r',S))=r'-r$. Set $d_P((r,{1})\to(r,{1,2}))=:d_1$ and $d_P((r,{2})\to(r,{1,2}))=:d_2$. The length of a path is the sum of its edge weights, because when $k=2$ each path can only contain one edge in $P_2$, so the difference between the geodesic and ultrametric distances is not significant. We now calculate the $P$-landscape in three parts. To save on notation, we denote by $\lambda_A(n,r)$ the the $\R\times P_2$-landscape of $\mathrm{LPH}_0(r,\cdot)$ restricted to $\R\times\{X_1\}$, that is, $\lambda_{X_1}(n,r)\coloneqq\lambda_V(n,(r,X_1))$, and similarly for $\lambda_{X_2}$ and $\lambda_{X}$. First note that $\lambda_{X_1}(n,r)=\lambda_{X_2}(n,r)\equiv0$, because all vector spaces are zero.
        
        To calculate $\lambda_{X}$, we only need to consider $r\in[0,1)$. Qualitatively, there are six paths through $(r,\{1,2\})$, where in the list below, for $r\neq0$, $0\leq r'<r$, each corresponding under $V$ to one of the compositions of maps listed below. We provide the bounds on $\epsilon$ determined by the maps corresponding to the shortest possible paths when that map appears for the first time.
        \begin{itemize}
            \item $V^{(-\infty,0)}_{X_1}\to V^r_{X_1}\to V^r_{X}\to V^{[1,\infty)}_{X}$: the map $V^r_{X_1}\to V^r_{X}$ is rank zero, and is included for $\epsilon\geq d_1$; the map $V^r_{X}\to V^{[1,\infty)}_{X}$ is rank zero, and is included for $\epsilon\geq 1-r$.
            \item $V^{(-\infty,0)}_{X_1}\to V^{r'}_{X_1}\to V^{r'}_{X}\to V^r_{X}\to V^{[1,\infty)}_{X}$: the map $V^{r'}_{X_1}\to V^{r'}_{X}$ is rank zero, and included for $\epsilon\geq r-r'+d_1$.\footnote{From now on we exclude paths such as $(r,{1})\preccurlyeq(r',\{1,2\})\preccurlyeq(r,\{1,2\})$ when $V^{r'}_{X_1}\to V^{r'}_{X}$ is the same map between the same modules as $V^r_{X_1}\to V^r_{X}$, as the rank $\leq n$ part of the map always occurs farther away from $(r,\{1,2\})$ and so lead to weaker, duplicate bounds on $\epsilon$.}
         \item $V^{(-\infty,0)}_{X_1}\to V^{(-\infty,0)}_{X}\to V^r_{X}\to V^{[1,\infty)}_{X}$: the map $V^{(-\infty,0)}_{X}\to V^r_{X}$ is rank zero, and is included for $\epsilon\geq r$.
        \end{itemize}
        (We list only those starting at $V_{X_1}^{(-\infty,0)}$; those starting at $V_{X_2}^{(-\infty,0)}$ are analogous.
        
        Therefore if $r\in[0,1)$,
        \[
        \lambda_{X}(1,r)=\lambda_{X}(2,r)=\min\{d_1,d_2,1-r,r-r'+d_1,r-r'+d_2,t\}=\min\{d_1,d_2,1-t,t\},
        \]
        because $r-r'\geq0$, therefore $r-r'+d_1\geq d_1$ (and similarly for $d_2$). 
        Otherwise, $\lambda_{X}(n,r)=0$.
    \end{example}

    \begin{example}[Two two-point classes; $H_1$] Let $X_1=\{(0,0),(1,0)\}$ and let $X_2=\{(0,1),(1,1)\}$. Then
    \[
    \mathrm{LPH}_1((-\infty,\infty),X_1)=\mathrm{LPH}_1((-\infty,\infty),X_2)=0,
    \]
    while
    \[
    \mathrm{LPH}_1((-\infty,1),X)=0, \quad \mathrm{LPH}_1([1,\sqrt{2}),X)=\mathbb{F}, \quad \mathrm{LPH}_1([\sqrt{2},\infty),X)=0.
    \]
    The image of the $\R\times P$-module is extremely simple:
     \[
    \begin{tikzcd}
        V^{(-\infty,1)}_{X_1}=0 \arrow[r] \arrow[d] & V^{[1,\sqrt{2})}_{X_1}=0 \arrow[r] \arrow[d] & V^{[\sqrt{2},\infty)}_{X_1}=0 \arrow[d]
        \\V^{(-\infty,1)}_{X}=0 \arrow[r] & V^{[1,\sqrt{2})}_{X}=\mathbb{F} \arrow[r] & V^{[\sqrt{2},\infty)}_{X}=0
        \\V^{(-\infty,1)}_{X_2}=0 \arrow[r] \arrow[u] & V^{[1,\sqrt{2})}_{X_2}=0 \arrow[r] \arrow[u] & V^{[\sqrt{2},\infty)}_{X_2}=0 \arrow[u]
    \end{tikzcd}
    \]
    The same argument as in the previous example shows that $\lambda_{X_1}(k,r)=\lambda_{X_2}(k,r)\equiv0$, and
    \[
    \lambda_{X}(1,r)=\min\{d_1,d_2,\sqrt{2}-r,r-1\}
    \]
    for $r\in[1,\sqrt{2})$ and otherwise $\lambda_{X}(n,r)=0$.
    \end{example}
        
    \begin{example}[Two three-point classes; $H_0$]
    Let 
    \[
    X_1=\{(0,0),(0.4,0),(1,0)\} \quad \mbox{and} \quad X_2=\{(0,1),(0.4,1),(1,1)\}.
    \]
    We have
    \[
    \mathrm{LPH}_0((-\infty,0),X_1)=0, \quad \mathrm{LPH}_0([0,0.4),X_1)=\mathbb{F}^3, \quad \mathrm{LPH}_0([0.4,0.6),X_1)=\mathbb{F}^2, \quad \mathrm{LPH}_0([0.6,\infty),X_1)=\mathbb{F},
    \]
    and the homology of the Vietoris-Rips complex of $X_2$ is the same. We also have
    \begin{multline*}
        \mathrm{LPH}_0((-\infty,0),X)=0, \quad \mathrm{LPH}_0([0,0.4),X)=\mathbb{F}^6, \quad \mathrm{LPH}_0([0.4,0.6),X)=\mathbb{F}^4,
        \\ \mathrm{LPH}_0([0.6,1),X)=\mathbb{F}^2, \quad \mathrm{LPH}_0([1,\infty),X)=\mathbb{F}.
    \end{multline*}
    The poset diagram is below (with $V^{(-\infty,0)}_S$ and $V^{[d,\infty)}_S$ labels elided for space, where $d$ is the diameter of $S$, since both are zero).
    \[
    \begin{tikzcd}
        0 \arrow[r] \arrow[d] & V^{[0,0.4)}_{X_1}=\mathbb{F}^3 \arrow[r, "\mathbf{V}_{X_1}^{0.4}"] \arrow[d, "\begin{psmallmatrix}I_3\\0\end{psmallmatrix}"] & V^{[0.4,0.6)}_{X_1}=\mathbb{F}^2 \arrow[d, "\begin{psmallmatrix}I_2\\0\end{psmallmatrix}"] \arrow[r, "\mathbf{V}_{X_1}^{0.6}"] & V^{[0.6,1)}_{X_1}=\mathbb{F} \arrow[r] \arrow[d, "\begin{psmallmatrix}1\\0\end{psmallmatrix}"] & 0 \arrow[d]\arrow[r] & 0 \arrow[d]
        \\0 \arrow[r] & V^{[0,0.4)}_{X}=\mathbb{F}^6 \arrow[r, "\mathbf{V}_{X}^{0.4}"] & V^{[0.4,0.6)}_{X}=\mathbb{F}^4 \arrow[r, "\mathbf{V}_{X}^{0.6}"] & V^{[0.6,1)}_{X}=\mathbb{F}^2 \arrow[r, "\mathbf{V}_{X}^1"] & V^{[1,\sqrt{2})}_{X}=\mathbb{F} \arrow[r] & 0
        \\0 \arrow[r] \arrow[u] & V^{[0,0.4)}_{X_2}=\mathbb{F}^3 \arrow[r, "\mathbf{V}_{X_2}^{0.4}"] \arrow[u, "\begin{psmallmatrix}0\\I_3\end{psmallmatrix}"] & V^{[0.4,0.6)}_{X_2}=\mathbb{F}^2 \arrow[r, "\mathbf{V}_{X_2}^{0.6}"] \arrow[u, "\begin{psmallmatrix}0\\I_2\end{psmallmatrix}"] & V^{[0.6,1)}_{X_2}=\mathbb{F} \arrow[r] \arrow[u, "\begin{psmallmatrix}0\\1\end{psmallmatrix}"] & 0 \arrow[u]\arrow[r] & 0 \arrow[u]
    \end{tikzcd}
    \]
    Now set $\mathbf{V}^{\min I'}_S\coloneqq V((r,S)\preccurlyeq(r',S))$ for $r\in I, r'\in I'$ with $\sup I=\min I'$. Specifically, the maps are:
    \begin{align*}
        \mathbf{V}_{X_1}^{0.4}=\mathbf{V}_{X_2}^{0.4}&=\begin{pmatrix}1&1&0\\0&0&1\end{pmatrix}
        \\\mathbf{V}_{X_1}^{0.6}=\mathbf{V}_{X_2}^{0.6}&=\begin{pmatrix}1&1\end{pmatrix}
        \\\mathbf{V}_{X}^{0.4}&=\mathbf{V}_{X_1}^{0.4}\oplus \mathbf{V}_{X_2}^{0.4}
        \\\mathbf{V}_{X}^{0.6}&=\mathbf{V}_{X_1}^{0.6}\oplus \mathbf{V}_{X_2}^{0.6}
        \\\mathbf{V}_{X}^1&=\begin{pmatrix}1&1\end{pmatrix}.
    \end{align*}

    First we calculate $\lambda_{X_1}$; the calculation for $\lambda_{X_2}$ is analogous. For $\lambda_{X_1}(1,r)$, there are three intervals to consider. We list them along with the shortest maps through $(r,\{1\})$ for $r$ in that interval with rank zero (and providing a new bound).
    \begin{itemize}
        \item $r\in[0,0.4)$: $0\to V^r_{X_1}$ has rank zero, and is included if $\epsilon\geq r$; $V^r_{X_1}\to V^{r'}_{X}$ has rank zero for $r'\geq\sqrt{2}$, and is included if $\epsilon\geq d_1+\sqrt{2}-r$; $V^r_{X_1}\to 0$ has rank zero, and is included if $\epsilon\geq 1-r$.
        \item $r\in[0.4,0.6)$ and $r\in[0.6,1)$ provide no new bounds.
    \end{itemize}
    Therefore
    \[
    \lambda_{X_1}(1,r)=\lambda_{X_2}(1,r)=\min\{r,d_1+\sqrt{2}-r,1-r\}=\min\{r,1-r\}.
    \]

    For $\lambda_{X_1}(2,r)$, we have:
    \begin{itemize}
        \item $r\in[0,0.4)$: $0\to V^r_{X_1}$ has rank zero, and is included if $\epsilon\geq r$; $V^r_{X_1}\to V^{r'}_{X_1}$ has rank one for $r'\geq0.6$, and is included if $\epsilon\geq 0.6-r$; $V^r_{X_1}\to V^{r'}_{X}$ has rank one for $r'\in[0.6,1)$, and is included if $\epsilon\geq d_1+0.6-r$.
        \item $r\in[0.4,0.6)$: no new bounds are added.
        \item $r\in[0.6,1)$: here $\lambda_{X_1}(2,r)\equiv0$ as the rank of $V^r_{X_1}$ is one.
    \end{itemize}
    Therefore
    \[
    \lambda_{X_1}(2,r)=\lambda_{X_2}(2,r)=\min\{r,0.6-r,d_1+0.6-r\}=\min\{r,0.6-r\},
    \]
    By a similar calculation,
    \[
    \lambda_{X_1}(3,r)=\lambda_{X_2}(3,r)=\min\{r,0.4-r\},
    \]
    and all other $\lambda_{X_1}(n,r)$ and $\lambda_{X_2}(n,r)$ are zero.

    Next we turn to $\lambda_{X}$. First, as in the previous two examples, we have $\lambda_{X}(1,r)=\min\{d_1+r,d_2+r,\sqrt{2}-r,r\}=\min\{\sqrt{2}-r,r\}$. Next, we consider $\lambda_{X}(2,r)$. Again, we will consider only paths going through the $X_1$ module, and obtain analogous bounds on $\epsilon$ via the symmetric paths through the $X_2$ module.
    \begin{itemize}
        \item $r\in[0,0.4)$: $0\to V^r_{X_1}\to V^r_{X}$ has rank zero, and is included if $\epsilon\geq d_1+r$; $V^r_{X}\to V^{r'}_{X}$ has rank one or zero if $r'\geq 1$, and is included if $\epsilon\geq 1-r$; $0\to V^{[0,0.4)}_{X}$ has rank zero, and is included if $\epsilon\geq r$. There is also one more complicated map: $V^r_{X_1}\to V^r_{X}\to V^{r'}_{X}$ is rank one for $r'\geq0.6$, and is included if $\epsilon\geq\max\{d_1,0.6-r\}$.
        \item $r\in[0.4,0.6)$: no new bounds are added.
        \item $r\in[0.6,1)$: now the map $V^r_{X_1}\to V^r_{X}$ is itself rank one, and is included if $\epsilon\geq d_1$.
    \end{itemize}
    Therefore (simplifying by using $r\leq d_1+r$, $1-r<r$ when $r>0.5$, etc.), we have
    \[
    \lambda_{X}(2,r)=\begin{cases}
        \min\{1-r,r,\max\{d_1,0.6-r\},\max\{d_2,0.6-r\}\} & \text{ if $r\in[0,0.6)$, and}
        \\\min\{1-r,d_1,d_2\} & \text{ if $r\in[0.6,1)$, and }
        \\0 & \text{ else.}
    \end{cases}
    \]

    Next, we have $\lambda_{X}(3,r)$.
    \begin{itemize}
        \item $r\in[0,0.4)$: the bounds $\epsilon\geq d_1+r$ and $\epsilon\geq r$ arise for the same reasons as in the case of $n=2$, but now we also need to consider the map $V^r_{X}\to V^{r'}_{X}$ for $r'\geq0.6$, included when $\epsilon\geq0.6-r$, and the map $V^r_{X_1}\to V^{r'}_{X}$ for $r'\geq0.4$, which is rank at most two, and is included when $\epsilon\geq\max\{d_1,0.4-r\}$.
        \item $r\in[0.4,0.6)$: now the map $V^r_{X_1}\to V^r_{X}$ is rank two, and included if $\epsilon\geq d_1$. Note that $0.6-r\leq r$ in this range.
    \end{itemize}
    Therefore
    \[
    \lambda_{X}(3,r)=\begin{cases}
        \min\{0.6-r,r,\max\{d_1,0.4-r\},\max\{d_2,0.4-r\}\} & \text{ if $r\in[0,0.4)$, and}
        \\\min\{0.6-r,d_1,d_2\} & \text{ if $r\in[0.4,0.6)$, and }
        \\0 & \text{ else.}
    \end{cases}
    \]

    Next, to compute $\lambda_{X}(4,r)$, we have
    \begin{itemize}
        \item $r\in[0,0.4)$: the map $V^r_{X_1}\to V^r_{X}$ is rank three, and is included if $\epsilon\geq d_1$; the map $V^r_{X_1}\to V^{r'}_{X}$ for $r'\geq0.4$ is rank at most two, and is included if $\epsilon\geq\max\{d_1,0.4-r\}$; the map $V^r_{X}\to V^{r'}_{X}$ for $r'\geq0.6$ is rank two, and included if $\epsilon\geq0.6-r$; the map $0\to V^r_{X}$ is rank zero, and included if $\epsilon\geq r$.
        \item $r\in[0.4,0.6)$: there are no new bounds.
    \end{itemize}
    Therefore
    \[
    \lambda_{X}(4,r)=\min\{d_1,d_2,\max\{d_1,0.4-r\},\max\{d_2,0.4-r\},0.6-r,r\}=\min\{d_1,d_2,0.6-r,r\}
    \]
    when $r\in[0,0.6)$ and is zero otherwise.

    Finally, note that $\lambda_{X}(5,r)=\lambda_{X}(6,r)$ as there are no vector spaces or maps of rank five; since there are no maps of rank six, we have
    \[
    \lambda_{X}(6,r)=\lambda_{X}(6,r)=\min\{d_1,d_2,0.4-r,r\}
    \]
    when $r\in[0,0.4)$ and $\lambda_{X}(6,r)\equiv0$ otherwise. For $n>6$, all $\lambda_{X}(n,r)\equiv0$.

    See Figure \ref{fig:ex3} for the output of our code when $d_1=d_2=0.1$; the plots agree with our computations.
        
\end{example}

\end{document}

%% file: tikz/interleaving_2_tikz.tex
\begin{tikzcd}
{V(r,p)} \arrow[rr, "{V((r,p)\preccurlyeq (r+2\epsilon, p))}"] \arrow[rd, "{\phi_{r,p}}"] &                                                           & {V(r+2\epsilon, p)} &                                                                                              & {V(r+\epsilon, p)} \arrow[rd, "{\phi_{r+\epsilon, p}}"] &                       \\
  & {V'(r+\epsilon, p)} \arrow[ru, "{\psi_{r+\epsilon, p}}"] &                      & {V'(r,p)} \arrow[ru, "{\psi_{r,p}}"] \arrow[rr, "{V'((r,p)\preccurlyeq (r+2\epsilon, p))}"] &                                                          & {V'(r+2\epsilon, p)}
\end{tikzcd}

%% file: tikz/VR_diagrams_1.tex
\begin{tikzcd}
{\VR_r(\cup_{i\in p} X_i)} \arrow[rd, "{\phi_{r,p}}"] \arrow[rr, "\mathrm{inc}"] &                                                                                      & {\VR_{r'}(\cup_{i\in p'} X_i)} \arrow[rd, "{\phi_{r',p'}}"] &                      &                                                                                     & {\VR_{r+\epsilon}(\cup_{i\in p} X_i)} \arrow[rr, "\mathrm{inc}"] &                                         & {\VR_{r'+\epsilon}(\cup_{i\in p'} X_i)} \\
 & {\VR_{r+\epsilon}(\cup_{i\in p} Y_i)} \arrow[rr, "\mathrm{inc}"] &                                        & {\VR_{r'+\epsilon}(\cup_{i\in p'} Y_i)} & {\VR_r(\cup_{i\in p} Y_i)} \arrow[ru, "{\psi_{r,p}}"] \arrow[rr, "\mathrm{inc}"] &                                                                                   & {\VR_{r'}(\cup_{i\in p'}Y_i)} \arrow[ru, "{\psi_{r',p'}}"] &                     
\end{tikzcd}

%% file: tikz/VR_diagrams_2.tex
\begin{tikzcd}
{\VR_r(\cup_{i\in p} X_i)} \arrow[rr, "\mathrm{inc}"] \arrow[rd, "{\phi_{r,p}}"] &                                                           & {\VR_{r+2\epsilon}(\cup_{i\in p} X_i)} &                                                                                              & {\VR_{r+\epsilon}(\cup_{i\in p} X_i)} \arrow[rd, "{\phi_{r+\epsilon, p}}"] &                       \\
  & {\VR_{r+\epsilon}(\cup_{i\in p} Y_i)} \arrow[ru, "{\psi_{r+\epsilon, p}}"] &                      & {\VR_r(\cup_{i\in p} Y_i)} \arrow[ru, "{\psi_{r,p}}"] \arrow[rr, ""] &                                                          & {\VR_{r+2\epsilon}(\cup_{i\in p} Y_i)}
\end{tikzcd}